\documentclass[]{amsart}
\usepackage[margin=1in]{geometry}
\usepackage[utf8]{inputenc}
\usepackage{amsmath,amsthm,mathrsfs}
\usepackage{amssymb}
\usepackage{diffcoeff}
\usepackage{mathtools}
\usepackage{hyperref}
\usepackage{xcolor}

\DeclarePairedDelimiter\abs{\lvert}{\rvert}
\DeclarePairedDelimiter\autobracket{(}{)}
\newcommand{\pb}[1]{\autobracket*{#1}}
\newcommand{\Mod}[1]{ (\mathrm{mod}\ #1)}
\newcommand{\ov}{\overline}

\newcommand{\shortmod}{\ensuremath{\negthickspace \negthickspace \negthickspace \pmod}}
\newcommand{\legendre}[2]{\ensuremath{\left( \frac{#1}{#2} \right) }}

\theoremstyle{plain}		
	\newtheorem{mytheo}{Theorem} [section]
	
	\newtheorem{myprop}[mytheo]{Proposition}
	\newtheorem{mycoro}[mytheo]{Corollary}
     \newtheorem{mylemma}[mytheo]{Lemma}
	\newtheorem{mydefi}[mytheo]{Definition}
    \newtheorem{myexam}[mytheo]{Example}
	
	\newtheorem{myremark}[mytheo]{Remark}

\theoremstyle{remark}

\usepackage{amsfonts}

\numberwithin{equation}{section}
\numberwithin{figure}{section}

\usepackage[backend=biber, doi=false,isbn=false,url=false, articlein=false, maxnames=4, maxalphanames=4, sorting=nty, style=ext-alphabetic]{biblatex}
\DeclareFieldFormat[article,misc,inbook,incollection]{title}{#1}
\begin{document}
\author{Agniva Dasgupta}
\address{Mathematical Sciences Dept, University of Texas at Dallas, USA}
\email{agniva.dasgupta@utdallas.edu}
\thanks{The author was partially supported by the NSF grants DMS-2302210 (P.I.-Matthew P. Young) and DMS-2341239 (P.I.- Rizwanur Khan) while working on this paper.}
\title{Short Second Moment Bound for GL(2) $L$-functions in $q$-Aspect}
\date{}
\begin{abstract}
 We prove a Lindel\"{o}f-on-average upper bound for the second moment of the $L$-functions associated to a level 1 holomorphic cusp form, twisted along a coset of subgroup of the characters modulo $q^{2/3}$ (where $q = p^3$ for some odd prime $p$). This result should be seen as a $q$-aspect analogue of Anton Good's (1982) result on upper bounds of the second moment of cusp forms in short intervals. The results generalize easily to higher prime powers as well.
\end{abstract}
\maketitle

\section{Introduction}

\begin{subsection}{Statement of Results}
The study of moments of $L$-functions at the central point, $s=\frac12$, has been an important area of research in analytic number theory. While originally this interest grew due to connections with the Lindel\"{o}f Hypothesis, estimation of similar moments for a family of $L$-functions has now become an interesting point of study on its own.

One key result in this area is the following proposition, due to Iwaniec in 1978 (Theorem $3$ in \cite{iw4m}).

\begin{myprop} 
    \label{prop:IwaniecZeta}
    For $T \geq 2$ and $\varepsilon >0$, 
    \begin{equation}
        \label{eq:IwaniecZeta}
        \int_{T}^{T+T^{\frac23}} \abs{\zeta\pb{\tfrac12+it}}^4 dt \ll_{\varepsilon} T^{\frac23+\varepsilon}.
    \end{equation}
\end{myprop}
Proposition \ref{prop:IwaniecZeta} is one of the first examples of an upper bound on a `short' moment of $L$-functions. Most of the earlier results were concerned with estimating integrals of the type $\int_{0}^{T}  \abs{\zeta\pb{\frac12+it}}^k dt$, for some positive even integer $k$. 

Note that \eqref{eq:IwaniecZeta} is consistent with the Lindel\"{o}f Hypothesis, and is an example of a Lindel\"{o}f-on-average upper bound. For level $1$ holomorphic cusp forms, an analogous result follows from Good \cite{go2m},
\begin{myprop} 
    \label{prop:GoodCusp}
    Let $f$ be a fixed level $1$ holomorphic cusp form. For $T \geq 2$ and $\varepsilon >0$, 
    \begin{equation}
        \label{eq:GoodCusp}
        \int_{T}^{T+T^{\frac23}} \abs{L\pb{f,\tfrac12+it}}^2 dt \ll_{f,\varepsilon} T^{\frac23+\varepsilon}.
    \end{equation}
\end{myprop}

Equations \eqref{eq:IwaniecZeta} and \eqref{eq:GoodCusp} imply a Weyl-type subconvexity bound for the respective $L$ functions, $\zeta(\cdot)$, and $L(f,\cdot)$. In fact, for holomorphic cusp forms, \cite{go2m} was the first instance where such a result was obtained. 

In their paper on the Weyl bound for Dirichlet $L$-functions \cite{py4m}, Petrow and Young, proved a $q$-aspect analogue of Proposition \ref{prop:IwaniecZeta}. A special case of Theorem $1.4$ (with $q=p^3, d =p^2$) in \cite{py4m} can be stated as follows. 

\begin{myprop}
    \label{prop:PYshortmoment}
    Let $q = p^3$, for an odd prime $p$. Let $\alpha$ be a fixed primitive character mod $q$. For any $\varepsilon >0$, we have
    \begin{equation}
        \label{eq:PYshortmoment}
        \sum_{\psi \shortmod{q^{\frac23}}} \abs{L\pb{\alpha \cdot \psi, \tfrac12}}^4 \ll_{\varepsilon} q^{\frac23 + \varepsilon}.
    \end{equation}
\end{myprop}
We discuss this analogy in more detail in Section \ref{subsec:shortmominlit}. 

In this paper, following up on the ideas in \cite{py4m}, we derive a $q$-aspect analogue of Good's result in Proposition \ref{prop:GoodCusp}. We prove the following theorem.

\begin{mytheo}
    \label{thm:MainThm}
    Let $f$ be a level $1$ cusp form. Let $q = p^3$, for an odd prime $p$, and let $\alpha$ be a primitive character modulo $q$. For any $\varepsilon >0$, 
    \begin{equation} 
        \label{eq:MainThm}
        \sum_{\psi\Mod{q^\frac23}} \abs{ L\pb{ f \otimes \pb{\alpha\cdot \psi}, \tfrac{1}{2}}}^2 \ll_{f,\varepsilon} q^{\frac23+\varepsilon}.
    \end{equation}
\end{mytheo}

We note that, similar to Propositions \ref{prop:IwaniecZeta}, \ref{prop:GoodCusp}, \ref{prop:PYshortmoment}, this is also a Lindel\"{o}f-on-average bound.

Also, even though Theorem \ref{thm:MainThm} is stated for $q=p^3$, we expect this to generalise to a short moment exactly analogous to the Theorem $1.4$ \cite{py4m}. While the authors in \cite{py4m} needed the fully general result to prove the Weyl bound for all Dirichlet $L$-functions (see also \cite{pyweyl}), we do not have any such demand. So, for the sake of simplicity, we choose to work with $q=p^3$. We also outline a sketch for the higher prime power cases ($q=p^j, j \geq 4$) in Section \ref{sec:higherppower} .

Using a lower bound on the associated first moment, we can deduce from Theorem \ref{thm:MainThm} the following result about non-vanishing of $L$-functions within this family.
\begin{mytheo}
\label{thm:nonv}
Let $f,p,q,\alpha$ be as in Theorem \ref{thm:MainThm}. Then for $\varepsilon > 0$,
\begin{equation}
   \# \{ \psi \shortmod {q^{\frac23}} ; \ L\pb{ f \otimes \pb{\alpha\cdot \psi}, \tfrac{1}{2}}  \neq 0\} \gg_{\varepsilon} p^{2-\varepsilon}.  
\end{equation}
\end{mytheo}

 Moments of a family of $L$-functions encode a lot of information about the individual members of the family. For one such example, the strength of the second moment bound in $\eqref{eq:MainThm}$ is enough to immediately deduce a Weyl-type subconvexity bound for individual $L$-functions in this family. This result was originally proven by Munshi and Singh in 2019  (Theorem $1.1$ in \cite{ms19} with $r=1, t=0$). The authors use a completely different approach. They do not compute moments for an associated family; relying instead on a novel variant of the circle method, first introduced in \cite{mdelta}.
 
 We state their result as a corollary.

\begin{mycoro}
    \label{cor:Weylsubcon}
    Let $f,p,q,\alpha$ be as in Theorem \ref{thm:MainThm}. Let $\psi$ be a character modulo $q^{\frac23}$. Then for any $\varepsilon>0$,
    \begin{equation}
        L\pb{ f \otimes \pb{\alpha\cdot \psi},\tfrac{1}{2}} \ll_{f,\varepsilon} q^{\frac13+\varepsilon}. 
    \end{equation}
\end{mycoro}
\subsection{Notations}
 We use standard conventions present in analytic number theory. We use $\varepsilon$ to denote an arbitrarily small positive constant. For brevity of notation, we allow $\varepsilon$ to change depending on the context. 
\\The expression $F \ll G$ implies there exists some constant $k$ for which $\abs{F} \leq k\cdot G$ for all the relevant $F$ and $G$. We use $F \ll_{\varepsilon} G$ to emphasize that the implied constant $k$ depends on $\varepsilon$ (it may also depend on other parameters). For error terms, we often use the big $O$ notation, so $f(x) = O(g(x))$ implies that $f(x) \ll g(x)$ for sufficiently large $x$. We use the term `small' or `very small' to refer to error terms which are of the size $O_A(p^{-A})$, for any arbitrarily large $A$.
\\By a dyadic interval, we mean an interval of the type $[2^{\frac{k}{2}}M, 2^{\frac{k+1}{2}}M]$ for some $k \in \mathbb{Z}$, $M \in \mathbb{R}$. We also use $m \asymp M_0$ to mean $m$ ranges over the dyadic interval $[M_0,2M_0]$.\\
We also use $\sideset{}{^*}\sum_{n}$ to denote $\smashoperator{\sum_{\substack{n \\ \text{gcd}(n,p)=1}}}$ . Similarly, $ \sideset{}{^*}\sum_{a \Mod{c}}$ is used for $\smashoperator{\sum_{\substack{a \Mod{c} \\ \text{gcd}(a,c)=1}}}$ , and $\sideset{}{^*}\sum_{\chi(p)}$ for $\smashoperator{\sum_{\substack{ \chi \Mod{p} \\ \chi \text{ primitive}}}}$.
\\As usual, $e(x) = e^{2\pi i x}$. Also, $e_p(x) = e(\frac{x}{p}) = e^{2\pi i \frac{x}{p}}$.
\end{subsection}

\subsection{Sketch of Proof}
We give a brief sketch of Theorem $\ref{thm:MainThm}$ in this section. 
\\In this sketch, we use the symbol `$\approx$' between two expressions to mean the expressions on the left side can be written as an expression similar to the right side, along with an acceptable error term. 

Using an approximate functional equation and orthogonality of characters (see Section \ref{sec:ProofReduction}), it suffices to prove the following result on an associated shifted convolution problem.

\begin{mytheo}
    \label{thm:redthm}
    Let $f,p,q, \alpha$ be as in Theorem \ref{thm:MainThm}. Let $\lambda_f(\cdot)$ be the coefficients of the associated Dirichlet series. Let
    \begin{equation}
        \label{eq:redthm}
        S(N,\alpha) \coloneqq \sideset{}{^*}\sum_{l,n} \lambda_f(n+p^2l)\ov{\lambda_f}(n)\alpha(n+p^2l)\ov{\alpha(n)}w_N(n+p^2l) w_N(n), 
    \end{equation}
     where  $N  \ll_{\varepsilon} p^{3+\varepsilon}$, and $w_N(\cdot)$ is some smooth function supported on $[N,2N]$ satisfying $w_N^{(j)}(x) \ll N^{-j}$. We then have 
    \begin{equation}
        \label{eq:shiftedconvbound}
        S(N,\alpha) \ll_{f,\varepsilon} Np^{\varepsilon}.
    \end{equation}
   
\end{mytheo}

\begin{myremark}
    \label{rem:nlbound}
     As $w_N(\cdot)$ is supported in $[N,2N]$, we have that, in \eqref{eq:redthm}, $ 0< l \leq \frac{N}{p^2}$, and $n \asymp N$.
\end{myremark}
  The trivial bound on $S(N,\alpha)$ is $\frac{N^2}{p^2}p^{\varepsilon} \ll Np^{1+\varepsilon}$. To improve on this, we first note that $S(N,\alpha)$ displays a conductor dropping phenomenon, notably
\begin{equation}
\label{eq:introcond}
    \alpha(n+p^2l)\ov{\alpha(n)} = e_p(a_\alpha l \ov{n}).
\end{equation}
for some non-zero $a_\alpha \Mod{p}$.
 
Notice that, when $p \mid l$ in \eqref{eq:redthm}, $S(N,\alpha)$ does not have any cancellations. However, we are saved by the fact that the number of such terms is $O(Np^{\varepsilon})$. So for the rest of the sketch, we assume $(l,p)=1$.

In order to separate the variables $n$ and $(n+p^2l)$ we introduce a delta symbol (see Section \ref{subsec:Delta}). This introduces a new averaging variable $c$, and once again we split the resulting expression into whether gcd$(c,p) =1$, or not.
\\The former requires more work, and we focus on this term for the sketch.  We have 
\begin{multline}
    \label{eq:IntroS3}
    S(N,\alpha) \approx \sideset{}{^*}\sum_{l \leq \frac{N}{p^2}}\sideset{}{^*}\sum_{c \leq \sqrt{N}} \ \sideset{}{^*}\sum_{a \shortmod{c}} e\pb{\frac{-ap^2l}{c}} \int_{-\infty}^{\infty} g_c(v)e(-p^2lv)  \\ \cdot \sum_{m \asymp N} \lambda_f(m)e\pb{\frac{am}{c}}w_N(m)e(mv)   \sideset{}{^*}\sum_{n \asymp N} \ov{\lambda}_f(n)e_p(a_\alpha l \ov{n})e\pb{\frac{-an}{c}}w_N(n)e(-nv) \ dv.
\end{multline}
Here, $g_c(v)$ is a smooth function which is small when $v \gg \frac{1}{c\sqrt{N}}$.

We can now use the Voronoi summation formula (see Sec. \ref{subsec:voronoi}), for the $m$ and $n$ sums in \eqref{eq:IntroS3}. Note that, as there is an extra additive character $e_p(a_\alpha l \ov{n})$ in the $n$ sum, we need to use a modified Voronoi summation formula (see Prop \ref{prop:modifiedVoronoi}) here. While this modification does introduce some extra terms to the final expression, these terms are pretty easily dealt with (see Section \ref{subsec:Remtermsvoronoi}). We get that
\begin{equation}
    \label{eq:IntroE0}
    S(N,\alpha) \approx \sideset{}{^*}\sum_{\substack{l \leq \frac{N}{p^2} \\ m \ll p^{\varepsilon} \\ n \ll p^{2+\varepsilon}}} \frac{\lambda_f(m)\ov{\lambda_f}(n)}{p^2} \sideset{}{^*}\sum_{c \leq \sqrt{N}}\frac{1}{c^2} S(\ov{p}^2n-m,-p^2l;c)\text{Kl}_3(n\ov{c}^2a_\alpha l,1,1;p) I_N(c,l,m,n),
\end{equation}
with  $S(a,b;c) = \sum_{mn \equiv 1 \Mod{c}} e\pb{\frac{am + bn}{c}}$, and $\text{Kl}_3(x,y,z;c) = \sum_{pqr \equiv 1 \Mod{c}} e\pb{\frac{px + qy + rz}{c}}$ denoting the Kloosterman and hyper-Kloosterman sums respectively. Also, $I_N(\cdot)$ is an integral of special functions (see \eqref{eq:INclmn}) which trivially satisfies $I_N(c,l,m,n) \ll_{f,\varepsilon} \frac{1}{c}N^{\frac32+\varepsilon}$.

Here, we should indicate that while the original sums were of length $N$ each (and $N \ll p^{3+\varepsilon}$), the dual sums are significantly shorter - one is $m \ll p^{\varepsilon}$, and the other $n \ll p^{2 + \varepsilon}$. Even though this represents significant savings (by a factor of $\frac{N^2}{p^2}$), it is still not sufficient. If we use Weil's bound for the Kloosterman sum, and Deligne's bound for the hyper-Kloosterman sum, along with the trivial bound on $I_N(\cdot)$, we get \eqref{eq:IntroE0} is bounded by $Np^{\frac54 + \varepsilon}$. This still falls short of the desired bound by a factor of $p^{-\frac54}$.

In order to get more cancellations, we want to use a spectral decomposition for the $c$-sum. We do this by using the Bruggeman Kuznetsov formula (Prop. \ref{prop:Kuznetsov}). Note that if gcd$(r,p)=1$, the hyper-Kloosterman sum can be decomposed as 
\begin{equation*}
    \text{Kl}_3(r,1,1;p) = \frac{1}{\phi(p)}\sum_{\chi(p)} \tau(\chi)^3 \chi(r).
\end{equation*}
Using this we rewrite \eqref{eq:IntroE0}  as
\begin{equation*}
   \sideset{}{^*}\sum_{\substack{l \leq \frac{N}{p^2} \\ m \ll p^{\varepsilon} \\ n \ll p^{2+\varepsilon}}} \frac{\lambda_f(m)\ov{\lambda_f}(n)}{\phi(p)p^2} \sum_{\chi(p)} \tau(\chi)^3 \chi(n a_\alpha l) \sideset{}{^*}\sum_{c}\frac{1}{c^2}S(\ov{p}(n-p^2m),-pl;c) \ov{\chi}^2(c) I_N(c,l,m,n). 
\end{equation*}
We now use the Bruggeman-Kuznetsov formula for $\Gamma_0(p)$ with central character $\ov{\chi}^2$ at the cusps $\infty$ and $0$. We use two different versions of the integral transforms, depending on whether the integral $I_{N}(c,l,m,n)$ has oscillatory behavior, or not (see Section \ref{subsec:kuznetsov}). We focus on the non-oscillatory case here, the other one is similar. We have

\begin{multline}
\label{eq:sketchfinal}
    S(N,\alpha) \approx \frac{N}{p^2} \sideset{}{^*}\sum_{\chi(p)} \chi(a_\alpha) \frac{\tau(
    \chi)^3}{p^3} \sum_{t_j}  \sum_{\pi \in \mathcal{H}_{it_j}(p,\chi^2)}  L(\tfrac12, \ov{\pi}\otimes\chi) \sum_{\substack{m \ll p^{\varepsilon} \\ n \ll p^{2+\varepsilon}}} \frac{\ov{\lambda}_f(n)\ov{\lambda}_{\pi}(\abs{p^2m-n})\chi(n)}{\sqrt{\abs{p^2m-n}}} \\ +\text{(Holomorphic terms)} + \text{(Eisenstein terms)}.
\end{multline}
Here $\mathcal{H}_{it} (n,\psi)$ denotes the set of Hecke-Maass newforms of conductor $n$, central character $\psi$, and spectral parameter $it$. Analysing the associated integral transforms, it suffices to restrict \eqref{eq:sketchfinal} to when $t_j \ll p^{\varepsilon}$.

The contribution of the holomorphic and Eisenstein terms are similar to the Maass form term written down here. 
\\Using the Cauchy-Schwarz inequality, it then suffices to bound the sums 
\begin{equation*}
    \sum_{t_j \ll p^{\varepsilon}} \sideset{}{^*}\sum_{\chi(p)} \sum_{\pi \in \mathcal{H}_{it_j}(p,\chi^2)}   \abs{L(\tfrac12, \ov{\pi}\otimes\chi) }^2 \ , \ \  \sum_{t_j \ll p^{\varepsilon}} \sideset{}{^*}\sum_{\chi(p)} \sum_{\pi \in \mathcal{H}_{it_j}(p,\chi^2)}  \sum_{n \ll p^{2+\varepsilon}} \left \vert\frac{\ov{\lambda}_f(n)\ov{\lambda}_{\pi}(\abs{p^2m-n})\chi(n)}{\sqrt{\abs{p^2m-n}}} \right \vert ^2.
\end{equation*}

Using the fact that $\ov{\pi} \otimes \chi  \in \mathcal{H}_{it_j}(p^2,1) $, we can get the required bounds by applying the spectral large sieve inequality on each of the two sums. The details can be seen in Section \ref{sec:SpectralAnalysis}.

If we compare our proof and the proof of Prop \ref{prop:PYshortmoment}, we can see that the authors in \cite{py4m} do not use a delta symbol for the shifted convolution sum at the beginning, instead relying on  an approximate functional equation-type formula for the divisor function. A more significant difference can be observed later, by considering the shape of the terms in \eqref{eq:sketchfinal} and equation (1.20) in \cite{py4m}. Unlike our case, the authors get a product of three $L$-functions that they then bound using Holder's inequality. This reduces their problem to bounding the fourth moment of L -functions, which is accomplished via the use of the spectral large sieve inequality. 

\begin{myremark}
    We expect the proof steps outlined here to easily generalise to the case when $f$ is a Hecke-Maass cusp form. We do not use any result which is known for holomorphic cusp forms, but not for Maass forms. In particular, our proof does not need Deligne's bound on the Fourier coefficients of holomorphic cusp forms.
\end{myremark}

\subsection{Short Moments}
\label{subsec:shortmominlit}
Theorem \ref{thm:MainThm} joins a growing list of results on short moments of families of $L$-functions. The general idea here is to work with a subfamily of $L$-functions that exhibit a conductor lowering phenomenon. Say, we start with a family $\mathcal{F}$ of analytic $L$-functions; we choose a subfamily $\mathcal{F}_0$, such that the quantity $C(\pi_1 \otimes \ov{\pi_2})$, where $C(\pi)$ denotes the analytic conductor of $\pi$, is of a comparatively smaller size, when $\pi_1, \pi_2$ are in $\mathcal{F}_0$. (See Section 1.5 in \cite{py4m} for a more detailed discussion on this). 

We see an archimedean version of this in Propositions \ref{prop:IwaniecZeta} and \ref{prop:GoodCusp}. The full family of $L$-functions considered in Proposition \ref{prop:GoodCusp},  for example, is $\mathcal{F}=\{ L(f \otimes \abs{\cdot}^{it},\tfrac12 ), T\leq t \leq\ 2T\}.$ Here $f$ is a level $1$ cusp form. The analytic conductor of $L(f \otimes \abs{\cdot}^{it_1} \otimes \ov{f \otimes \abs{\cdot}^{it_2}})$ is proportional to $\abs{t_1 - t_2}^4$. Now, for arbitrary $t_1,t_2$ in $[T,2T]$, this can be as high as $T^4$. However, if we restrict to the subfamily, $\mathcal{F}_0=\{ L(f \otimes \abs{\cdot}^{it},\tfrac12 ), T\leq t \leq\ T+T^{\tfrac23}\}$ - as considered in \eqref{eq:GoodCusp}, we have that $\abs{t_1 - t_2}^4 \leq T^{\tfrac83}$, which is significantly smaller than $T^4$.
This leads to a conductor lowering analogous to \eqref{eq:introcond}.

For our work, we start with the family, $\mathcal{F}=\{ L(f \otimes \chi,\tfrac12), \chi \Mod{q}\}$, where $q=p^3$. Using some of the results in \cite{Li}, we know that the analytic conductor of $L\pb{(f \otimes \chi_1 \otimes \ov{(f \otimes \chi_2)},\tfrac12 }$ depends on 
the fourth power of the conductor of $(\chi_1\ov{\chi}_2)$. Now, for arbitrary $\chi_1$, and $\chi_2$, this can be as high as $q^4$. To get a short moment in the level aspect, we choose a subfamily of $\mathcal{F}$ that lowers this significantly. We consider the subfamily $\mathcal{F}_0=\{ L(f \otimes (\alpha\cdot \psi), \tfrac12 ), \psi \Mod{q^{\tfrac23}}\},$ for a fixed primitive character $\alpha \Mod{q}$.Now, the analytic conductor of $L\pb{(f \otimes (\alpha\cdot\psi_1)) \otimes \ov{(f \otimes (\alpha\cdot\psi_2))},\tfrac12 }$ (more accurately, of $L\pb{(f\otimes \ov{f} \otimes (\psi_1\ov{\psi}_2),\tfrac12 }$) depends on $\pb{\text{conductor}(\psi_1 \ov{\psi_2})}^4 \leq q^{\frac83}$, which is significantly lower than $q^4.$

Short moments in the level aspect, in a $GL(1)$ setting, have also been considered in \cite{Nun}, and in \cite{MW}. Another similar application was considered along Galois orbits in \cite{KMN}.

\subsection{Acknowledgements} I am deeply grateful to my doctoral advisor, Matthew P. Young, for suggesting this problem, and for engaging in extensive discussions on this subject throughout the process of working on this paper. I would also like to thank Peter Humphries and Chung-Hang Kwan for some insightful discussions regarding this work. I am also grateful to the reviewer for their valuable comments and suggestions.

\section{Preliminary Results from Analytic Number Theory}
\label{sec:ANTresults}
In this section, we note down some of the analytic number theory preliminaries we will use later. We omit proofs in most cases. Interested readers can check the references listed next to the results.
\subsection{Approximate Functional Equation}
\label{subsec:spproxfe}
We state a version of the approximate functional equation for analytic $L$-functions. 
For proofs, see Theorem $5.3$ and Proposition $5.4$ in \cite{ikant}.
\begin{myprop}
Let $X>0$, and $L(f,s)$ be an $L$-function given by $L(f,s) = \sum_{n\geq1} \frac{\lambda_f(n)}{n^s} $ when $\Re{(s)}>1$. If the completed $L$-function $\Lambda(f,s)$ is entire, then we have the following approximation for $L\left(f,s \right ) $, when $ 0 \leq \Re{(s)} \leq 1 $ - 
\label{prop:approxfe}
    \begin{equation} 
    L\left(f,s \right ) = \sum_{n} \frac{\lambda_f(n)}{n^s}V\left ( \frac{n}{X\sqrt{q}} \right ) + \varepsilon\left(f,s\right)\sum_{n}\frac{\ov{\lambda_f}(n)} {n^{1-s}}V\left ( \frac{nX}{\sqrt{q}} \right ).
    \end{equation}

Here, $q$ is the conductor of $L(f,s)$, and $V(y)$ is a smooth function that satisfies the following bounds -
\begin{equation}
    V(y) \ll \left ( 1 +     \frac{y}{\sqrt{\mathfrak{q}_\infty}} \right )^{-A},
\end{equation}
where $q\cdot\mathfrak{q}_\infty$ is the analytic conductor of $L(f,s)$ and $A>0$.

\end{myprop}

As $\abs{\varepsilon\left(f,\tfrac{1}{2}\right)} = 1$, we have the following immediate corollary (using $s = \frac12, X=1$), 
\begin{mycoro}
\label{cor:approx.func}
Let $L(f,s), q, \lambda$ be as above.  For any fixed $\varepsilon>0, A>0$,
    \begin{equation}\left \lvert L\left(f,\tfrac{1}{2} \right ) \right \rvert^2 \leq 4\left \lvert \sum_{n \leq q^{1+\varepsilon}} \frac{\lambda_f(n)}{\sqrt{n}}V\left ( \frac{n}{\sqrt{q}} \right ) \right \rvert^2 + O(q^{-A}).
    \end{equation}

\end{mycoro}
\subsection{Postnikov Formula}
\label{subsec:postnikov}
We will need to use a particular case of the Postnikov formula stated below.

\begin{myprop}
\label{prop:Postnikov}
    Let $p$ be an odd prime, $\alpha$ be a primitive character modulo $p^3$ and $l \in \mathbb{Z}_{>0}$. Then for all $n$ with  gcd$(n,p)= 1, \ \alpha(n+p^2l)\ov{\alpha(n)} = e_p(a_\alpha l \ov{n}),$ for some $a_\alpha \in \pb{\mathbb{Z}/{p\mathbb{Z}}}^{\times},$ independent of $n$.
\end{myprop}

\begin{proof}
Define a function $\psi$ on $\mathbb{Z}$ as $\psi(n) = \alpha(1+p^2n)$. We have
\begin{equation*}
\psi(m+n)= \alpha (1+p^2(m+n))= \alpha (1+p^2(m+n)+p^4mn)= \alpha(1+p^2n)\cdot\alpha(1+p^2m) = \psi(m)\cdot\psi(n).
\end{equation*}
Hence, $\psi$ is an additive character modulo $p$, and $\psi(n) = e_p(a_\alpha n)$, for some  $a_\alpha \in \pb{\mathbb{Z}/{p\mathbb{Z}}}^{\times}$ (as $\alpha$ is primitive, $a_\alpha$ cannot be $0$).
\\Now, when gcd$(n,p)=1$, we have $\alpha(n+p^2l)\ov{\alpha(n)} = \alpha(1 + p^2l\ov{n}) = \psi(l\ov{n}) = e_p(a_\alpha l \ov{n})$.
\end{proof}
\subsection{Delta Symbol}
\label{subsec:Delta}

This section is a brief review of the delta-symbol method from Section $20.5$ in \cite{ikant}. 

Let $w(u)$ be a smooth, compactly supported function on $[C,2C]$ for some $C>0$. Additionally suppose that $\sum_{q=1}^\infty w(q)= 1$.  Then we can rewrite the Kronecker delta function at zero as
\begin{equation}
\label{eq:delta}
    \delta(n) = \sum_{q \mid n } \pb{w(q) - w\pb{\frac{\lvert n \rvert}{q}}}.
\end{equation} 

\begin{myprop}
\label{prop:delta}
Let $\delta$ be as in \eqref{eq:delta}. Using orthogonality of the additive characters modulo $q$, we can rewrite \eqref{eq:delta} as
\begin{equation}
    \label{eq:deltamodified}
    \delta(n)= \sum_{c=1}^{\infty} S(0,n;c) \Delta_c(n).
\end{equation}

\end{myprop}
Here, $$S(0,n;c) = \sideset{}{^*}\sum_{d \shortmod{c}} e\pb{\frac{dn}{c}}, \ \text {and } \Delta_c(u) = \sum_{r=1}^{\infty} \frac{1}{cr}\pb{w(cr) - w\pb{\frac{\abs{u}}{cr}}}. $$  
Since $\Delta_c(u)$ is not compactly supported, we multiply it by a smooth function $f$  supported on $[-2N,2N]$  ($N >0$) such that $f(0)$ = 1. We also assume, for any $a\in \mathbb{N}$, $f^{(a)}(u) \ll N^{-a}$.

Thus we have 
\begin{equation}
\label{eq:deltanobound}
    \delta(n) = \sum_{c=1}^{\infty} S(0,n;c)\Delta_c(n)f(n).
\end{equation}

As $w$ is supported on $[C,2C]$ and $f$ is supported on $[N,2N]$, the sum on the right hand side is only upto $c \leq 2 \text{ max } \pb{C, \frac{N}{C}} = X$. The optimal choice would thus be to take $C= \ \sqrt{N}$, giving $X = 2C = 2 \sqrt{N}$.

Thus \eqref{eq:deltanobound} becomes

\begin{equation}
\label{eq:deltabound}
    \delta(n) = \sum_{c \leq 2C} S(0,n;c) \Delta_c(n)f(n).
\end{equation}

We can get additional bounds on the derivatives of $\Delta_c(u)$ if we assume more conditions on $w$. 

\begin{myprop}
Suppose $w(u)$ is smooth, compactly supported in the segment $C \leq u \leq 2C$ with $C \geq 1$. Additionally, suppose that $\sum_{q=1}^{\infty} w(q) =1$, and $w(u)$ has derivatives satisfying 
\begin{equation}
    w^{(a)} (u) \ll C^{-a-1},
\end{equation}
for arbitrarily large $a \in \mathbb{N}$. Then for any $c \geq 1$ and $u \in \mathbb{R}$  we have 
\begin{equation}
\label{eq:Deltacbound}
    \Delta_c(u) \ll \frac{1}{(c+C)C} + \frac{1}{\pb{\lvert u \rvert +cC}},
\end{equation}

\begin{equation}
\label{eq:Deltacderbound}
    \Delta_c^{(a)}(u) \ll (cC)^{-1} (\abs{u} + cC)^{-a}.
\end{equation}
\end{myprop}

Finally, it is often useful to express $\Delta_c(n)f(n)$ in terms of additive characters, by considering its Fourier transform.
\\Let 
\begin{equation}
\label{eq:gc}
g_c(v) = \int_{-\infty}^{\infty} \Delta_c(u) f(u)e(-uv)du,    
\end{equation}
be the Fourier transform of $\Delta_c(u)f(u)$. By the Fourier inversion formula we have
\begin{equation}
    \Delta_c(u)f(u) = \int_{-\infty}^{\infty} g_c(v)e(uv)dv.
\end{equation}

Using this in  \eqref{eq:deltabound} , we have
\begin{equation}
\label{eq:deltafourier}
    \delta(u) = \sum_{c \leq 2C} S(0,n;c)\int_{-\infty}^{\infty} g_c(v)e(nv)dv.
\end{equation}

We have the following bounds on $g_c(v)$ and its derivatives.
\begin{myprop}
\label{prop:gcbound}
Suppose $g_c(v)$ is defined as in \eqref{eq:gc}, and $C= \sqrt{N}$. Then, for $a>1$,
\begin{equation}
 \label{eq:gcbound1}
    g_c(v) = 1 + O\pb{\frac{C}{c}\pb{\frac{c}{C}+ \abs{v}cC}^a},
\end{equation}
\begin{equation}
  \label{eq:gcbound2}
    g_c(v) \ll (\abs{v}cC)^{-a},
\end{equation}
and
\begin{equation}
 \label{eq:gcboundder}
    \frac{\partial^{j}}{\partial v^j}g_c(v) = \abs{vcC}^{-j} \min(\abs{vcC}^{-1}, C/c) \log{C}, \ \ \ \text{for }j \geq 1.
\end{equation}
\end{myprop}

Proposition \ref{prop:gcbound} follows from \cite[eq. (20.158), (20.159)]{ikant}, and the bounds in \eqref{eq:Deltacbound} and \eqref{eq:Deltacderbound}. See \cite[Lemma 15]{Huang} for a complete proof. As pointed out by Huang, there is a typo in \cite[eq. (20.158)]{ikant}, which has been corrected here. We thank the reviewer for bringing this to our attention.

\begin{myremark}
    Proposition \ref{prop:gcbound} allows us to essentially restrict the integral in \eqref{eq:gc} to the range $[-\frac{p^\varepsilon}{cC}, \frac{p^\varepsilon}{cC}]$. In the language of inert functions (see the next section for definitions), $g_c(v)$ satisfies the same derivative bounds as an $p^{\varepsilon}$-inert family of smooth functions. See also the discussion immediately following \cite[Lemma 15]{Huang}.
\end{myremark}

\subsection{Inert Functions and Stationary Phase}
\label{subsec:inertfunctions}
We mention some properties of inert functions in this section. Inert functions are special families of smooth functions characterised by certain derivative bounds. See  Sections $2$ and $3$ in \cite{kpyosc}, and Section $4$ in \cite{ky5m} for proofs.

\begin{mydefi}
Let $\mathcal{F}$ be an index set. A family $\{w_T\}_{T \in \mathcal{F}}$ of smooth function supported on a product of dyadic intervals in $\mathbb{R}_{>0}^{d}$ is called $X$-inert if for each $j = (j_1,j_2,\cdots,j_d) \in \mathbb{Z}_{>0}^d$, we have
\begin{equation*}
    C(j_1,j_2,\cdots,j_d) \coloneqq \sup_{T \in \mathcal{F}} \ \sup_{(x_1,x_2,\cdots,x_d) \in \mathbb{R}_{>0}^{d}}  X^{-j_1-\cdots -j_d} \left\vert x_1^{j_1}\cdots x_d^{j_d} {w_T}^{(j_1,\cdots,j_d)}(x_1,\cdots,x_d)\right\vert < \infty.
\end{equation*}
\end{mydefi}

We will often denote the sequence of constants $C(j_1,j_2,\dots,j_d)$ associated with this inert function as $C_{\mathcal{F}}$.

We note that the requirements for the functions to be supported on dyadic intervals can be easily achieved by applying a aartition of unity.

We give one simple example to highlight how such families can be constructed.

\begin{myexam}
    \label{ex:InertFamily}
    Let $w(x_1,\cdots,x_d)$ be a fixed smooth function that is supported on $[1,2]^d$, and define,
    \begin{equation}
        w_{X_1,\cdots,X_d} (x_1,\cdots,x_d) = w\pb{\frac{x_1}{X_1},\cdots,\frac{x_d}{X_d}}.
    \end{equation}
    Then with $\mathcal{F} = \{ T = (X_1,\cdots,X_d) \in \mathbb{R}_{>0}^d\}$, the family $\{ w_T\}_{T \in \mathcal{F}}$ is $1$-inert.
\end{myexam}

The following propositions can be checked immediately:

\begin{myprop}
    \label{prop:inertxpowerinside}
    Let $a,b \in \mathbb{R}$ with $b>0$. If $w_T(x)$ is a family of $X$-inert functions supported on $[N,2N]$, then the family $W_T(x)$ given by, $W_T(x) =  w_T(bx^a)$ is also $X$-inert, with support $\left [\pb{\frac{N}{b}}^{\frac{1}{a}}, 2\pb{\frac{N}{b}}^{\frac{1}{a}} \right ]$.

\end{myprop}

\begin{myprop}
    \label{prop:inertxpower}
    Let $a \in \mathbb{R}$. If $w_T(x)$ is a family of $X$-inert functions supported on $x \asymp N$, then so is the family $W_T(x)$ given by $W_T(x) = (\frac{x}{N})^{-a} w_T(x)$.

\end{myprop}

\begin{myprop}
If $w$ is an $X$-inert function and $v$ is a $Y$-inert function, then their product $w\cdot v$ is a $\max{(X,Y)}$-inert function.
\end{myprop}
Suppose that $w_T(x_1,\cdots,x_d)$ is $X$-inert and is supported on $x_i \asymp X_i$. Let
\begin{equation}
    \label{eq:InertFourierdefn}
    \widehat{w}_T(t_1,x_2,\cdots,x_d) = \int_{-\infty}^{\infty} w_T(x_1,\cdots,x_d) e(-x_1t_1)dx_1,
\end{equation}
denote its Fourier transform in the $x_1$-variable.

We state the following result (Prop 2.6 in \cite{kpyosc}) regarding the Fourier transform of inert functions.
\begin{myprop}
\label{prop:FourierInert}
Suppose that $\{ w_T : T \in \mathcal{F}\}$ is a family of $X$-inert functions such that $x_1$ is supported on $x_1 \asymp X_1$, and $\{w_{\pm Y_1}: Y_1 \in (0,\infty)\}$ is a $1$-inert family of functions with support on $\pm t_1 \asymp Y_1$. Then the family $\{ X_1^{-1} w_{\pm Y_1}(t_1)  \widehat{w}_T(t_1,x_2,\cdots,x_d) \ : \ (T, \pm Y_1) \in \mathcal{F} \times \pm(0,\infty)\}$ is $X$-inert. Furthermore if $Y_1 \gg p^\varepsilon \frac{X}{X_1}$, then for any $A>0$, we have
\begin{equation}
    \label{eq:inertfourier}
     X_1^{-1} w_{\pm Y_1}(t_1)  \widehat{w}_T(t_1,x_2,\cdots,x_d) \ll_{\varepsilon, A} p^{-A}.
\end{equation}
\end{myprop}

We can similarly describe the Mellin transform of such functions as well. We state the following result (Lemma 4.2 in \cite{ky5m}).

\begin{myprop}
\label{prop:MellinInert}
Suppose that $\{ w_T(x_1,x_2,\cdots,x_d) : T \in \mathcal{F}\}$ is a family of $X$-inert functions such that $x_1$ is supported on $x_1 \asymp X_1$. Let
\begin{equation}
    \label{eq:inertmellin}
     \widetilde{w}_T(s,x_2,\cdots,x_d) = \int_{0}^{\infty} w_T(x,x_2,\cdots,x_d) x^s \frac{dx}{x}.
\end{equation}
Then we have $ \widetilde{w}_T(s,x_2,\cdots,x_d) = {X_1}^s W_T(s,x_2,\cdots,x_d)$ where $W_T(s,\cdots)$ is a family of $X$-inert functions in $x_2,\cdots,x_d$, which is entire in $s$, and has rapid decay for $\abs{\Im(s)} \gg {X_1}^{1+\varepsilon}$.
\end{myprop}

The following is a restatement of Lemma 3.1 in \cite{kpyosc}.
\begin{myprop}
\label{prop:stationeryphaseinert}
Suppose that $w$ is an $X$-inert function, with compact support on $[Z,2Z]$, so that $w^{(j)}(t) \ll (Z/X)^{-j}$. And suppose that $\phi$ is smooth and satisfies $\phi^{(j)}(t) \ll \frac{Y}{Z^j}$ for some $Y/X^2 \geq R \geq 1$, and for all $t \in [Z,2Z]$. Let
\begin{equation}
    I = \int_{-\infty}^{\infty} w(t)e^{i\phi(t)} dt.
\end{equation}
We then have
\begin{enumerate}
    \item[(a)] If $\abs{\phi'(t)} \geq \frac{Y}{Z}$ for all $t \in [Z,2Z]$, then $I \ll ZR^{-A}$ for $A$ arbitrary large.
    \item[(b)] If $\phi''(t) \gg \frac{Y}{Z^2}$ for all $t \in [Z,2Z]$, and there exists a (necessarily unique) $t_0 \in \mathbb{R}$ such that $\phi'(t_0) = 0$, then 
\begin{equation}
\label{eq:statphaselemmabound}
    I = \frac{e^{i\phi(t_0)}}{\sqrt{\phi''(t_0)}} F(t_0) + O_A(Z R^{-A}),
\end{equation}
where $F$ is an $X$-inert function supported on $t_0 \asymp Z$.
\end{enumerate}
\end{myprop}

The previous result has a natural generalisation (Main Theorem in \cite{kpyosc}). 

\begin{myprop}

Suppose that $w$ is an $X$-inert function in $t_1,\cdots, t_d $, supported on $t_1 \asymp Z$ and $t_i \asymp X_i$ for $i=2,3,\dots,d$. Suppose that the smooth function $\phi$ satisfies
\begin{equation}
  \frac{\partial^{a_1+a_2+\cdots+a_d}}{\partial t_1^{a_1} \partial t_2^{a_2}\dots\partial t_d^{a_d}} \phi(t_1,t_2,\dots,t_d) \ll_C \frac{Y}{Z^{a_1}}\frac{1}{X_2^{a_2}\dots X_d^{a_d}},
\end{equation}
 for all $a_1,a_2,\dots,a_d \in \mathbb{N}$.
\\Now, suppose that $\phi'(t_1,t_2,\dots,t_d) \gg \frac{Y}{Z^2}$ (here $\phi''$ and $\phi'$ denote the derivatives of $\phi$ with respect to $t_1$), for all $t_1,t_2,\dots,t_d$ in the support of $w$, and there exists a (necessarily unique) $t_0 \in \mathbb{R}$ such that $\phi'(t_0)=0$. Suppose also that $\frac{Y}{X^2} \geq R \geq 1$, then
\begin{equation}
    I = \int_{\mathbb{R}} e^{i\phi(t_1,\dots,t_d) w(t_1,\dots,t_d)}dt_1 = \frac{Z}{\sqrt{Y}}e^{i\phi(t_0,t_2,\dots,t_d)}W(t_2,\dots,t_d) + O_A(ZR^{-A}),
\end{equation}
for some $X$-inert function $W$, and $A$ can be taken to be arbitrarily large. The implied constant depends on $A$ and $C$.

\end{myprop}

\subsection{Voronoi Summation Formula}
\label{subsec:voronoi}
We will need the following two Voronoi-type summation formulas.
\begin{myprop}
\label{prop:Voronoi}
    Let $f$ be a cusp form of weight $k \geq 1$, level $1$ and trivial central character. Suppose $c \geq 1$, and  $ad \equiv 1 \Mod{c}$. Then for any smooth, compactly supported function $g$ on $\mathbb{R}^+$, we have
\begin{equation}
\label{eq:Voronoi}
    \sum_{m=1}^\infty \lambda_f(m)e\pb{\frac{am}{c}}g(m) = \frac{1}{c}\sum_{n=1}^\infty \lambda_f(n)e\pb{\frac{-dn}{c}}H(n),
\end{equation}
where
\begin{equation}
\label{eq:VoronoiH}
H(n) = 2\pi i^k \int_0^\infty g(x)J_{k-1}\pb{\frac{4\pi}{c}\sqrt{xn}}dx.
\end{equation}

\end{myprop}

We will also use a modified Voronoi summation formula, which we state and prove below.

\begin{myprop}
\label{prop:modifiedVoronoi}
    Let $f, g, a, c,$ and $d$ be as in Proposition \ref{prop:Voronoi}. Additionally, assume that $f$ is a Hecke eigenform. Let $p$ be a prime with gcd$(c,p)=1$ and fix a constant $r\not\equiv 0 \Mod{p}$. We then have

\begin{multline}
\label{eq:Voronoifinal}
    \sum_{\text{gcd}(m,p)=1}\ov{\lambda_f}(m) e\pb{\frac{r\ov{m}}{p}}e\pb{-\frac{am}{c}}g(m) = \\ \sum_{\text{gcd}(n,p) = 1} \frac{\ov{\lambda_f}(n)H_1(n)}{p^2c}e\pb{\frac{\ov{ap^2}n}{c}}\text{Kl}_3(n\ov{c}^2r,1,1;p)+ \sum_{\text{gcd}(n,p^2) = p} \frac{\ov{\lambda_f}(n)H_1(n)}{p^2c}e\pb{\frac{\ov{ap^2}n}{c}} \\+\sum_{n=1}^{\infty} \frac{\ov{\lambda_f}(np) \ov{\lambda_f}(p)H_1(p^2n)}{p^2c}e\pb{\frac{\ov{a}n}{c}}  -\pb{1+\frac{1}{p}}\sum_{n=1}^{\infty}\frac{\ov{\lambda_f}(n)H_1(p^2n)} {pc}e\pb{\frac{\ov{a}n}{c}}.
\end{multline}
\end{myprop}
 Here, 
    \begin{equation}
    \label{eq:Kl3}
     \text{Kl}_3(x,y,z;p) = \sum_{abc\equiv 1\Mod{p}} e\pb{\frac{ax+by+cz}{p}},   
    \end{equation}
    and 
    \begin{equation}
    \label{eq:H1}
     H_1(y) = 2\pi i^k \int_o^\infty J_{k-1}\pb{\frac{4\pi\sqrt{xy}}{pc}}g(x)dx.   
    \end{equation}

\begin{proof}[Proof]
Let 
\begin{equation}
\label{eq:s}
    S = \sum_{\text{gcd}(m,p)=1}\ov{\lambda_f}(m) e\pb{\frac{r\ov{m}}{p}}e\pb{-\frac{am}{c}}g(m). 
\end{equation}
From the orthogonality of characters modulo $p$, we know that when gcd$(m,p)=1$,
\begin{equation}
\label{eq:orthogonality}
    e\pb{\frac{r\ov{m}}{p}} = \frac{1}{p}\sum_{t\shortmod{p}} \  \sideset{}{^*}\sum_{h \shortmod{p}} e\pb{\frac{r\ov{h}-ht+mt}{p}}. 
\end{equation}
We note that if gcd$(m,p) \neq 1$, then the right hand side of \eqref{eq:orthogonality} is zero. We thus have

\begin{equation}
\label{eq:simp.Voronoi}
    S = \frac{1}{p}\sum_{t\shortmod{p}} \ \sideset{}{^*}\sum_{h \shortmod{p}} e\pb{\frac{r\ov{h}-ht}{p}}\sum_{m=1}^{\infty}\ov{\lambda_f}(m)g(m)e\pb{\frac{-am}{c}+\frac{tm}{p}}.
\end{equation}

Separating the $t\equiv 0 \Mod{p}$ term, \eqref{eq:simp.Voronoi} can be rewritten as 
\begin{equation}
\label{eq:modifiedsimp.Voronoi}
    S= S_1 + S_2,
\end{equation}
where
\begin{equation}
\label{eq:s1}
    S_1 = \frac{1}{p} \ \sideset{}{^*}\sum_{h,t\Mod{p}} e\pb{\frac{r\ov{h}-ht}{p}}\sum_{m=1}^{\infty}g(m)\ov{\lambda_f}(m)e\pb{\frac{m(-ap+tc)}{pc}},  
\end{equation}
and
\begin{equation}
\label{eq:s2}
    S_2 = \frac{1}{p} \ \sideset{}{^*}\sum_{h\Mod{p}} e\pb{\frac{r\ov{h}}{p}}\sum_{m=1}^{\infty}g(m)\ov{\lambda_f}(m)e\pb{\frac{-am}{c}}.
\end{equation}
We can now use Proposition \ref{prop:Voronoi} to simplify $S_1$ and $S_2$. Note that gcd$(a,c)=1$ and gcd$(-ap+tc,pc)=1$ if gcd$(t,p)=1$. Thus \eqref{eq:s1} becomes
\begin{equation}
\label{eq:s1Voronoi}
    S_1 = \frac{1}{p}\ \sideset{}{^*}\sum_{h,t\Mod{p}}e\pb{\frac{r\ov{h}-ht}{p}} \pb{\frac{1}{pc}\sum_{n=1}^{\infty} \ov{\lambda_f}(n) e\pb{-\frac{\ov{tc-ap}}{pc}\cdot n}H_1(n)},
\end{equation}
where $H_1(y)$ is defined in \eqref{eq:H1}.

We now use the Chinese Remainder Theorem to simplify  \eqref {eq:s1Voronoi} further.
\begin{multline}
\label{eq:s1expanded}
    S_1 = \frac{1}{p^2c}\sum_{n=1}^{\infty}\ov{\lambda_f}(n) H_1(n)\pb{\ \sideset{}{^*}\sum_{h,t\Mod{p}}e\pb{\frac{r\ov{h}-ht}{p}}e\pb{\frac{-\ov{p}(\ov{tc-ap})n}{c}} e\pb{\frac{-\ov{c}(\ov{tc-ap})n}{p}}}\\ =\frac{1}{p^2c}\sum_{n=1}^{\infty}\ov{\lambda_f}(n) H_1(n)e\pb{\frac{\ov{ap^2}n}{c}}\pb{\ \sideset{}{^*}\sum_{h,t\Mod{p}}e\pb{\frac{r\ov{h}-ht-\ov{tc^2}n}{p}} } \\ = \sum_{n=1}^{\infty} \frac{\ov{\lambda_f}(n)H_1(n)}{p^2c}e\pb{\frac{\ov{ap^2}n}{c}}\text{Kl}_3(n\ov{c}^2r,1,1;p),
\end{multline}
where $\text{Kl}_3(\cdot)$ is defined in \eqref{eq:Kl3}.
\\Note that if $p \mid n$, then this hyper-Kloosterman sum can be simplified as
\begin{equation}
    \text{Kl}_3(n\ov{c}^2r,1,1;p) = \sum_{xyz \equiv 1 \Mod{p}}e\pb{\frac{x(n\ov{c}^2r) + y + z}{p}} = \ \sideset{}{^*}\sum_{y,z \Mod{p}} e\pb{\frac{y+z}{p}} = 1.
\end{equation}
This allows us to rewrite \eqref{eq:s1expanded} as,
\begin{equation}
\label{eq:s1sum}
    S_1 = S_1^{*} + S_1^p + S_1^{p^2}.
\end{equation}
Here,
\begin{equation}
\label{eq:s1*}
    S_1^{*} = \sum_{\text{gcd}(n,p) = 1} \frac{\ov{\lambda_f}(n)H_1(n)}{p^2c}e\pb{\frac{\ov{ap^2}n}{c}}\text{Kl}_3(n\ov{c}^2r,1,1;p),  
\end{equation}

\begin{equation}
\label{eq:s1p}
    S_1^p = \sum_{\text{gcd}(n,p^2) = p} \frac{\ov{\lambda_f}(n)H_1(n)}{p^2c}e\pb{\frac{\ov{ap^2}n}{c}}\text{Kl}_3(n\ov{c}^2r,1,1;p) =\sum_{\text{gcd}(n,p^2) = p} \frac{\ov{\lambda_f}(n)H_1(n)}{p^2c}e\pb{\frac{\ov{ap^2}n}{c}}, 
\end{equation}
and
\begin{equation}
\label{eq:s1p^2}
    S_1^{p^2} =\sum_{\text{gcd}(n,p^2) = p^2}  \frac{\ov{\lambda_f}(n)H_1(n)}{p^2c}e\pb{\frac{\ov{ap^2}n}{c}}\text{Kl}_3(n\ov{c}^2r,1,1;p)=  \sum_{\text{gcd}(n,p^2) = p^2} \frac{\ov{\lambda_f}(n)H_1(n)}{p^2c}e\pb{\frac{\ov{ap^2}n}{c}}.
\end{equation}
Using Proposition \ref{prop:Voronoi} again on \eqref{eq:s2} we get
\begin{equation}
\label{eq:s2Voronoi}
    S_2 = \frac{1}{p}\ \sideset{}{^*}\sum_{h\Mod{p}} e\pb{\frac{r\ov{h}}{p}}\frac{1}{c}\sum_{n=1}^{\infty}\ov{\lambda_f}(n)e\pb{\frac{\ov{a}n}{c}}H(n) = \sum_{n=1}^{\infty}\frac{\ov{\lambda_f}(n)H(n)}{pc}e\pb{\frac{\ov{a}n}{c}}S(r,0;p),
\end{equation}
where $S(r,0;p)$ is the Ramanujan sum modulo $p. \ H(y)$ is the same as in \eqref{eq:VoronoiH}. 
\\Note that $S(r,0;p) = -1$ as $r \not\equiv 0 \Mod{p}$. Also, $H(y) = H_1(p^2y)$.
\\Finally, as $f$ is a Hecke eigenform, $\lambda_f(p^2n) = \lambda_f(p)\lambda_f(np) - \lambda_f(n)$. Now, \eqref{eq:modifiedsimp.Voronoi}, \eqref{eq:s1sum} and \eqref{eq:s2Voronoi}, together imply \eqref{eq:Voronoifinal}.
\end{proof}

\subsection{Bruggeman-Kuznetsov Formula}
\label{subsec:kuznetsov}
We state the version of Bruggeman-Kuznetsov formula that we will use. This is a restatement of Theorem $6.10$ in \cite{py4m}.

Let $\chi$ be an even Dirichlet character modulo $q$ and $\Phi \in C_C^{\infty}(\mathbb{R}_{>0})$,

\begin{equation}
\label{eq:Volofq}
    V(q) = \text{Vol }(\Gamma_0(q)\backslash \mathcal{H}) = \frac{\pi}{3}q\prod_{p \mid q} \pb{1+\frac{1}{p}},
\end{equation}
and
\begin{equation}
\label{eq:KuznetsovLHS}
    \mathcal{K} = \sum_{(c,q)=1}\ov{\chi}(c)S(\ov{q}m,n,c)\Phi(q^{\frac12}c),
\end{equation}

where $S(a,b;c)$ denotes the usual Kloosterman sum.
 
Additionally, for fixed integers $m$ and $n$, we define the following integral transforms - 
\begin{equation}
\label{eq:LholPhialt}
 \mathcal{L}^{\text{hol}} \Phi(k) = \int_0^{\infty} J_{k-1} \pb{\frac{4\pi\sqrt{\abs{mn}}}{x}} \Phi(x) dx .
\end{equation}
Also,

\begin{equation}
\label{eq:L+/-Phialt}
    \mathcal{L}^{\pm} \Phi(t) = \int_0^\infty B_{2it}^{\pm}\pb{\frac{4\pi\sqrt{\abs{mn}}}{x}}  \Phi(x) dx,
\end{equation}

where

\begin{equation}
    \label{eq:Bplus}
    B_{2it}^+(x) = \frac{i}{2\sin{(\pi t)}} \pb{J_{2it}(x) - J_{-2it}(x)},
\end{equation}

\begin{equation}
    \label{eq:Bminus}
    B_{2it}^-(x) = \frac{2}{\pi} \cosh{(\pi t)} K_{2it}(x).
\end{equation}

Here, $J$ and $K$ denote the usual $J$-Bessel and $K$-Bessel functions.
\\We could also use Mellin inversion formula to rewrite these integral transforms as-

\begin{equation}
\label{eq:LholPhi}
 \mathcal{L}^{\text{hol}} \Phi(k) = \frac{1}{2\pi i}\int_{(1)} \frac{2^{s-1}\Gamma\pb{\frac{s+k-1}{2}}}{\Gamma\pb{\frac{k+1-s}{2}}}\tilde{\Phi}(s+1)(4\pi\sqrt{\abs{mn}})^{-s}ds,
\end{equation}

\begin{equation}
\label{eq:L+/-Phi}
    \mathcal{L}^{\pm} \Phi(t) = \frac{1}{2\pi i}\int_{(2)}h_{\pm}(s,t)\tilde{\Phi} (s+1)(4\pi\sqrt{\abs{mn}})^{-s}ds,
\end{equation}
where
\begin{equation}
\label{eq:h+/-}
    h_{\pm}(s,t) = \frac{2^{s-1}}           {\pi}\Gamma(\frac{s}{2}+it)\Gamma(\frac{s}{2} - it)
    \begin{cases}
    \cos(\pi s/2), & \pm=+\\
    \cosh(\pi t), & \pm=- 
    \end{cases}
\end{equation}
and $\tilde{\Phi}(\cdot)$ denotes the Mellin transform of $\Phi(\cdot)$.

Let $\mathcal{H}_{it}(m,\psi)$ denote the set of Hecke-Maass newforms of conductor $m$, central character $\psi$, and spectral parameter  $it$. Similarly, we can define $\mathcal{H}_{k}(m,\psi)$ as the set of Hecke newforms of conductor $m$, central character $\psi$ and weight $k$. We can also define $\mathcal{H}_{it, \text{Eis}}(m,\psi)$ as the set of newform Eisenstein series of level $m$ and character $\chi$.

We work with the same explicit choice of basis as in \cite[\S 6.6]{py4m}. 

For any Hecke eigenform $\pi$ define
\begin{equation}
\label{eq:lambdapidelta}
    \lambda_\pi^{(\delta)} (n) = \sum_{d \mid \delta} d^{\frac12} \xi_\delta(d)\lambda_\pi(n/d),
\end{equation}

where $\lambda_\pi(n/d)=0$ if $\frac{n}{d}$ is not an integer. $\xi_\delta$'s are certain coefficients that arise during the construction of this basis. We refer the reader to \cite[\S 6.6]{py4m} for definitions, but mention that $\xi_\delta(d)$ is jointly multiplicative in $d$ and $\delta$, and satisfies $\xi_\delta (d)\ll (d\delta)^\varepsilon$
.

We can now state a version of the Bruggeman-Kuznetsov formula -
\begin{myprop}
\label{prop:Kuznetsov}
Let $\Phi \in C_C^{\infty}(\mathbb{R}_{>0})$. We have
\begin{equation}
\label{eq:Kuznetsov}
    \sum_{\text{gcd}(c,q)=1} \ov{\chi}(c)S(\ov{q}m,n,c)\Phi(q^{\frac12}c) = \mathcal{K}_{\text{Maass}}+\mathcal{K}_{\text{Eis}}+\mathcal{K}_{\text{hol}} \ .
\end{equation}
Here, 
\begin{equation}
\label{eq:KMaass}
    \mathcal{K}_{\text{Maass}} = \sum_{t_j} \mathcal{L}^{\pm}\Phi(t_j) \sum_{lr=q} \sum_{\pi \in \mathcal{H}_{it_j}(r,\chi)} \frac{4\pi\epsilon_{\pi}}{V(q)\mathscr{L}_\pi^*(1)}\sum_{\delta \mid l} \ov{\lambda}_\pi^{(\delta)}(\abs{m})\ov{\lambda}_\pi^{(\delta)}(\abs{n}),
\end{equation}
and
\begin{equation}
\label{eq:KEis}
    \mathcal{K}_{\text{Eis}} = \frac{1}{4\pi}\int_{-\infty}^{\infty}\mathcal{L}^{\pm}\Phi(t_j) \sum_{lr=q} \sum_{\pi \in \mathcal{H}_{it, \textrm{Eis}}(r,\chi)} \frac{4\pi\epsilon_{\pi}}{V(q)\mathscr{L}_\pi^*(1)}\sum_{\delta \mid l} \ov{\lambda}_\pi^{(\delta)}(\abs{m})\ov{\lambda}_\pi^{(\delta)}(\abs{n})dt,
\end{equation}
where one takes $\Phi+$ (resp. $\Phi^-$) if $mn > 0$ (resp. $mn<0$), and $\epsilon_{\pi}$ is the finite root number of $\pi$. Also,
\begin{equation}
\label{eq:Khol}
    \mathcal{K}_{\text{hol}} = \sum_{k>0, \text{ even}} \mathcal{L}^{\text{hol}}\Phi(k) \sum_{lr=q} \sum_{\pi \in \mathcal{H}_{k}(r,\chi)} \frac{4\pi\epsilon_{\pi}}{V(q)\mathscr{L}_\pi^*(1)}\sum_{\delta \mid l} \ov{\lambda}_\pi^{(\delta)}(\abs{m})\ov{\lambda}_\pi^{(\delta)}(\abs{n}),
\end{equation}
if $mn>0$, and $ \mathcal{K}_{\text{hol}}=0$ if $mn<0$.
\end{myprop}

\subsection{Spectral Large Sieve Inequality}
\label{subsec:speclargesieve}

We state the version of spectral large sieve inequality that we will use. This is a restatement of Lemma $7.4$ in \cite{py4m}.

Let us denote by

\begin{equation*}
    \int_{*\leq T} \ \text{ any of } \sum_{\abs{t_j}\leq T}, \sum_{k \leq T}, \text{ or } \int_{\abs{t}\leq T} dt,
\end{equation*}
according to whether $* = it_j, k,$ or $it,\text{Eis}$.

\begin{myprop}
\label{prop:spectrallargesieve}
For any sequence of complex numbers $a_n$, we have
\begin{equation}
    \int_{* \leq T} \sum_{\pi \in \mathcal{H}_{*}(q)} \left \vert \sum_{n\leq N} a_n \lambda_{\pi}(n) \right \vert^2 \ll_{\varepsilon} (T^2q+N)(qTN)^{\varepsilon}\sum_{n \leq N} \abs{a_n}^2.
\end{equation}
    
\end{myprop}

\subsection{Fourth Moment of Fourier Coefficients of Cusp Forms}

We will use the following result that bounds the fourth moment of the Fourier coefficients of cusp forms.

\begin{myprop}
    \label{prop:fourthmomtbound}
    Let $\lambda_f(n)$ denote the $n^{\text{th}}$ normalised Fourier coefficient of a primitive holomorphic or Maass cusp form $f$ for $SL_2(\mathbb{Z})$. Let $x \in \mathbb{R}$ be positive. Then, for any $\varepsilon >0$,
    \begin{equation}
        \label{eq:fouthmomtbound}
        \sum_{n\leq x} \abs{\lambda_f(n)}^4 \ll_f x^{1+\varepsilon}.
    \end{equation}
\end{myprop}

Proposition \ref{prop:fourthmomtbound} follows from Theorem 1.5 and Remark 1.7 in \cite{laulu}. In fact, Moreno and Shahidi first obtained a similar bound for the fourth moment of the Ramanujan $\tau$-function in \cite{morsha}. They were able to prove $\sum_{n\leq x} \tau(n)^4 \sim cx\log{x}$, for some positive constant $c$.

This result was then extended to primitive holomorphic cusp forms by L{\"u} (see \cite{lu}) and to primitive Maass cusp forms by Lau and L{\"u} (see \cite{laulu}).

\section{Reduction of Theorem \ref{thm:MainThm} to Theorem \ref{thm:redthm}}
\label{sec:ProofReduction}
In this section, we will prove Theorem \ref{thm:MainThm}, assuming Theorem \ref{thm:redthm} is true.

Using Corollary \ref{cor:approx.func} and a dyadic partition of unity, it suffices to show (via an application of Cauchy-Schwarz inequality) that for any $\varepsilon >0$, 
\begin{equation}
    \label{eq:SwNalpha}
    S(N,\alpha) \coloneqq \sum_{\psi \Mod{p^2}} \abs*{\sideset{}{^*}\sum_{n \asymp N} \lambda_f (n)\psi(n)\alpha(n)w_N(n) }^2 \ll N p^{2+\varepsilon}, 
\end{equation}
where $w_N$ is a smooth function supported on $[N,2N]$ satisfying $w_N^{(j)}(x) \ll N^{-j}$ and $N \ll p^{3+\varepsilon}$. Recall that $\sideset{}{^*}\sum_n$ denotes that the sum is over the values of $n$ relatively prime to $p$.

Expanding the square in $S(N,\alpha)$ and rearranging terms, we get
\begin{equation}
\label{eq:expandsq}
    S(N,\alpha) = \sideset{}{^*}\sum_{m,n}  \lambda_f(m)\ov{\lambda_f}(n)\alpha(m)\ov{\alpha(n)}w_N(m)w_N(n) \sum_{\psi \Mod{p^2}} \psi(m\ov{n}).
\end{equation}
The inner sum on the right-hand side of \eqref{eq:expandsq} vanishes unless $m \equiv n \Mod{p^2}$, in which case it is equal to $\phi(p^2)$. We can then rewrite $S(N,\alpha)$ as a sum of diagonal and off-diagonal terms. We get

\begin{equation}
    S(N,\alpha) =\sideset{}{^*} \sum_{n}  \phi(p^2) \left \lvert \lambda_f(n) \right \rvert^2w_N(n)^2 + S_0(N,\alpha),
\end{equation}
where 
\begin{equation}
S_0(N,\alpha) = \sideset{}{^*}\sum_{\substack{m \neq n \\ m\equiv n \Mod{p^2}}} \phi(p^2)\lambda_f(m) \ov{\lambda_f}(n)\alpha(m) \ov{\alpha(n)}w_N(m)w_N(n),
\end{equation}
is the off-diagonal term.

Using known bounds on $w_N$ and $\lambda_f(n)$ the diagonal term is of the order $O(Np^{2+\varepsilon})$. Thus \eqref{eq:SwNalpha} follows if we can show that the off-diagonal terms ($S_0(N,\alpha)$) satisfy the same bound. 

Now, for $S_0(N,\alpha)$, we note that it suffices to consider the sum for the terms with $m>n$, since the sum for the terms with $m<n$ is just the complex conjugate of the $m>n$ sum.

Letting $l = \frac{m-n}{p^2}$, and considering only positive values of $l$, we get (using $\Re(z)$ to denote the real part of $z$ , $z \in \mathbb{C}$)
\begin{equation}
\label{eq:S0bound}
   S_0(N,\alpha) = 2 \Re \pb{\phi(p^2)\sum_{l} \ \sideset{}{^*}\sum_{n} \lambda_f(n+p^2l) \ov{\lambda_f}(n) \alpha(n+p^2l)\ov{\alpha(n)}w_N(n+p^2l) w_N(n)}.
\end{equation}

We define 
\begin{equation} 
\label{eq:lintro}
    S_1(N,\alpha) \coloneqq \sideset{}{^*} \sum_{n,l}  \lambda_f(n+p^2l)\ov{\lambda_f}(n)\alpha(n+p^2l)\ov{\alpha(n)}w_N(n+p^2l) w_N(n),
\end{equation}
and 
\begin{equation} 
\label{eq:lintro2}
    S_2(N,\alpha) \coloneqq \sum_{l \equiv 0 \shortmod{p}}  \ \sideset{}{^*}\sum_{n} \lambda_f(n+p^2l)\ov{\lambda_f}(n)\alpha(n+p^2l)\ov{\alpha(n)}w_N(n+p^2l) w_N(n).
\end{equation}

Note that $S_1(N,\alpha)$ is precisely the left hand side in \eqref{eq:redthm} in Theorem \ref{thm:redthm}.

Then \eqref{eq:S0bound} can be written as 
\begin{equation}
     S_0(N,\alpha) = 2\phi(p^2) \Re\pb{(S_1(N,\alpha)  +  S_2(N,\alpha))}. 
\end{equation}

So, in order to prove Theorem \ref{thm:MainThm}, it suffices to show that

\begin{equation}
\label{eq:S1finalbound}
    S_1(N,\alpha) \ll \frac{Np^{2+\varepsilon}}{\phi(p^2)} \ll Np^{\varepsilon}, 
\end{equation}

and
\begin{equation}
\label{eq:S2finalbound}
    S_2(N,\alpha) \ll Np^{\varepsilon}.
\end{equation}
The bound in \eqref{eq:S1finalbound} is just a restatement of Theorem \ref{thm:redthm}. As $N \ll p^{3+\varepsilon}$, \eqref{eq:S2finalbound} can be obtained by simply taking trivial bounds for all the terms. Thus Theorem \ref{thm:MainThm} follows.

\section{Harmonic Analysis}
\label{sec:harmonic}
We note that bounding all the terms trivially, we get that $S_1(N,\alpha) \ll \frac{N^2}{p^2}p^{\varepsilon} \ll Np^{1+\varepsilon}$.
\\The proof for Theorem \ref{thm:redthm} starts with the introduction of the delta symbol, followed by the use of Voronoi summation formula to get additional cancellations on the terms that appear. Barring a `main term' this approach proves to be very fruitful, and we end up getting the desired upper bounds. 
\subsection{Application of the Delta Symbol}
\label{subsec:deltaapplication}
As gcd$(n,p)=1$ in the definition of $S_1(N,\alpha)$, Proposition \ref{prop:Postnikov} guarantees the existence of a non-zero $a_\alpha \Mod{p}$ such that $\alpha(n+p^2l)\ov{\alpha(n)} = e_p(a_\alpha l\ov{n})$.

We recall that from Remark \ref{rem:nlbound}, we know that $0<l\leq \frac{N}{p^2}$, and $N \leq n \leq 2N$ in \eqref{eq:lintro}. Using $m = n + p^2l$, we can rewrite \eqref{eq:lintro} as 
\begin{equation}
    S_1(N,\alpha) =  \sum_{m}\sideset{}{^*} \sum_{l,n}  \lambda_f(n+p^2l) \ov{\lambda_f}(n)e_p(a_\alpha l \ov{n})w_N(m) w_N(n) \delta(m-n-p^2l).
\end{equation}
We can now introduce the $\delta$-symbol in the above sum and get via \eqref{eq:deltafourier} (using $C=\sqrt{N}$),
\begin{multline}
 \label{eq:S1Nalpha}
    S_1(N,\alpha)= \sum_{m}\sideset{}{^*} \sum_{l,n}\lambda_f(n+p^2 l) \ov{\lambda_f}(n)e_p(a_\alpha l \ov{n})w_N(m) w_N(n) \\
    \cdot \sum_{c \leq 2C} S(0,m-n-p^2l;c)\int_{-\infty}^{\infty} g_c(v)e\pb{(m-n-p^2l)v}dv.
\end{multline}
Expanding out the Ramanujan sum and combining the $m$ and $n$ terms, we get
\begin{equation}
\label{eq:beforeVoronoi}
    S_1(N,\alpha) =  \sideset{}{^*} \sum_{l} \sum_{c\leq 2C} V_{N,l}(c). 
\end{equation}
Here,
\begin{equation}
    V_{N,l}(c) =  \sideset{}{^*}\sum_{a \shortmod{c}}e\pb{\frac{-ap^2l}{c}}\int_{-\infty}^{\infty} g_c(v)e(-p^2l v) \cdot T_1(a,c,v) \cdot T_2(a,c,v) 
     dv,
\end{equation}

\begin{equation}
    \label{eq:T1defn}
    T_1(a,c,v) = \sum_{m=1}^{\infty} \lambda_f(m) e\pb{\frac{am}{c}}w_N(m)e(mv),
\end{equation}

\begin{equation}
     \label{eq:T2defn}
     T_2(a,c,v) =\sum_{\text{gcd}(n,p)=1} \ov{\lambda_f}(n) e_p(a_\alpha l \ov{n}) e\pb{\frac{-an}{c}} w_N(n)e(-nv).
\end{equation}

We further split \eqref{eq:beforeVoronoi} based on whether $c$ is relatively prime to $p$, or not. We get

\begin{equation}
\label{eq:beforeVorono2}
    S_1(N,\alpha) = S_{3}(N,\alpha) +  S_4(N,\alpha), 
\end{equation}

where

\begin{equation}
   \label{eq:S3defn}
   S_{3}(N,\alpha) = \sideset{}{^*}\sum_{l}\sideset{}{^*}\sum_{c \leq 2C}  V_{N,l}(c) , 
\end{equation}

and
\begin{equation}
   \label{eq:S4defn}
    S_4(N,\alpha)  = \sideset{}{^*}\sum_{l} \sum_{\substack{c \leq 2C \\ p \mid c} } V_{N,l}(c).  
\end{equation}

We have the following two lemmas.

\begin{mylemma}
    \label{lem:S3}
Let $\varepsilon >0$ and $N \ll p^{3+\varepsilon}$. Let $S_3(N,\alpha)$ be defined as in \eqref{eq:S3defn}. Then
    \begin{equation}
    \label{eq:S3}
 S_3(N,\alpha) \ll_{\varepsilon} Np^{\varepsilon}.
        \end{equation}
\end{mylemma}

\begin{mylemma}
    \label{lem:S4}
Let $\varepsilon >0$ and $N \ll p^{3+\varepsilon}$. Let $S_4(N,\alpha)$ be defined as in \eqref{eq:S4defn}. Then
    \begin{equation}
    \label{eq:S4}
S_4(N,\alpha) \ll_{\varepsilon} Np^{\varepsilon}.
    \end{equation}
   
\end{mylemma}

It is clear that Theorem \ref{thm:redthm} follows from these two lemmas.

We proceed with the proof of Lemma \ref{lem:S3} for now, and delay the proof of Lemma \ref{lem:S4} to Section \ref{sec:Remterms}.

\subsection{Voronoi Summation}
\label{subsec:voronoiapply}

We recall that $T_1(a,c,v)$ and $T_2(a,c,v)$ are defined in \eqref{eq:T1defn} and \eqref{eq:T2defn}.

We use Proposition \ref{prop:Voronoi}  on $T_1(a,c,v)$ to get
\begin{equation}
\label{eq:msumpostVoronoi}
    T_1(a,c,v) = \frac{1}{c}\sum_{m=1}^\infty \lambda_f(m)e\pb{\frac{-\ov{a}m}{c}}\widetilde{w}_{c,v,N}(m). 
\end{equation}

Here, 

\begin{equation}
\label{eq:wmtilde}
    \widetilde{w}_{c,v,N}(m) = 2\pi i^k \int_o^\infty J_{k-1}\pb{\frac{4\pi\sqrt{mx}}{c}}e(xv)w_N(x)dx.
\end{equation}
For $T_2(a,c,v)$, (note that gcd$(a_\alpha l,p)=1$) we can use 
Proposition \ref{prop:modifiedVoronoi},  to get
\begin{equation}
\label{eq:nsumpostVoronoi}    
    T_2(a,c,v) = D_0 + D_1 + D_2 +D_3 
\end{equation}   

    where
    
    \begin{equation}
        \label{eq:D_0}
        D_0 =  \sum_{\text{gcd}(n,p) = 1} \frac{\ov{\lambda_f}(n)\widetilde{w}_{pc,-v,N}(n)}{p^2c}e\pb{\frac{\ov{ap^2}n}{c}}\text{Kl}_3(n\ov{c}^2a_\alpha l,1,1;p), 
    \end{equation}

    \begin{equation}
        \label{eq:D1}
        D_1 =  \sum_{\text{gcd}(n,p^2) = p} \frac{\ov{\lambda_f}(n)\widetilde{w}_{pc,-v,N}(n)}{p^2c}e\pb{\frac{\ov{ap^2}n}{c}},
    \end{equation}
     \begin{equation}
        \label{eq:D2}
        D_2 = \sum_{n=1}^{\infty} \frac{\ov{\lambda_f}(np) \ov{\lambda_f}(p)\widetilde{w}_{pc,-v,N}(p^2n)}{p^2c}e\pb{\frac{\ov{a}n}{c}},
    \end{equation}
    and
    \begin{equation}
        \label{eq:D3}
        D_3 = -\pb{1+\frac{1}{p}}\sum_{n=1}^{\infty}\frac{\ov{\lambda_f}(n)\widetilde{w}_{pc,-v,N}(p^2n)}{pc}e\pb{\frac{\ov{a}n}{c}}.
    \end{equation}
  
 Here, $\widetilde{w}_{-}(\cdot) $ is the same as in \eqref{eq:wmtilde}.

Thus using \eqref{eq:msumpostVoronoi} and \eqref{eq:nsumpostVoronoi}, we have

\begin{equation}
\label{eq:afterVoronoi}
   S_3(N,\alpha) = E_0 +  E_1 + E_2 + E_3, 
\end{equation}

where
\begin{multline}
\label{eq:E_0}
    E_0 = \sideset{}{^*}\sum_{l} \ \sideset{}{^*}\sum_{c \leq 2C} \ \sideset{}{^*}\sum_{a \shortmod{c}}e\pb{\frac{-ap^2l}{c}}\int_{-\infty}^{\infty} g_c(v)e(-p^2l v) \pb{\frac{1}{c}\sum_{m } \lambda_f(m)e\pb{\frac{-\ov{a}m}{c}}\widetilde{w}_{c,v,N}(m)}\\
    \pb{\sideset{}{^*}\sum_{n} \frac{\ov{\lambda_f}(n)\widetilde{w}_{pc,-v,N}(n)}{p^2c}e\pb{\frac{\ov{ap^2}n}{c}}\text{Kl}_3(n\ov{c}^2a_\alpha l,1,1;p)} dv,
\end{multline} 
is the contribution to $S_3(N,\alpha)$ from $D_0$.
\\ Similarly, $E_1, E_2, E_3$ represent the contributions to $S_3(N,\alpha)$ from the remaining three terms of \eqref{eq:nsumpostVoronoi}.

We have the two following lemmas.

\begin{mylemma}
    \label{lem:E0bound}
    Let $\varepsilon>0$ and let $N \ll p^{3+\varepsilon}$. Let $E_0$ be defined as in \eqref{eq:afterVoronoi}. We have 
    \begin{equation}
        E_0 \ll_\varepsilon Np^{\varepsilon}.
    \end{equation}
\end{mylemma}

\begin{mylemma}
    \label{lem:E123bound}
Let $\varepsilon>0$ and let $N \ll p^{3+\varepsilon}$. Let $E_1, E_2$ and $E_3$ be defined as in \eqref{eq:afterVoronoi}. We have 
    \begin{equation}
        E_j \ll_\varepsilon N^{\frac{3}{4}}p^{\varepsilon}.
    \end{equation}
\end{mylemma}

It is clear that together, these two lemmas imply Lemma \ref{lem:S3}. We prove Lemma \ref{lem:E0bound} first, and postpone the proof of Lemma \ref{lem:E123bound} to Section \ref{sec:Remterms}. 

The main idea for proving Lemma \ref{lem:E0bound} is via spectral analysis, using the Bruggeman-Kuznetsov formula, followed by applying spectral large sieve inequality, along with other results, to bound each of the resultant terms.

We first note that the $a$-sum in \eqref{eq:E_0} forms a Kloosterman sum. Using this we have 

\begin{equation}    
    \label{eq:E_0mod}
    E_0 =  \sideset{}{^*}\sum_{l,n}\sum_{m} \frac{\lambda_f(m)\ov{\lambda_f}(n)}{p^2} \sideset{}{^*}\sum_{c\leq 2C}\frac{1}{c^2} S(\ov{p}^2n-m,-p^2l;c)\text{Kl}_3(n\ov{c}^2a_\alpha l,1,1;p) I_N(c,l,m,n),
\end{equation}

where 

\begin{equation}
    \label{eq:INclmn}
    I_N(c,l,m,n) = \int_{-\infty}^{\infty} g_c(v)e(-p^2lv) \widetilde{w}_{c,v,N}(m) \widetilde{w}_{pc,-v,N}(n) dv.
\end{equation}

Using the fact that when gcd$(r,p)=1$,  
 \begin{equation}
     \text{Kl}_3(r,1,1;p) = \frac{1}{\phi(p)}\sum_{\chi(p)} \tau(\chi)^3\chi(r),
 \end{equation}

we can rewrite \eqref{eq:E_0mod} as

\begin{equation}    
    \label{eq:E_0mod2}
    E_0 =   \sideset{}{^*}\sum_{l,n}\sum_{m} \frac{\lambda_f(m)\ov{\lambda_f}(n)}{\phi(p)p^2} \sum_{\chi(p)} \tau(\chi)^3 \chi(n a_\alpha l)\cdot \mathcal{K},
\end{equation}
where
\begin{equation}
    \label{eq:KuznetsovK}
    \mathcal{K} = \sideset{}{^*}\sum_{c \leq 2C}\frac{1}{c^2}S(\ov{p}(n-p^2m),-pl;c) \chi(\ov{c}^2) I_N(c,l,m,n),
\end{equation}

We want to use the Bruggeman-Kuznetsov formula here. However, before that, we need to dyadically decompose \eqref{eq:E_0mod2}.

\subsection{Dyadic Partition of Unity}
\label{subsec:dyadic}
We want to modify \eqref{eq:E_0mod2} by applying a dyadic decomposition. We recall some definitions from Section $6.1$ in \cite{ky5m}.
\\We define a number $N$ as dyadic if $N =2^{\frac{k}{2}}$ for some $k \in \mathbb{Z}$. 
\\Let $g$ be a fixed smooth function supported on the interval $[1,2]$. A dyadic partition of unity is a partition of unity of the form $\sum_{k \in \mathbb{Z}} g(2^{-\frac{k}{2}}x) \equiv 1,$ for $x>0$. 
\\The family $g_N(x) = g(\frac{x}{N})$ forms a $1$-inert family of functions. 

We want to dyadically decompose $E_0$ (as defined in \eqref{eq:E_0mod2}) in the $l,m,n,c,$ and $v$ variables. As $l,m,n,c$ are all positive integers, this decomposition is relatively simple, but a dyadic decomposition in the $v$-variable is a bit more involved. We show that first.

\subsubsection{Dyadic decomposition of $I_{N}(c,l,m,n)$}

Recall that $g_c(v)$ is defined in \eqref{eq:gc}, and satisfies bounds given by \eqref{eq:gcbound1} and \eqref{eq:gcbound2}. Also $\widetilde{w}_{c,v,N}(m)$ is defined in \eqref{eq:wmtilde}, and $I_N(c,l,m,n)$ is defined in \eqref{eq:INclmn}.

We define, 
\begin{equation}
    \label{eq:Hdefn}
     H(v) \coloneqq g_c(v)e(-p^2lv)\widetilde{w}_{c,v,N}(m)\widetilde{w}_{pc,-v,N}(n).     
\end{equation}
We then have, 
\begin{equation}
    \label{eq:HInclm}
    I_N(c,l,m,n) = \int_{-\infty}^{\infty} H(v)dv. 
\end{equation}

The bound in \eqref{eq:gcbound2} immediately implies that
for any $A>0$ arbitrarily large,
\begin{equation}
    \label{eq:bounded_I}
    I_N(c,l,m,n)=  \int_{\abs{v}\ll \frac{p^\varepsilon}{cC}} H(v) dv + O(p^{-A}).
\end{equation}

We now state and prove the following lemma that allows us to apply a dyadic decomposition.

\begin{mylemma}
    \label{lem: Inclmndyadic}
    Let $A>0$ be arbitrarily large. Then there exists $V_0 >0$, and a smooth function $F(x)$ with $\abs{F(x)}\leq 1$ (both depending on $A$), such that
    \begin{equation}
        \label{eq:Inclmndyadic}
        I_N(c,l,m,n) = \int_{V_0\leq \abs{v} \ll \frac{p^\varepsilon}{cC}} H(v) F(v) dv + O(p^{-A}) .
    \end{equation}
\end{mylemma}

Note that since the integral in \eqref{eq:Inclmndyadic} is defined away from zero, we can easily apply a dyadic decomposition to it.

\begin{proof}
  
  Fix $A>0$. For any $V_0 > 0$, we can choose a smooth even function $G(x)$ compactly supported on $[-2V_0,2V_0]$ such that $G(x) \leq 1$ and $G(x) = 1$ on $[-V_0,V_0]$. Notice that $1-G(\cdot)$ is supported on $\abs{x} > V_0$.

Thus for any $V_0 < \frac{p^\varepsilon}{cC}$, we can rewrite \eqref{eq:HInclm} as 

\begin{equation}
    \label{eq:Idyadic2}
     I_N(c,l,m,n) =  \int_{\abs{v}\ll \frac{p^\varepsilon}{cC}} H(v) G(v) dv + \int_{\abs{v}\ll \frac{p^\varepsilon}{cC}} H(v) (1-G(v)) dv  + O(p^{-A}).
\end{equation}

 Now, as $G$ is supported on $[-2V_0,2V_0]$ and equals $1$ on $ [-V_0,V_0]$, we can modify this further as,
 
 \begin{equation}
    \label{eq:Idyadic3}
     I_N(c,l,m,n) =  \int_{\abs{v}\leq 2V_0} H(v) G(v) dv + \int_{V_0 \leq \abs{v}\ll \frac{p^\varepsilon}{cC}} H(v) (1-G)(v) dv  + O(p^{-A}).
\end{equation}

As $N \ll p^{3+\varepsilon}$, and $\int_{\abs{v}\leq 2V_0} H(v) G(v) dv  \ll V_0N^{2+\varepsilon}$, if we choose $V_0 = p^{-A-6}$ then this integral is also $O(p^{-A})$. We choose $F(x) = 1-G(x)$ which completes the proof.

\end{proof}

 Recall that $g(\cdot)$ is a fixed smooth function supported on $[1,2]$. We can now apply a dyadic decomposition to \eqref{eq:Inclmndyadic}, to get
\begin{equation}
    \label{eq:Inclmndyadic2}
    I_N(c,l,m,n) = \sum_{\substack{V \text{dyadic}\\ V_0 \leq \abs{V} \ll \frac{p^\varepsilon}{cC} }} I_{N,V} (c,l,m,n) + O(p^{-A}),
\end{equation}

where

\begin{equation}
    \label{eq:INVclmn}
    I_{N,V} (c,l,m,n) = \int_{V}^{2V} g_c(v)e(-p^2lv) \widetilde{w}_{c,v,N}(m) \widetilde{w}_{pc,-v,N}(n) F(v) g\pb{\frac{v}{V}} dv.
\end{equation}
As $F(\cdot)$ and $g(\cdot)$ are  smooth, we can absorb them in $\widetilde{w}_{c,v,N}(m)$, and by a slight abuse of notation, we denote the new function as $\widetilde{w}_{c,v,N}(m)$ also.
\subsubsection{Dyadic Decomposition of $E_0$}

We can now use apply a dyadic decomposition to \eqref{eq:E_0mod2}, to get

\begin{equation}
    \label{eq:E_0decomposition}
    E_0 = \sum_{\substack{C_0,L_0,M_0,N_0,V\\ \text{dyadic}}} E_{C_0,L_0,M_0,N_0,V} + O(p^{-A}),
\end{equation}

where

\begin{equation}
\label{eq:E_0dyadic}
    E_{C_0,L_0,M_0,N_0,V} = \sum_{m}\sideset{}{^*}\sum_{l,n,c} \frac{\lambda_f(m)\ov{\lambda_f}(n)}{\phi(p)p^2} \sum_{\chi(p)} \tau(\chi)^3 \chi(n a_\alpha l) \cdot \mathcal{K}_{C_0,L_0,M_0,N_0,V},
\end{equation} 
with
\begin{equation}
    \label{eq:KuznetsovKdyadic}
    \mathcal{K}_{C_0,L_0,M_0,N_0,V} = \sideset{}{^*}\sum_{c}\frac{1}{c^2}S(\ov{p}(n-p^2m),-pl;c) \ov{\chi^2}({c})I_{N,V}(c,l,m,n)  g_T\pb{l,m,n,c},
\end{equation}

and
\begin{equation}
    g_T\pb{l,m,n,c} = g\pb{\frac{l}{L_0}}g\pb{\frac{m}{M_0}}g\pb{\frac{n}{N_0}}g\pb{\frac{c}{C_0}}.
\end{equation}

 Also, we claim the dyadic numbers satisfy 
\begin{equation}
    \label{eq:dyadicbounds}
     2^{-\frac{1}{2}} \leq M_0 \ll p^{\varepsilon}, \  2^{-\frac{1}{2}} \leq N_0 \ll p^{2+\varepsilon}, \   2^{-\frac{1}{2}} \leq \ C_0 \leq 2\sqrt{N}, \ 2^{-\frac{1}{2}} \leq L_0 \leq \frac{N}{p^2}, \ V_0 \leq \abs{V} \ll \frac{p^\varepsilon}{cC}.
\end{equation}

The bounds for $C_0,L_0$ and $V_0$ are clear. The bounds for $M_0$ and $N_0$ follow as a consequence of Lemmas \ref{lem:wtildebehavior} and \ref{lem:wntildebehavior}.

We will use the Bruggeman-Kuznetsov formula on \eqref{eq:KuznetsovKdyadic} in Section \ref{sec:SpectralAnalysis}. We finish this section by stating some results regarding the behaviour of $I_{N,V}(c,l,m,n)$. 

\subsection{Analysis of \texorpdfstring{$I_{N,V}(c,l,m,n)$}{I\_NV(c,l,m,n)}}
\label{subseq:INVanalysis}
We use stationary phase methods to analyse the behaviour of some of the oscillatory integrals we have obtained so far. The general scheme is to identify regions when the integral is highly oscillatory, and when it is not.
\subsubsection{Bounds for $\widetilde{w}_{c,v,N}(m)$}

We want to use the properties of inert functions to analyse the behaviour of  $\widetilde{w}_{c,v,N}(m)$.
\begin{mylemma}
    \label{lem:wtildebehavior}
    Let $\widetilde{w}_{c,v,N}(m)$ be as in \eqref{eq:wmtilde}. Let $l,m,n,c,v$ be in dyadic intervals as before. We have
    \begin{itemize}
        \item[(a)] (Non-Oscillatory) If $\frac{\sqrt{M_0N}}{C_0} \ll p^{\varepsilon}$, then
        \begin{equation}
            \label{eq:wmtildenonosc}
            \widetilde{w}_{c,v,N}(m)= \pb{\frac{\sqrt{M_0N}}{C_0}}^{k-1}\cdot N \cdot W_{T}(c,m,v).
        \end{equation}
        Here, $T= (V,M_0,C_0)$
        and $W_{T}(c,m,v)$ is a $p^{\varepsilon}$-inert family in $c,m,v$. \\Also, $\widetilde{w}_{c,v,N}(m)$ is small unless $N\abs{V} \ll p^{\varepsilon}$.
        \item[(b)] (Oscillatory) If $\frac{\sqrt{M_0N}}{C_0} \gg p^{\varepsilon}$, then
        \begin{equation}
            \label{eq:wmtildeosc}
            \widetilde{w}_{c,v,N}(m)=\frac{NC_0}{\sqrt{M_0N}}\cdot e\pb{-\frac{m}{c^2v}} \cdot W_{T}(c,m,v) + O(p^{-A}).
        \end{equation}
         Here, $T= (V,M_0,C_0)$ and $W_{T}(c,m,v)$ is $p^{\varepsilon}$-inert in $c,m,v$. $A$ can be chosen to be arbitrarily large. 
         \\Also, $\widetilde{w}_{c,v,N}(m)$ is small unless $N\abs{V} \asymp \frac{\sqrt{M_0N}}{C_0}$.

    \end{itemize}

\end{mylemma}

\begin{proof}

Note that, given the conditions of the lemma, and the fact that the integral in \eqref{eq:wmtilde} can be truncated up to $N$, we have $\frac{4\pi\sqrt{mx}}{c} \asymp \frac{\sqrt{M_0N}}{C_0}$.
\begin{itemize}
    \item [(a)]
        If $\frac{\sqrt{M_0N}}{C_0} \ll p^{\varepsilon}$, the Bessel function is not oscillatory, and we can write, $J_{k-1}(t) = t^{k-1}W(t)$, where $t^j W^{(j)}(t) \ll T_0$ with $T_0 \ll p^{\varepsilon}$. This is the same derivative bound satisfied by a $T_0$-inert function (again, when $Y \ll p^{\varepsilon}$). Thus we have
        
        \begin{equation}
            \label{eq:wmtildenonosc2}
            \widetilde{w}_{c,v,N}(m) = 2 \pi i^k \int_{0}^{\infty} \pb{\frac{4\pi\sqrt{mx}}{c}}^{k-1} W\pb{\frac{4\pi\sqrt{mx}}{c}}w_N(x)e(xv) dx.
        \end{equation}
        
        We can now use Proposition \ref{prop:inertxpower} and \ref{prop:FourierInert} to get \eqref{eq:wmtildenonosc}. Proposition \ref{prop:FourierInert} also implies that $\widetilde{w}_{c,v,N}(m)$ is small unless $\abs{V} \ll \frac{1}{N}p^{\varepsilon}$, or, $N\abs{V} \ll p^{\varepsilon}$.
    \item [(b)]
        We use the fact that when $t \gg 1$, the J-Bessel Function has an oscillatory behaviour and satisfies,
        \begin{equation}
            \label{eq:JBesselosc}
            J_{k-1}(t) = \frac{1}{\sqrt{t}}(e^{it}W_+(t) + e^{-it}W_-(t)),
        \end{equation}
        
         where $W_+$ and $W_-$ satisfy the same derivative bounds as a $1$-inert family of functions. Thus we have (using $t=\frac{4\pi\sqrt{mx}}{c}, \ T_0 = \frac{\sqrt{M_0N}}{C_0} $),
        \begin{equation}
            \label{eq:wmtildeosc2}
            \widetilde{w}_{c,v,N}(m) = \sum_{\pm} \int_{0}^{\infty} e(xv)w_N(x)\frac{e^{\pm it}}{\sqrt{t}}W_{\pm}(t)dx = \sum_{\pm}  \frac{1}{\sqrt{T_0}}\int_{0}^{\infty} e^{i\phi_{\pm}(x)}W_{N_{\pm}}(x)dx, 
        \end{equation}
        
        where
        
        \begin{equation}
            \label{eq:phiplusminus}
            \phi_{\pm}(x) = 2 \pi x v \pm 4 \pi \frac{\sqrt{mx}}{c},
        \end{equation}
        
        and
        
        \begin{equation}
            \label{eq:Wplusminusnonosc}
            W_{N_{\pm}}(x) = \frac{w_N(x)W_{\pm}(t)\sqrt{T_0}}{\sqrt{t}}.
        \end{equation}
        
        We want to use Proposition \ref{prop:stationeryphaseinert} to analyse \eqref{eq:wmtildeosc2}. Note that $W_{N_{\pm}}(x)$ forms a $p^{\varepsilon}$-inert family.
        
        Notice that, as $c \leq \sqrt{N}$, when $V>0$, $\abs{\phi_+'(x)} = \abs{2 \pi \pb{v +  \frac{\sqrt{m}}{c\sqrt{x}}}} \geq \frac{1}{N} $. So we can use Proposition \ref{prop:stationeryphaseinert} (a) to conclude that the `+' integral is small. Similarly, when $V<0$, we can conclude that the `-' integral is small. 
        \\So, it suffices to consider the `-' (resp. `+') integral only when $V > 0$ (resp. $V<0$). As the cases are similar, we only consider the first. So, assume $V>0$.

        We can check that $\phi_{-}''(x) = \frac{\pi \sqrt{m}}{c x\sqrt{x}} \gg \frac{1}{N^2}$. Also, $\phi_-'(x_0) = 0$, when $x_0 = \frac{m}{c^2v^2}$; and $\phi''(x_0) = \frac{\pi c^2v^3}{m}$.
        \\Using Proposition \ref{prop:stationeryphaseinert} (b), we can rewrite \eqref{eq:wmtildeosc2} (when $V>0$) as  
        \begin{equation}
            \label{eq:wmtildeosc3}
            \widetilde{w}_{c,v,N}(m) = \frac{1}{\sqrt{T_0}} \frac{1}{\sqrt{\phi''(x_0)}}  e^{i\phi_{-}(x_0)} \cdot W_{T}(x_0) + O(p^{-A}).
        \end{equation}
        
        where $T= (V,M_0,C_0)$ and $W_{T}(x_0)$ is $p^{\varepsilon}$-inert in $c,m,v$, and is supported on $x_0 \asymp N$.
        
        The condition $x_0 \asymp N$ implies $\frac{M_0}{N{C_0}^2} \asymp V^2 $. This, along with the fact that $T_0 =\frac{\sqrt{M_0N}}{C_0} $ gives \eqref{eq:wmtildeosc}. This also implies the integral is small unless $N\abs{V} \asymp \frac{\sqrt{M_0N}}{C_0}$.
        
\end{itemize}
\end{proof}
\begin{myremark}
    \label{eq:voronoidualbound}
    We note that as $\abs{V} \leq \frac{1}{C_0C} = \frac{1}{C_0\sqrt{N}}$, and $C_0 \leq \sqrt{N}$, $N\abs{V} \asymp \frac{\sqrt{M_0N}}{C_0}$ is only possible when $M_0 \ll p^{\varepsilon}$. As the condition in $(a)$ automatically implies $M_0 \ll p^{\varepsilon}$, it suffices to only consider $M_0 \ll p^{\varepsilon}$ for \eqref{eq:E_0decomposition}.
\end{myremark}
We can also get a similar result for $w_{pc,-v,N}(n)$ by following the same arguments. We state the result here. This also allows us to only consider $N_0\ll p^{2+\varepsilon}$ for \eqref{eq:E_0decomposition}.

\begin{mylemma}
    \label{lem:wntildebehavior}
   Let $l,m,n,c,v$ be in dyadic intervals at $L_0,M_0,N_0,C_0,V$. We have
    \begin{itemize}
        \item[(a)] (Non-Oscillatory) If $\frac{\sqrt{N_0N}}{pC_0} \ll p^{\varepsilon}$, then
        \begin{equation}
            \label{eq:wntildenonosc}
            \widetilde{w}_{pc,-v,N}(n)= \pb{\frac{\sqrt{N_0N}}{pC_0}}^{k-1}\cdot N \cdot W_{T}(c,n,v).
        \end{equation}
        Here, $T= (V,N_0,C_0)$
        and $W_{T}(c,n,v)$ is a $p^{\varepsilon}$-inert family in $c,n,v$. \\Also, $\widetilde{w}_{pc,-v,N}(n)$ is small unless $N\abs{V} \ll p^{\varepsilon}$
        \item[(b)] (Oscillatory) If $\frac{\sqrt{N_0N}}{pC_0} \gg p^{\varepsilon}$, then
        \begin{equation}
            \label{eq:wntildeosc}
            \widetilde{w}_{pc,-v,N}(n)=\frac{NpC_0}{\sqrt{N_0N}}\cdot e\pb{\frac{n}{p^2c^2v}} \cdot W_{T}(c,n,v) + O(p^{-A}).
        \end{equation}
         Here, $T= (V,N_0,C_0)$ and $W_{T}(c,n,v)$ is $p^{\varepsilon}$-inert in $c,n,v$. $A$ can be chosen to be arbitrarily large. 
         \\Also, $\widetilde{w}_{pc,-v,N}(n)$ is small unless $N\abs{V} \asymp \frac{\sqrt{N_0N}}{pC_0}$.

    \end{itemize}
\end{mylemma}

\subsubsection{Bounds on $I_{N,V}(c,l,m,n)$}

We can use Lemma \ref{lem:wtildebehavior} and Lemma \ref{lem:wntildebehavior}, along with Proposition \ref{prop:stationeryphaseinert} to analyse the behaviour of 
\begin{equation}
    \label{eq:Invclmndefn}
    I_{N,V}(c,l,m,n) = \int_{V}^{2V} g_c(v)e(-p^2lv)\widetilde{w}_{c,v,N}(m)\widetilde{w}_{pc,-v,N}(n) dv .
\end{equation}

We prove the following proposition.

\begin{myprop}
    \label{prop:Invlcmnbehavior}
    Let $I_{N,V}(c,l,m,n)$ be as in \eqref{eq:Invclmndefn}. Let Let $l,m,n,c,v$ be in dyadic intervals at $L_0$, $M_0$, $N_0$, $C_0$, $V$ satisfying \eqref{eq:dyadicbounds}. Then
    \begin{itemize}
       \item[(a)] (Non-oscillatory) If $N\abs{V} \ll p^{\varepsilon}$, then
        \begin{equation}
            \label{eq:Invclmnnonosc}
            I_{N,V}(c,l,m,n)= N^2 V \cdot W_{T}(c,l,m,n).
        \end{equation}
        Here, $T= (C_0,L_0,M_0,N_0)$
        and $W_{T}(c,l,m,n)$ is a $p^{\varepsilon}$-inert family in $c,l,m,n$. \\Also, $ I_{N,V}(c,l,m,n)$ is small unless $C_0 \gg N^{\frac12 - \varepsilon}$. 
        \item[(b)] (Oscillatory) If $N\abs{V} \gg p^{\varepsilon}$, then $ I_{N,V}(c,l,m,n)$ is small unless $p^2m - n >0$, and $N\abs{V} \asymp  \frac{\sqrt{M_0N}}{C_0} \asymp \frac{\sqrt{N_0N}}{pC_0} $.
         \\If  $p^2m - n >0$, and $N\abs{V} \asymp  \frac{\sqrt{M_0N}}{C_0} \asymp \frac{\sqrt{N_0N}}{pC_0} $,
        \begin{equation}
            \label{eq:Invclmnosc}
            I_{N,V}(c,l,m,n)=\frac{N}{(N\abs{V})^{\frac32}}\cdot e\pb{\frac{2\sqrt{l(p^2m-n)}}{c}} \cdot W_{T}(c,l,m,n)+ O(p^{-A}).
        \end{equation}
         Here, $T= (C_0,L_0,M_0,N_0)$
        and $W_{T}(c,l,m,n)$ is a $p^{\varepsilon}$-inert family in $c,l,m,n$. $A$ can be chosen to be arbitrarily large. 
         
    \end{itemize}
    Note that, in both cases we have the trivial bound, $I_{N,V} (c,l,m,n) \ll Np^{\varepsilon}$.
\end{myprop}

\begin{proof}
\begin{itemize}
    \item [(a)]
            We have $N\abs{V} \ll p^{\varepsilon}$. Now, using Lemma \ref{lem:wtildebehavior}(b) (resp. Lemma \ref{lem:wntildebehavior}(b)),  if $\frac{\sqrt{M_0N}}{C_0} \gg p^{\varepsilon}$  (resp., $\frac{\sqrt{N_0N}}{pC_0} \gg p^{\varepsilon}$), then $I_{N,V}(c,l,m,n)$ is small since 
             $N\abs{V} \ll \frac{\sqrt{N_0N}}{pC_0}$ (resp. $N\abs{V} \ll \frac{\sqrt{N_0N}}{pC_0} $).
        \\So, we must have  $\frac{\sqrt{M_0N}}{C_0} \ll p^{\varepsilon}$ and $\frac{\sqrt{N_0N}}{pC_0} \ll p^{\varepsilon}$. Using Lemma \ref{lem:wtildebehavior}(a) and Lemma \ref{lem:wntildebehavior}(a), we get
        \begin{equation}
            \label{eq:Invnonosc1}
            I_{N,V}(c,l,m,n) = N^2\pb{\frac{\sqrt{M_0N}}{C_0}}^{k-1} \pb{\frac{\sqrt{N_0N}}{pC_0}}^{k-1} \int_{V}^{2V} g_c(v)e(-p^2lv)\cdot W_{T}(c,m,n,v)dv .
        \end{equation}
        Using the fact that $g_c(v)$ satisfies the same derivative bounds as a $p^{\varepsilon}$-inert function, the integral is the Fourier transform of a $p^{\varepsilon}$-inert family of functions at $p^2l$. As $p^2l \leq N \ll \frac{1}{\abs{V}}p^{\varepsilon}$, we can use Proposition \ref{prop:FourierInert} to complete the proof. Also, as $M_0\ll p^{\varepsilon}$, $\frac{\sqrt{M_0N}}{C_0} \ll p^{\varepsilon}$ implies $C_0 \gg N^{\frac{1}{2}-\varepsilon}$.
    \item [(b)]
        We have $N\abs{V} \gg p^{\varepsilon}$. Using Lemma \ref{lem:wtildebehavior}(a) and. Lemma \ref{lem:wntildebehavior}(a),  if $\frac{\sqrt{M_0N}}{C_0} \ll p^{\varepsilon}$  or if $\frac{\sqrt{N_0N}}{pC_0} \ll p^{\varepsilon}$, then  $I_{N,V}(c,l,m,n)$ is small, as $N\abs{V} \gg p^{\varepsilon} $. So, we must have  $\frac{\sqrt{M_0N}}{C_0} \gg p^{\varepsilon}$ and $\frac{\sqrt{N_0N}}{pC_0} \gg p^{\varepsilon}$.
         \\Now, using Lemma \ref{lem:wtildebehavior}(b) (resp. Lemma \ref{lem:wntildebehavior}(b)) the integral then is small unless $N\abs{V} \asymp  \frac{\sqrt{M_0N}}{C_0} \asymp \frac{\sqrt{N_0N}}{pC_0} $. If $N\abs{V} \asymp  \frac{\sqrt{M_0N}}{C_0} \asymp \frac{\sqrt{N_0N}}{pC_0} $, we have (for $A>0$ arbitrarily large) 
        \begin{equation}
            \label{eq:Invosc1}
            I_{N,V}(c,l,m,n) = \frac{Np{C_0}^2}{\sqrt{M_0N_0}}\int_{V}^{2V} g_c(v)e\pb{\frac{n}{p^2c^2v} - \frac{m}{c^2v} - p^2lv} W_{T}(c,m,n,v)dv  + O(p^{-A}).
        \end{equation}
            We can then use Proposition \ref{prop:stationeryphaseinert}(a) to conclude that the integral is small if $p^2m - n \leq 0$. If $p^2m-n >0$, we can locate the stationary point and complete the proof using Proposition \ref{prop:stationeryphaseinert}(b).
\end{itemize}
\end{proof}

\begin{myremark}
    \label{rem:IngT}
    In \eqref{eq:KuznetsovKdyadic}, we actually need to work with the product $I_{N,V}(c,l,m,n) \cdot g_T(l,m,n,c)$. However since $g_T(l,m,n,c)$ satisfies the same derivative bounds in $l,m,n,c$ as a $1$-inert family of functions, we can use the results proven in Proposition \ref{prop:Invlcmnbehavior} for the product $I_{N,V}(c,l,m,n) \cdot g_T(l,m,n,c)$ also. 
\end{myremark}

\section{Spectral Analysis}
\label{sec:SpectralAnalysis}
We continue with our proof of Lemma \ref{lem:E0bound}. One simplification is immediate.

If $\chi$ is trivial, then $\tau(\chi) = -1$ and we can trivially bound all the terms in \eqref{eq:E_0dyadic} (we can use the trivial bound $I_{N,V}(c,l,m,n) \ll N p^{\varepsilon}$) to get $E_{C_0,L_0,M_0,N_0,V} = O(Np^{\varepsilon})$. As the number of dyadic terms is also $O(p^{\varepsilon})$, we get the required bound for $E_0$. So, it suffices to work with $\chi$ non-trivial in \eqref{eq:E_0dyadic}.

Recall $\mathcal{K}_{C_0,L_0,M_0,N_0,V}$, as defined in \eqref{eq:KuznetsovKdyadic}. Comparing this with Proposition \ref{prop:Kuznetsov}, we see that the parameters $(q,m,n,\chi,\Phi(\cdot))$ in \eqref{eq:Kuznetsov} take the form $(p,n-p^2m,-pl,\chi^2, I_{N,V}(\cdot))$ in our case. Thus from \eqref{eq:KuznetsovKdyadic} and \eqref{eq:Kuznetsov}, we have

\begin{equation}
\label{eq:BK}    
    \mathcal{K}_{C_0,L_0,M_0,N_0,V} = \mathcal{K}_{\text{Maass}}+\mathcal{K}_{\text{hol}} +\mathcal{K}_{\text{Eis}},
\end{equation}
where

\begin{equation}
    \label{eq:KMaass0}
    \mathcal{K}_{\text{Maass}} =\sum_{t_j} \mathcal{L}^{\pm}\Phi(t_j) \sum_{r_1r_2=p} \sum_{\pi \in \mathcal{H}_{it}(r_2,\chi^2)} \frac{4\pi\epsilon_{\pi}}{V(p)\mathscr{L}_\pi^*(1)}\sum_{\delta \mid r_1} \ov{\lambda}^{(\delta)}_\pi(\abs{n -p^2m})\ov{\lambda}^{(\delta)}_\pi(\abs{-pl}).
\end{equation}

 $\mathcal{K}_{\text{hol}}$ and $\mathcal{K}_{\text{Eis}}$ are defined similarly as in Proposition \ref{prop:Kuznetsov}.

We can now rewrite \eqref{eq:E_0dyadic}  as 

\begin{equation}
\label{eq:E_0dyadicmod}
    E_{C_0,L_0,M_0,N_0,V}=\mathcal{M}_{\text{Maass}}+\mathcal{M}_{\text{hol}}+\mathcal{M}_{\text{Eis}},
\end{equation} 

where

\begin{equation}
    \label{eq:MMaass0}
     \mathcal{M}_{\text{Maass}} =  \sum_{\substack{l \asymp L_0,m \asymp M_0,\cdots \\ \text{gcd}(ln,p)=1}} \frac{\lambda_f(m)\ov{\lambda_f}(n)}{\phi(p)p^2}\sideset{}{^*} \sum_{\chi(p)} \tau(\chi)^3 \chi(n a_\alpha l) \cdot \mathcal{K}_{\text{Maass}},
\end{equation}

is the contribution from the $\mathcal{K}_{\text{Maass}}$ term. Similarly,  $\mathcal{M}_{\text{hol}}$ and $\mathcal{M}_{\text{Eis}}$ are the contributions of the $\mathcal{K}_{\text{Eis}}$ and $\mathcal{K}_{\text{hol}}$ terms respectively.

We note that since the number of dyadic components in \eqref{eq:E_0dyadic} is $O(p^{\varepsilon})$, Lemma \ref{lem:E0bound} follows if we can show that each of these terms, $\mathcal{M}_{\text{Maass}}, \ \mathcal{M}_{\text{Eis}}$ and $\mathcal{M}_{\text{hol}}$ are $O(Np^{\varepsilon})$. This is what we will show.

We use two different expressions for the integral transforms $\mathcal{L}^{\text{hol}}\Phi(\cdot)$ and $\mathcal{L}^{\pm}\Phi(\cdot)$. If $N\abs{V} \ll p^{\varepsilon}$ (non-oscillatory range), we use the ones given in \eqref{eq:LholPhi} and \eqref{eq:L+/-Phi}. If $N\abs{V} \gg p^{\varepsilon}$ (oscillatory range), we use the ones given in  \eqref{eq:LholPhialt} and \eqref{eq:L+/-Phialt}.

We first work with the Maass form term.
\subsection{Maass Form Term Analysis}
\label{subsec:Maassanalysis}

As $r_1r_2=p$  and $\delta \mid r_1$ in \eqref{eq:KMaass0}, the right side of \eqref{eq:KMaass0} can be rewritten as a sum of three terms corresponding to $(r_1,r_2,\delta) = (1,p,1)$, or $(p,1,1)$, or $(p,1,p)$. 

Also, recall that $\lambda_\pi^{(\delta)} (\cdot)$ is defined in \eqref{eq:lambdapidelta}. In particular, $\lambda_\pi^{(1)} (m) = \lambda_\pi (m)$.

We can use this to split \eqref{eq:MMaass0} as - 

\begin{equation}
    \label{eq:MMaassdefnsplit}
    \mathcal{M}_{\text{Maass}} =  \mathcal{M}_{\text{Maass}_0} + \mathcal{M}_{\text{Maass}_1} + \mathcal{M}_{\text{Maass}_2},
\end{equation}

where 
\begin{multline}
    \label{eq:MM0}
    \mathcal{M}_{\text{Maass}_0} =  \sum_{\substack{l \asymp L_0,m \asymp M_0,\cdots \\ \text{gcd}(ln,p)=1}} \frac{\lambda_f(m)\ov{\lambda_f}(n)}{\phi(p)p^2} \sideset{}{^*}\sum_{\chi(p)} \tau(\chi)^3 \chi(n a_\alpha l)   \\ \cdot \sum_{t_j} \mathcal{L}^{\pm}\Phi(t_j)  \sum_{\pi \in \mathcal{H}_{it_j}(p,\chi^2)} \frac{4\pi\epsilon_{\pi}}{V(p)\mathscr{L}_\pi^*(1)} \ov{\lambda}^{(1)}_\pi(\abs{n-p^2m })\ov{\lambda}^{(1)}_\pi(\abs{-pl}),
\end{multline}

is the contribution from $K_{\text{Maass}}$  when $(r_1,r_2,\delta) = (1,p,1)$. The other two are defined similarly; $\mathcal{M}_{\text{Maass}_1}$ corresponds to $(r_1,r_2,\delta) = (p,1,1)$, and $\mathcal{M}_{\text{Maass}_2}$ corresponds to $(r_1,r_2,\delta) = (p,1,p)$.

We note that, for $\mathcal{M}_{\text{Maass}_1}$ and $\mathcal{M}_{\text{Maass}_2}$, as $r_2=1$, we must have that $\chi^2$ is trivial. As $\chi$ itself is non-trivial, $\chi$ must be the unique quadratic character modulo $p$ (Legendre symbol).

We show that each of the three terms in \eqref{eq:MMaassdefnsplit} is $O(Np^{\varepsilon})$. We start with $\mathcal{M}_{\text{Maass}_0}$.

We split the analysis into two cases: non-oscillatory range (when $N\abs{V} \ll p^{\varepsilon}$) and oscillatory range (when $N\abs{V} \gg p^{\varepsilon}$).

\subsubsection{Non-Oscillatory Case}
We begin first by noting some properties of the functions $\Phi(\cdot)$ and $\mathcal{L}^{\pm}\Phi(\cdot)$.
\subsection* {\textbf{Analysis of \texorpdfstring{$\Phi(\cdot)$}{Phi}:}}

We have

\begin{equation}
    \label{eq:Phidefn}
    \Phi(y,\cdot) = \frac{p}{y^2}I_{N,V} \pb{\frac{y}{\sqrt{p}},l,m,n}.
    \end{equation}

The Mellin transform of $\Phi$ is given by,

\begin{equation}
    \label{eq:Phimellin}
    \widetilde{\Phi}(s,\cdot) = \int_0^{\infty} \Phi(y, \cdot)y^{s-1} dy ,
    \end{equation}

As $N\abs{V}\ll p^{\varepsilon}$, Proposition \ref{prop:Invlcmnbehavior}(a) implies that $I_{N,V}(c,l,m,n) = N^2\cdot V \cdot W_{T} (c,l,m,n) $, where $W_{T}$ is $p^{\varepsilon}$-inert in the variables $c,l,m,n$. We also recall that Proposition \ref{prop:Invlcmnbehavior}(a) allows us to restrict to $C_0 \geq N^{\frac12 - \varepsilon}$, as $I_{N,V}(\cdot)$ is small otherwise.

Thus $\widetilde{\Phi}(s+1,\cdot)$ is the Mellin transform of a $p^{\varepsilon}$-inert family of functions at $s-1$. Proposition \ref{prop:MellinInert} then implies,
\begin{equation}
    \label{eq:Phimellin2}
    \widetilde{\Phi}(s+1,\cdot) =(Np) (N\abs{V}) (\sqrt{p}{C_0})^{s-1} W_{T} (s-1,l,m,n).
\end{equation}
Here, $W_{T}(\cdot)$ is $p^{\varepsilon}$-inert in $l,m,n$ and is small when $\abs{\text{Im}(s-1)} \gg p^{\varepsilon}$.
\subsection*{\textbf{Analysis of \texorpdfstring{ $ \mathcal{L}^{\pm}\Phi(\cdot)$}{L+/-Phi}:}}

Recall that $h_{\pm}(s,t)$ was defined in \eqref{eq:h+/-}. We note that, for any $\sigma_0$ fixed and with $d(\frac{\sigma_0}{2}, \mathbb{Z}_{ \leq 0}) \geq \frac{1}{100}$, we have the bound,
\begin{equation}
    \label{eq:boundsforh}
    h_{\pm} (\sigma_0 + iv,t) \ll \pb{1 + \abs{t + \frac{v}{2}}}^{\frac{\sigma_0-1}{2}}\pb{1 + \abs{t - \frac{v}{2}}}^{\frac{\sigma_0-1}{2}}.
\end{equation}
Also recall that $\mathcal{L}^{\pm}\Phi(t)$ was defined in \eqref{eq:L+/-Phi}.
We define, 
\begin{equation}
    \label{eq:Hplusminus}
    H_{\pm} (s,t,l,m,n) = \frac{1}{2\pi i}h_{\pm}(s,t) \widetilde{\Phi}(s+1,\cdot) (4\pi)^{-s}.
    \end{equation}
So, we can rewrite \eqref{eq:L+/-Phi} as
\begin{equation}
    \label{eq:LPhi}
    \mathcal{L}^{\pm}\Phi(t)  = \int_{\Re(s)=\sigma} H_{\pm} (s,t,l,m,n) \lvert (n-p^2m) (-pl) \rvert ^ {-\frac{s}{2}} ds.
\end{equation}
We state the following lemma regarding the behaviour of $\mathcal{L}^{\pm}\Phi(t)$.

\begin{mylemma}
    \label{lem:tLemma}
    Let $H_{\pm} (s,t,l,m,n)$ and $\mathcal{L}^{\pm}\Phi(t)$ be as in \eqref{eq:Hplusminus} and \eqref{eq:LPhi}. Let $A>0$ be arbitrarily large.
    \begin{itemize}
    \item[(a)] If $\abs{t} \gg p^{\varepsilon},$ then 
    \begin{equation}
        \label{eq:LPhioscbound}
        \mathcal{L}^{\pm}\Phi(t) \ll_{A,\varepsilon} (1+ \abs{t})^{-A}(Np)^{-100}.
    \end{equation}
      
    \item[(b)] If $\abs{t} \ll p^{\varepsilon},$ Then for $\sigma = \Re(s) > \frac14$,  
    \begin{equation}
        \label{eq:LPhinonoscbound}
        \abs{H_{\pm}(s,t)} \ll (Np)^{1+\varepsilon} (\sqrt{p}{C_0})^{\sigma-1}(1+\abs{s})^{-A}.
    \end{equation}
    \end{itemize}
\end{mylemma}

\begin{proof}
    \begin{itemize}
        \item [(a)] If we take the contour of integration in \eqref{eq:LPhi}  far to the left, we encounter poles at $\frac{s}{2} \pm it = 0, -1, -2, \cdots$. This means that $\abs{\text{Im}(s)} \asymp \abs{t} \gg (Np)^{\varepsilon}$. Now, by \eqref{eq:Phimellin2}, $\widetilde{\Phi}(\cdot)$ is very small at this height.
        \item [(b)] This follows from repeated integration by parts on \eqref{eq:Phimellin}, and using Proposition \ref{prop:Invlcmnbehavior}(a) and \eqref{eq:boundsforh}.
    \end{itemize}
\end{proof} 

Lemma \ref{lem:tLemma} allows us to essentially restrict $t_j$ to $\abs{t_j} \ll p^{\varepsilon}$.
 \subsection*{\textbf{Mellin Transform of \texorpdfstring{$H_{\pm} (s,t,\cdot)$}{H+/- (s,t,..)}:}}
Consider the functions $H_{+}(\cdot)$ as defined in \eqref{eq:Hplusminus}. By the Mellin inversion theorem, we have for $\Re(u_j) = \sigma_j > 0$, $j=1,2,3$

\begin{equation}
    \label{eq:H+mellininverse}
    H_+(s,t,l,m,n) = \int_{\Re (u_j) = \sigma_j} l^{-u_1} m^{-u_2} n^{-u_3} \widetilde{H}(s,t,u_1,u_2,u_3) \,  du_1 \, du_2  \, du_3,
\end{equation}
where $\widetilde{H}(s,t,u_1,u_2,u_3)$ is the (partial) Mellin transform of $H_+(\cdot)$ with respect to the variables $l,m,n$. More specifically,
\begin{equation}
    \label{eq:H2mellin}
    \widetilde{H}(s,t,u_1,u_2,u_3) = \int_0^\infty\int_0^\infty\int_0^\infty H_+(s,t,l,m,n) l^{u_1-1} m^{u_2-1} n^{u_3-1} \, dl \, dm  \, dn.
\end{equation}

Using the fact that $\widetilde{\Phi}(\cdot)$ is $p^{\varepsilon}$-inert in $l,m,n$, when $\sigma = \Re(s) > 1$ and $\sigma_j = \Re(u_j) >0 $ for $j=1,2,3$, we have

\begin{equation}
    \label{eq:H2bounds}
     \widetilde{H}(s,t,u_1,u_2,u_3) \ll  (\sqrt{p}C_0)^{\sigma-1} (Np)^{1+\varepsilon}  L_0^{\sigma_1} M_0^{\sigma_2} N_0^{\sigma_3} (1+ \abs{s})^{-A} \prod_{j=1}^{3} (1 + \abs{u_j})^{-A} .  
\end{equation}

Using \eqref{eq:LPhi} and \eqref{eq:H+mellininverse}, we get that,

\begin{equation}
    \label{eq:LPhirewritten}
    \mathcal{L}^{+}\Phi(t)  = \int_{\Re(s)=\sigma} \int_{\Re (u_j) = \sigma_j} l^{-u_1} m^{-u_2} n^{-u_3} \widetilde{H}(s,t,u_1,u_2,u_3) \pb{(p^2m-n) pl} ^ {-\frac{s}{2}}  \, ds  \,  du_1 \, du_2  \, du_3.
\end{equation}

\begin{myremark}
    We can repeat these steps for $H_-(\cdot)$, and get an expression analogous to \eqref{eq:LPhirewritten} for $\mathcal{L}^{-}\Phi(t)$. 
\end{myremark}
\subsection*{\textbf{Final Steps:}}
We show the bounds required for $\mathcal{M}_{\text{Maass}_0}$ in the non-oscillatory region, when the sign is $+$. The proof when the sign is $-$ is very similar. 
\\We recall that the sign in \eqref{eq:MM0} depends on 
\begin{equation*}
    \text{sgn}(-pl(n-p^2m)) = \text{sgn}(p^2m-n)), \text{ as } l >0.
\end{equation*}
So, we only consider terms in \eqref{eq:MM0} with $p^2m -n >0$, and $NV \ll p^{\varepsilon}$.

Using \eqref{eq:LPhirewritten}, we can rewrite \eqref{eq:MM0} as,
\begin{multline}
\label{eq:MMaass}
    \mathcal{M}_{\text{Maass}_0}=  \sum_{\substack{l,m,n  \\ \text{gcd}(ln,p)=1}} \frac{\lambda_f(m)\ov{\lambda_f}(n)}{\phi(p)p^2} \sideset{}{^*}\sum_{\chi(p)} \tau(\chi)^3 \chi(n a_\alpha l)  \\ \cdot \sum_{t_j} \sum_{\pi \in \mathcal{H}_{it_j}(p,\chi^2)} \frac{4\pi\epsilon_{\pi}}{V(p)\mathscr{L}_\pi^*(1)} \ov{\lambda}_\pi(p^2m - n)\ov{\lambda}_\pi(pl) \int_{\substack{\Re(s)=\sigma \\ \Re(u_j)=\sigma_j}}  \frac{\widetilde{H}(s,t,u_1,u_2,u_3)}{\pb{(p^2m-n)pl} ^ {\frac{s}{2}} }  \frac{ds \  du_1 \ du_2  \ du_3}{l^{u_1} m^{u_2} n^{u_3}}.
\end{multline}
Rearranging the expression, we have
\begin{multline}
\label{eq:MMaassren}
    \mathcal{M}_{\text{Maass}_0}=  \frac{1}{\phi(p)p^2} \sideset{}{^*}\sum_{\chi(p)}  \cdot \sum_{t_j} \sum_{\pi \in \mathcal{H}_{it_j}(p,\chi^2)}  \int_{\substack{\Re(s)=\sigma \\ \Re(u_j)=\sigma_j}}\frac{\tau(\chi)^3 \chi( a_\alpha) 4\pi\epsilon_{\pi} }{V(p)\mathscr{L}_\pi^*(1)p^{\frac{s}{2}}} \widetilde{H}(s,t,u_1,u_2,u_3) \\ \pb{\sum_{\substack{l  \\ \text{gcd}(l,p)=1}} \frac{\ov{\lambda}_\pi(pl) \chi(l)}{l ^ {\frac{s}{2} + u_1} }} \pb{ \sum_{\substack{m,n  \\ \text{gcd}(n,p)=1}} \frac{\chi(n)\lambda_f(m)\ov{\lambda_f}(n)\ov{\lambda}_\pi(p^2m - n)}{\pb{p^2m-n} ^ {\frac{s}{2}}m^{u_2} n^{u_3}}} \ ds \  du_1 \ du_2  \ du_3.
\end{multline}

We can combine the $l$ terms in \eqref{eq:MMaassren} to form an L-function. In particular, we have
\begin{equation}
    \label{eq:lnonosc}
    \sum_{l} \frac{\chi(l)\ov{\lambda_{\pi}}(pl)}{l^{\frac{s}{2}+u_1}} = \ov{\lambda}_{\pi}(p) \sum_{l} \frac{\chi(l)\ov{\lambda}_{\pi}(l)}{l^{\frac{s}{2}+u_1}} = \ov{\lambda}_{\pi}(p) L(\ov{\pi} \otimes \chi,\tfrac{s}{2}+u_1). 
\end{equation}

We bound \eqref{eq:MMaassren} by taking the absolute value of the integrand. Note that as $\widetilde{H}(\cdot) $ decays rapidly along vertical lines (using \eqref{eq:H2bounds}), we can restrict the integral up to $\abs{\Im(s)} \ll p^{\varepsilon}$ and $\abs{\Im(u_j)} \ll p^{\varepsilon}$. We thus get, for arbitrarily large $A>0$,

\begin{multline}
\label{eq:MMaassrenewed}
    \mathcal{M}_{\text{Maass}_0}\ll \frac{(\sqrt{p}C_0)^{\sigma-1} (Np)^{1+\varepsilon}  L_0^{\sigma_1} M_0^{\sigma_2} N_0^{\sigma_3}}{\phi(p)p^2} \sideset{}{^*}\sum_{\chi(p)}  \cdot \sum_{t_j} \sum_{\pi \in \mathcal{H}_{it_j}(p,\chi^2)} \left \vert \frac{\tau(\chi)^3 \chi(a_\alpha) 4\pi\epsilon_{\pi} \ov{\lambda}_{\pi}(p)}{V(p)\mathscr{L}_\pi^*(1)p^\frac{\sigma}{2}} \right \vert  \\ 
     \int_{\substack{\Re(s)=\sigma \\ \Re(u_j)=\sigma_j}} \left \vert  \sum_{\substack{m,n  \\ \text{gcd}(n,p)=1}} \frac{\chi(n)\lambda_f(m)\ov{\lambda_f}(n)\ov{\lambda}_\pi(p^2m - n)}{\pb{(p^2m-n)} ^ {\frac{s}{2}}m^{u_2} n^{u_3}} \right \vert  \frac{\abs{L(\ov{\pi} \otimes \chi,\tfrac{s}{2}+u_1)}}{ (1+ \abs{s})^{A}  \prod_{j=1}^{3} (1 + \abs{u_j})^{A} } \ ds \  du_1 \ du_2  \ du_3 .  
\end{multline}

Lemma \ref{lem:tLemma} and Proposition  \ref{prop:Invlcmnbehavior}(a) allow us to restrict (up to a small error) to when $\abs{t_j} \ll p^{\varepsilon}$, and $N^{\frac{1}{2}-\varepsilon} \leq C_0 \leq \sqrt{N}$. 

Additionally we have the bounds  $\abs{\tau(\chi)} = \sqrt{p}, \ \abs{V(p)} \asymp p$. Also, as $\pi \in \mathcal{H}_{it_j}(p,\chi^2)$, $\abs{\lambda_{\pi}(p)} \leq 1$ ($p$ divides the level). These, along with \eqref{eq:lnonosc}, and the fact that $\chi(n) = -\chi (p^2m -n)$ give us,
\begin{equation}
\label{eq:MMaass2}
    \mathcal{M}_{\text{Maass}_0}\ll  \ \frac{p^{\frac32} L_0^{\sigma_1} M_0^{\sigma_2} N_0^{\sigma_3} (Np)^{\frac{1+\sigma}{2}}}{p\cdot p^3 \cdot p^{\frac{\sigma}{2}}}\cdot \mathcal{S}, 
    \end{equation}
where
\begin{multline}
    \label{eq:mathcalS}
    \mathcal{S} =p^{\varepsilon} \cdot \sum_{t_j \ll p^{\varepsilon}}\sideset{}{^*}\sum_{\chi(p)}  \sum_{\pi \in \mathcal{H}_{it_j}(p,\chi^2)} \int_{\substack{\Re(s)=\sigma, \ \Im(s) \ll p^{\varepsilon} \\ \Re(u_j)=\sigma_j \ \Im(u_j) \ll p^{\varepsilon}}}  \abs{L(\ov{\pi} \otimes \chi,\tfrac{s}{2}+u_1)} \\ \cdot \left \vert \sum_{m \asymp M_0}\frac{\lambda_f(m)}{m^{u_2}} \sum_{n \asymp N_0} \frac{\ov{\lambda}_f(n)\ov{\lambda}_\pi(p^2m - n)\chi(p^2m - n)}{ (p^2m-n)^{\frac{s}{2}}n^{u_3}} \right \vert    \ ds \  du_1 \ du_2  \ du_3 . 
\end{multline}

Using the Cauchy-Schwarz inequality on $\mathcal{S}$ we get
\begin{equation}
    \label{eq:boundsonmathcals}
    \abs{\mathcal{S}} \leq (\mathcal{S}_1)^{\frac12} \cdot (\mathcal{S}_2)^{\frac12},
\end{equation}
where
\begin{equation}
    \label{eq:mathcalS1}
    \mathcal{S}_1 = p^{\varepsilon} \cdot \sum_{t_j \ll p^{\varepsilon}} \sideset{}{^*} \sum_{\chi(p)}  \sum_{\pi \in \mathcal{H}_{it_j}(p,\chi^2)}
    \int_{\substack{\Re(s)=\sigma, \ \Im(s) \ll p^{\varepsilon} \\ \Re(u_j)=\sigma_j \ \Im(u_j) \ll p^{\varepsilon}}}  \left \vert  L(\ov{\pi} \otimes \chi,\tfrac{s}{2}+u_1) \right \vert ^2 \ ds \  du_1 \ du_2  \ du_3    ,
\end{equation}
 and
\begin{multline}
     \label{eq:mathcalS2}
     \mathcal{S}_2 = p^{\varepsilon} \cdot \sum_{t_j \ll p^{\varepsilon}}\sideset{}{^*}\sum_{\chi(p)}  \sum_{\pi \in \mathcal{H}_{it_j}(p,\chi^2)} \\ \cdot \int_{\substack{\Re(s)=\sigma, \ \Im(s) \ll p^{\varepsilon} \\ \Re(u_j)=\sigma_j \ \Im(u_j) \ll p^{\varepsilon}}} \left \vert \sum_{m \asymp M_0}\frac{\lambda_f(m)}{m^{u_2}} \sum_{n \asymp N_0} \frac{\ov{\lambda}_f(n)\ov{\lambda}_\pi(p^2m - n)\chi(p^2m - n)}{ (p^2m-n)^{\frac{s}{2}}n^{u_3}} \right \vert ^2 ds \  du_1 \ du_2  \ du_3 .
\end{multline}

Let $\Im(s) = v$, and $\Im(u_1) = v_1$, and $v_0 = \tfrac{v+2v_1}{2}$. We now choose $\sigma = \frac12, \ \sigma_1 = \sigma_3 = \frac{1}{4} + \frac{\varepsilon}{2}, \ \sigma_2 = \frac{\varepsilon}{2}$, and use the spectral large sieve inequality to get bounds on \eqref{eq:mathcalS1} and \eqref{eq:mathcalS2}.

We note that the map $(\pi,\chi) \rightarrow \ov{\pi} \otimes \chi$ is an at most two-to-one map, and $\ov{\pi} \otimes \chi  \in \mathcal{H}_{it_j}(p^2,1)$. 

Using Proposition \ref{prop:approxfe} (with $X=1$) in \eqref{eq:mathcalS1}, we get

\begin{multline}
    \label{eq:mathcalS1bound0}
    \mathcal{S}_1 \ll \sum_{t_j \ll p^{\varepsilon}} \sum_{\phi \in \mathcal{H}_{itj}(p^2,1)}\int_{\substack{v\ll p^{\varepsilon} \\ v_1 \ll p^{\varepsilon}}} \abs{L(\phi,\tfrac{1+\varepsilon}{2} + iv_0 }^2 \ dv \  dv_1 \\ \ll \int_{\substack{v\ll p^{\varepsilon} \\ v_1 \ll p^{\varepsilon}}} \sum_{t_j \ll p^{\varepsilon}} \sum_{\phi \in \mathcal{H}_{it_j}(p^2,1)}  \pb{\left \vert \sum_{n \ll p^{1+\varepsilon}}\frac{\lambda_\phi (n)}{n^{\tfrac{1+\varepsilon}{2} + iv_0}} V\pb{\frac{n}{p}}\right \vert ^2 + \left \vert \sum_{n \ll p^{1+\varepsilon}}\frac{\ov{\lambda_\phi} (n)}{n^{\tfrac{1-\varepsilon}{2} - iv_0}} V\pb{\frac{n}{p}}\right \vert ^2 } \ dv \  dv_1. 
\end{multline}
    
Using Proposition \ref{prop:spectrallargesieve} twice in \eqref{eq:mathcalS1bound0}, we get that
\begin{multline}
    \label{eq:mathcalS1bound}
    \mathcal{S}_1 \ll \pb{p^{2+\varepsilon} + p^{1+\varepsilon}}^{1+\varepsilon} \cdot \left (\sum_{n \leq p^{1+\varepsilon}} \frac{1}{n^{1+\varepsilon}} \pb{\int_{\substack{v\ll p^{\varepsilon} \\ v_1 \ll p^{\varepsilon}}} \abs*{n^{-iv_0} V\pb{\frac{n}{p}}}^2 \ dv \ dv_1} \right. \\+ \left.\sum_{n \leq p^{1+\varepsilon}} \frac{1}{n^{1-\varepsilon}} \pb{\int_{\substack{v\ll p^{\varepsilon} \\ v_1 \ll p^{\varepsilon}}} \abs*{n^{-iv_0} V\pb{\frac{n}{p}}}^2 \ dv \ dv_1} \right )  \ll p^{2+\varepsilon} . 
\end{multline}
    
For the bound on $\mathcal{S}_2$ we state and prove the following lemma. 

\begin{mylemma}
\label{lem:S2}
    For any $\varepsilon>0$,
    \begin{equation}
     \label{eq:s2lemma}   
        \mathcal{S}_2 \ll p^{2+\varepsilon}.
    \end{equation}
\end{mylemma}

Notice that assuming Lemma \ref{lem:S2}, we can complete the proof pretty easily. Using \eqref{eq:boundsonmathcals},\eqref{eq:mathcalS1bound} and \eqref{eq:s2lemma} in \eqref{eq:MMaass2}, we have (with $\sigma = \frac12, \ \sigma_1 = \sigma_3  = \frac14+\frac{\varepsilon}{2}, \sigma_2 = \frac{\varepsilon}{2}$),
\begin{equation}
\label{MMaass5}
    \mathcal{M}_{\text{Maass}_0}\ll \frac{p^{\frac32} (L_0M_0N_0)^{\frac{\varepsilon}{2}}(Np)^{1+\frac{\varepsilon}{2}}}{p\cdot p^3 \cdot p^{\frac{1}{2}}}  \cdot p^{2+\varepsilon} \ll Np^{\varepsilon}.
\end{equation}

\subsubsection*{Proof of Lemma \ref{lem:S2}}

We first define,
\begin{equation}
    \label{eq:S2N0m}
    S_{2,N_0}(m) = \sum_{n \asymp N_0} \frac{\ov{\lambda}_f(n)\ov{\lambda}_\pi(p^2m - n)\chi(p^2m - n)}{ (p^2m-n)^{\frac{s}{2}}n^{u_3}}.
\end{equation}
Now, consider the $m$-sum in  \eqref{eq:mathcalS2} (with $\sigma_2 =\frac{\varepsilon}{2}$). We recall that $M_0 \ll p^{\varepsilon}$, and $N_0 \ll p^{2+\varepsilon}$. Using the Cauchy-Schwarz inequality here, we see that 
\begin{equation}
    \label{eq:msummathcals2}
     \left \vert \sum_{m \asymp M_0}\frac{\lambda_f(m)}{m^{u_2}} S_{2,N_0}(m)\right \vert ^2 \leq  \sum_{m \asymp M_0} \left \vert \frac{\lambda_f(m)}{m^{\frac{\varepsilon}{2}}} \right \vert ^2 \cdot \sum_{m \asymp M_0} \abs{ S_{2,N_0}(m)}^2 \ll p^{\varepsilon}\sum_{m \asymp M_0} \abs{ S_{2,N_0}(m)}^2.
\end{equation}
 
Now,changing variables $n \rightarrow p^2m-n$, we can rewrite \eqref{eq:S2N0m} as, 
\begin{equation}
    \label{eq:S2N0mnew}
    S_{2,N_0}(m) = \sum_{n \ll p^{2+\varepsilon}} \frac{\ov{\lambda}_f(p^2m-n)}{ (p^2m-n)^{u_3}}\cdot \frac{\ov{\lambda}_{\pi\otimes \chi}(n)}{n^{\frac{s}{2}}}.
\end{equation}

Now, we can use \eqref{eq:msummathcals2}, and \eqref{eq:S2N0mnew} in \eqref{eq:mathcalS2}, to get

\begin{equation}
     \label{eq:mathcalS2bound0}
     \mathcal{S}_2 \ll p^{\varepsilon} \int_{\substack{\Re(s)=\sigma \\ \Re(u_j)=\sigma_j}} \sum_{m \asymp M_0}\sum_{t_j \ll p^{\varepsilon}} \sum_{\phi \in \mathcal{H}_{itj}(p^2,1)}  \left \vert \sum_{n \ll p^{2+\varepsilon}} \frac{\ov{\lambda}_f(p^2m-n)}{ (p^2m-n)^{u_3}}\cdot \frac{\lambda_{\phi}(n)}{n^{\frac{s}{2}}}\right \vert ^2 \ ds \ du_1 \ du_2 \ du_3.
\end{equation}
Using Proposition \ref{prop:spectrallargesieve} in \eqref{eq:mathcalS2bound0}, we get 
\begin{equation}
    \label{eq:mathcalS2boundren}
    \mathcal{S}_2 \ll \int_{\substack{\Re(s)=\sigma \\ \Re(u_j)=\sigma_j}} \pb{p^{2+\varepsilon} + p^{2+\varepsilon}}^{1+\varepsilon} \sum_{m \asymp M_0} \sum_{n \ll p^{2+\varepsilon}}\left \vert  \frac{\ov{\lambda}_f(p^2m-n)}{ (p^2m-n)^{u_3}}\cdot \frac{1}{n^{\frac{s}{2}}}\right \vert ^2 \ ds \ du_1 \ du_2 \ du_3.
\end{equation}

Using the fact that the integrals can be restricted up to $\abs{\Im(s)} \ll p^{\varepsilon}$ and $\abs{\Im(u_j)} \ll p^{\varepsilon}$, and that $\sigma=\tfrac12$, and $\sigma_3=\tfrac14 + \tfrac{\varepsilon}{2}$, we can simplify \eqref{eq:mathcalS2boundren} as,

\begin{equation}
    \label{eq:mathcalS2bound}
    \mathcal{S}_2 \ll  \pb{p^{2+\varepsilon} + p^{2+\varepsilon}}^{1+\varepsilon} \sum_{m \asymp M_0} \sum_{n \ll p^{2+\varepsilon}}\left \vert  \frac{\ov{\lambda}_f(p^2m-n)}{ (p^2m-n)^{\frac{1}{4}+\frac{\varepsilon}{2}}}\cdot \frac{1}{n^{\frac{1}{4}}}\right \vert ^2 .
\end{equation}

Let $c_f(m,n) \coloneqq \frac{\abs{\lambda_f(p^2m-n)}^2}{(p^2m-n)^{\frac12+\varepsilon}n^{\frac12}}$. It suffices to show that $\sum_{m \asymp M_0}\sum_{n \ll p^{\varepsilon}} c_f(m,n) \ll p^{\varepsilon}.$

We separate the $n$-sum into two parts depending on if $n \geq \frac{p^2m}{2}$ or not.

In the former case, we can utilize the Rankin-Selberg bound for the square mean of the coefficients $\lambda_f(\cdot)$ to get 
\begin{multline}
\label{eq:cf1}
 \sum_{m \asymp M_0}\sum_{\frac{p^2m}{2} \leq n \ll p^{\varepsilon}} c_f(m,n) \ll  \sum_{m \asymp M_0}\sum_{\frac{p^2m}{2} \leq n \ll p^{2+\varepsilon}}  \frac{\abs{\lambda_f(p^2m-n)}^2}{(p^2m-n)^{\frac12+\varepsilon}n^{\frac12}} \\ \ll \sum_{m \asymp M_0} \sum_{n \ll p^{2+\varepsilon}} \frac{\abs{\lambda_f(n)}^2}{n^{1+\varepsilon}} \ll M_0 \cdot p^{\varepsilon} \ll p^{\varepsilon}.
\end{multline}

The latter case (when $n < \frac{p^2m}{2}$), requires more work. We split the $n$-sum into dyadic segments of size, $n \asymp N_1$, $N_1 < p^2m$. 

\begin{multline}
\label{eq:cf2}
 \sum_{m \asymp M_0}\sum_{n < \frac{p^2m}{2}} c_f(m,n) \ll  \sum_{m \asymp M_0}\sum_{N_1 \text{ dyadic}}\sum_{n \asymp N_1  }  \frac{\abs{\lambda_f(p^2m-n)}^2}{(p^2m-n)^{\frac12+\varepsilon}n^{\frac12}} \\ \ll \sum_{\substack{m \asymp M_0 \\ N_1 \text{ dyadic}}} \frac{1}{p^{1+\varepsilon}N_1^{\frac12}}\sum_{n \asymp N_1 }  \abs{\lambda_f(p^2m-n)}^2 \ll \sum_{\substack{m \asymp M_0 \\ N_1 \text{ dyadic}}}  \frac{1}{pN_1^{\frac12}}\pb{\sum_{n \asymp N_1 }  \abs{\lambda_f(p^2m-n)}^4}^{\frac12} \cdot N_1^{\frac12}.
\end{multline}

Here the last inequality follows from an application of Cauchy's inequality. Using \eqref{eq:fouthmomtbound}, we can now get

\begin{equation}
    \label{eq:cf3}
    \sum_{n \asymp N_1 }  \abs{\lambda_f(p^2m-n)}^4 \ll \sum_{n \ll p^2m} \abs{\lambda_f(n)}^4 \ll (p^2m)^{1+\varepsilon}.
\end{equation}

Using \eqref{eq:cf3} in \eqref{eq:cf2}, we get that,
\begin{equation}
\label{eq:cf4}
     \sum_{m \asymp M_0}\sum_{n < \frac{p^2m}{2}} c_f(m,n) \ll  \sum_{m \asymp M_0}\sum_{N_1 \text{ dyadic}}\sum_{n \asymp N_1  }  \frac{\abs{\lambda_f(p^2m-n)}^2}{(p^2m-n)^{\frac12}n^{\frac12+\varepsilon}} \ll p^{\varepsilon}.
\end{equation}

Now, \eqref{eq:cf1} and \eqref{eq:cf4} jointly imply that $\sum_{m \asymp M_0}\sum_{n \ll p^{\varepsilon}} c_f(m,n) \ll p^{\varepsilon}.$ This completes the proof of the lemma.

\subsubsection{Oscillatory Case}
In this range, $NV \gg p^{\varepsilon}$. The analysis is similar to the non-oscillatory case, although we use the Bessel integral form for the Bruggeman-Kuznetsov formula. Once again, we begin by noting some properties of the functions $\Phi(\cdot)$ and $\mathcal{L}^{\pm}\Phi(\cdot)$.
\subsection*{\textbf{Analysis of \texorpdfstring{$\Phi(\cdot)$}{Phi()}:}}
Recall $ \Phi(y,\cdot)$, as defined in \eqref{eq:Phidefn}.  As $N\abs{V} \gg p^{\varepsilon}$, Proposition \ref{prop:Invlcmnbehavior}(b) implies that $I_{N,V} (c,l,m,n)$ is small unless $p^2m -n > 0$ \textbf{and} $NV \asymp \frac{\sqrt{M_0N}}{C_0} \asymp \frac{\sqrt{N_0N}}{pC_0}$, in which case we have (restating \eqref{eq:Invclmnosc})
\begin{equation}
    I_{N,V} (c,l,m,n)= \frac{N}{(NV)^{\frac32}} e \pb{\frac{2\sqrt{l(p^2m-n)}}{c}} \cdot W_{T}(l,m,n,c) + O(p^{-A}) .
\end{equation}
We recall that the sign in \eqref{eq:MM0} is $\text{sgn}((n-p^2m)(-pl)) = \text{sgn}(p^2m-n)$. Now, as  $I_{N,V} (c,l,m,n)$ is small if $p^2m - n <0$, we only need to consider the case when the sign is $+$. 
\subsection*{\textbf{Analysis of \texorpdfstring{$\mathcal{L}^{+}\Phi(\cdot)$}{L+Phi() }:} }

Recall that  $\mathcal{L}^{+}\Phi(\cdot)$ is defined as in \eqref{eq:L+/-Phialt}.
\\Let $z= 2\sqrt{pl(p^2m-n)}$. As $p^2m -n >0$, this is well defined. Note that $z \asymp \pb{\frac{NN_0}{p}}^{\frac12}$.
\\Now, 
\begin{equation}
    \label{eq:B+alt}
    \frac{J_{2ir(x)}-J_{-2ir}(x)}{\sinh \pb{\pi r}} = \frac{2}{\pi i} \int_{-\infty}^{\infty} \cos(x \cosh y) e\pb{\frac{ry}{\pi}} dy.
\end{equation}
 We can use \eqref{eq:Invclmnosc} and \eqref{eq:B+alt} to rewrite $\mathcal{L}^{+}\Phi(\cdot)$  (up to a small error term) as,
\begin{equation}
    \label{eq:L0+ver1}
    \mathcal{L}^{+} \Phi(t_j)  =   \frac{Np}{\pi(NV)^{\frac32}}\int_0^\infty \int_{-\infty}^{\infty} \cos \pb{\frac{2 \pi z}{x} \cosh y} e\pb{\frac{t_j y}{\pi}} e\pb{\frac{z}{x}}W_{T}(\frac{x}{\sqrt{p}},\cdot) \frac{ dy \ dx}{x^2} . 
\end{equation}
Changing variables, $x \rightarrow \frac{C_0 \sqrt{p}}{x} $, this becomes,
\begin{equation}
    \label{eq:L0+ver2}
    \mathcal{L}^{+} \Phi(t_j) =\frac{N\sqrt{p}}{\pi C_0(NV)^{\frac32}}  \int_0^\infty   \int_{-\infty}^{\infty} \cos \pb{\frac{2 \pi zx}{C_0\sqrt{p}} \cosh y} e\pb{\frac{t_j y}{\pi} + \frac{zx}{C_0 \sqrt{p}}}W_{T'}(x,\cdot) dy \ dx. 
\end{equation}
Here, $W_{T'}$ is another $p^\varepsilon$-inert family, supported on $x \asymp 1$.

We can extend the $x-$integral to all of $\mathbb{R}$  as $W_{T'}(\cdot)$ vanishes in the negative reals. Also,  as $\cos{u} = \frac{1}{2} \pb{e^{iu} + e^{-iu}}$,
we can use Fubini's theorem once to get

\begin{equation}
    \label{eq:L0+ver3}
    \mathcal{L}^{+} \Phi(t_j)  =\frac{N\sqrt{p}}{\pi C_0(NV)^{\frac32}}  \int_{-\infty}^\infty  e\pb{\frac{t_j y}{\pi}}  \widehat{W_{T'}}\pb{-\frac{z}{C_0\sqrt{p}}(1 \pm \cosh y)} dy .
\end{equation}

Here,

\begin{equation}
    \label{eq:Wtprimehat}
     \widehat{W_{T'}}(t)= \int_{-\infty}^{\infty} e\pb{-xt}W_{T'}(x,\cdot) dx .
\end{equation}

Using Proposition \ref{prop:FourierInert}, we can show that the the right hand side in \eqref{eq:Wtprimehat} is small unless $\frac{z}{C_0 \sqrt{p}}(1\pm\cosh{y}) \ll p^\varepsilon$.

So, for the integral to not be small, the sign should be $-$. \\ As, $1-\cosh{y} = -(\frac{y^2}{2!} + \frac{y^4}{4!} + \dots)$, we should have
\begin{equation}
    \label{eq:ybounds}
     y \ll \pb{\frac{C_0\sqrt{p}}{z}}^\frac12 p^{\varepsilon} \asymp \pb{\frac{C_0 p}{\sqrt{NN_0}}}^{\frac12} p^{\varepsilon} \asymp \pb{\frac{1}{NV}}^{\frac12} p^{\varepsilon} .
\end{equation}
 
We can use a Taylor series expansion of $\widehat{W_{T'}}(\cdot)$ , with the leading term being $\widehat{W_{T'}}\pb{-\frac{z}{C_0\sqrt{p}}\frac{y^2}{2}}$, to rewrite \eqref{eq:L0+ver3} as,

\begin{equation}
    \label{eq:L0+ver4}
    \mathcal{L}^{+} \Phi(t_j)  =\frac{N\sqrt{p}}{\pi C_0(NV)^{\frac32}}  \int_{-\infty}^\infty  e\pb{\frac{t_j y}{\pi}}  \widehat{W_{T'}}\pb{\frac{z}{C_0\sqrt{p}}\frac{y^2}{2}, \cdot }  dy  \ + O(p^{-A}).
\end{equation}

We define
\begin{equation}
    \label{eq:Qdefn}
    Q \coloneqq \frac{z}{2C_0 \sqrt{p}} \asymp \frac{\sqrt{NN_0}}{C_0 p} \asymp NV. 
\end{equation}
 Using a change of variables, $y \rightarrow \sqrt{Q} y$ in \eqref{eq:L0+ver4}, we get 

\begin{equation}
    \label{eq:L0+ver5}
    \mathcal{L}^{+} \Phi(t_j)  =\frac{N\sqrt{p}}{\pi C_0(NV)^{\frac32}\sqrt{Q}}  \int_{-\infty}^\infty  e\pb{\frac{t_j y}{\pi\sqrt{Q}}}  \widehat{W_{T'}}\pb{y^2, \cdot }  dy .
\end{equation}

Now, we can use Proposition \ref{prop:FourierInert} once again on \eqref{eq:L0+ver5}, to get

\begin{equation}
    \label{eq:L0+ver6}
    \mathcal{L}^{+} \Phi(t_j) =\frac{N\sqrt{p}}{\pi C_0(NV)^{\frac32}\sqrt{Q}} G\pb{\frac{t_j}{\sqrt{Q}}, \cdot} ,
\end{equation}
 where $G(t,l,m,n)$ satisfies the same derivative bounds as a $p^{\varepsilon}$-inert family in the first variable. It is $p^{\varepsilon}$-inert in the other variables ($l,m,n$). Also, 
 \begin{equation}
     \label{eq:Gbounds}
     \text{$G\pb{\frac{t_j}{\sqrt{Q}},\cdot}$
is small unless  } \abs{t_j} \ll \sqrt{Q} p^{\varepsilon} \asymp (NV)^{\frac12} p^{\varepsilon} .
 \end{equation}
\subsection*{\textbf{Mellin Transform of \texorpdfstring{$G (t,\cdot)$}{G(t,.)}:}}

Similar to \eqref{eq:LPhirewritten}, we take a Mellin transform of $G(\cdot)$ with respect to the variables, $l,m,n$, and use Mellin inversion to rewrite \eqref{eq:L0+ver6} as

\begin{equation}
    \label{eq:L0+final}
    \mathcal{L}^{+} \Phi(t_j)  =\frac{N\sqrt{p}}{\pi C_0(NV)^{2}} \int_{\substack{\Re(u_1)=\sigma_1 \\ \Re (u_2) = \sigma_2 \\ \Re (u_3) = \sigma_3}}l^{-u_1} m^{-u_2} n^{-u_3} \widetilde{G}(t_j,u_1,u_2,u_3) \  du_1 \ du_2  \ du_3,
\end{equation}

where
\begin{equation}
    \label{eq:G2}
    \widetilde{G}(t,u_1,u_2,u_3) = \int_0^\infty\int_0^\infty\int_0^\infty G(t,l,m,n) l^{u_1-1} m^{u_2-1} n^{u_3-1} \ dl \ dm  \ dn,
\end{equation}
and $\sigma_j > 0$, for $j=1,2,3.$

We note here that, $\widetilde{G}(t,\cdot)$ is small unless $\abs{t_j} \ll \sqrt{Q}p^{\varepsilon} \asymp (NV)^{\frac12}p^{\varepsilon}$, because of a similar bound on $G(t,\cdot)$. Also, similar to \eqref{eq:H2bounds}, we have that, for arbitrarily large $A>0$,
 \begin{equation}
    \label{eq:G2bounds}
     \widetilde{G}(t,u_1,u_2,u_3) \ll  p^{\varepsilon} L_0^{\sigma_1} M_0^{\sigma_2} N_0^{\sigma_3} \prod_{j=1}^{3} (1 + \abs{u_j})^{-A} .  
\end{equation}
\subsection*{\textbf{Final Steps:}}
Using \eqref{eq:L0+final}, we can rewrite \eqref{eq:MM0} (up to a small error term) as,

\begin{multline}
\label{eq:M0Maass+ver1}
    \mathcal{M}_{\text{Maass}_0}= \frac{1}{\phi(p)p^2}\frac{N\sqrt{p}}{\pi C_0(NV)^{2}} \sideset{}{^*} \sum_{\chi(p)} \sum_{t_j} \sum_{\pi \in \mathcal{H}_{it_j}(p,\chi^2)} \int_{\Re(u_j)=\sigma_j}  \frac{\tau(\chi)^3 \chi(a_\alpha ) 4\pi\epsilon_{\pi}}{V(p)\mathscr{L}_\pi^*(1)}   \widetilde{G}(t_j,u_1,u_2,u_3)   \\ \cdot \pb{\sum_{\substack{l\\ \text{gcd}(l,p)=1}} \frac{\ov{\lambda}(pl) \chi(l)}{l^{u_1}}} \pb{ \sum_{\substack{m,n \\ \text{gcd}(n,p)=1}} \frac{\chi(n)\lambda_f(m)\ov{\lambda_f}(n)\ov{\lambda}_\pi(p^2m - n)}{m^{u_2} n^{u_3}}} \ du_1 \ du_2  \ du_3  .
\end{multline}

We can again combine the $l$ terms to get (similar to \eqref{eq:lnonosc}),

\begin{equation}
\label{eq:losc}
\sum_{l} \frac{\chi(l) \ov{\lambda}_\pi(pl)}{l^{u_1}} = \ov{\lambda}_\pi(p) \sum_{l} \frac{\chi(l)\ov{\lambda}_\pi(l)}{l^{u_1}}= \ov{\lambda}_\pi(p) {L}(\ov{\pi} \otimes \chi,u_1).
\end{equation}

We bound \eqref{eq:M0Maass+ver1} by taking the absolute value of the integrand.  Note that as $\widetilde{G}(\cdot) $ decays rapidly along vertical lines (using \eqref{eq:G2bounds}), we can restrict the integral up to $\abs{\Im(u_j)} \ll p^{\varepsilon}$. Using \eqref{eq:G2bounds} gives us (for arbitrarily large $A>0$), 

\begin{multline}
\label{eq:M0Maassrenewed}
    \mathcal{M}_{\text{Maass}_0}\ll \frac{L_0^{\sigma_1} M_0^{\sigma_2} N_0^{\sigma_3}}{\phi(p)p^2}\frac{N\sqrt{p}}{\pi C_0(NV)^{2}} \sideset{}{^*} \sum_{\chi(p)} \sum_{t_j} \sum_{\pi \in \mathcal{H}_{it_j}(p,\chi^2)} \left \vert \frac{\tau(\chi)^3 \chi(a_\alpha ) 4\pi\epsilon_{\pi}\ov{\lambda}_\pi(p)}{V(p)\mathscr{L}_\pi^*(1)}   \right \vert \\ \int_{\Re(u_j)=\sigma_j} \left \vert \sum_{\substack{m,n  \\ \text{gcd}(n,p)=1}}  \frac{\chi(n)\lambda_f(m)\ov{\lambda_f}(n)\ov{\lambda}_\pi(p^2m - n)}{m^{u_2} n^{u_3}} \right \vert \frac{\abs{{L}(\ov{\pi} \otimes \chi,u_1)}}{\prod_{j=1}^{3} (1 + \abs{u_j})^{A}} \ du_1 \ du_2  \ du_3.
\end{multline}

We can use \eqref{eq:Gbounds} to restrict (up to a small error) to when $\abs{t_j} \ll \sqrt{Q}p^{\varepsilon}$. Also, as $\pi \in \mathcal{H}_{it_j}(p,\chi^2)$, $\abs{\lambda_{\pi}(p)} \leq 1$ ($p$ divides the level). Additionally we have the bounds  $\abs{\tau(\chi)} = \sqrt{p}, \ \abs{V(p)} \asymp p$. These, along with \eqref{eq:losc}, and the fact that $\chi(n) = -\chi (p^2m -n)$ give us,

\begin{equation}
    \label{eq:M0Maass2}
    \mathcal{M}_{\text{Maass}_0}\ll   \frac{L_0^{\sigma_1} M_0^{\sigma_2} N_0^{\sigma_3} \cdot N\sqrt{p}\cdot p^{\frac32}}{\pi C_0(NV)^{2} \cdot  p^3 \cdot p}  \cdot   \mathcal{S'} . 
    \end{equation}
with
\begin{multline*}
    \label{eq:mathcalSprime}
    \mathcal{S'} =p^{\varepsilon} \cdot \sum_{t_j \ll \sqrt{Q}p^{\varepsilon}}\sideset{}{^*}\sum_{\chi(p)}  \sum_{\pi \in \mathcal{H}_{it_j}(p,\chi^2)}  \int_{\substack{\Re(u_j)=\sigma_j \\ \Im(u_j) \ll p^{\varepsilon}}}  \left \vert L(\ov{\pi} \otimes \chi,u_1)\right \vert \\ \cdot  \left \vert \sum_{m \asymp M_0} \frac{\lambda_f(m)}{m^{u_2}}\sum_{n \asymp N_0} \frac{\ov{\lambda}_f(n) \ov{\lambda}_\pi(p^2m - n)\chi(p^2m - n)}{ n^{u_3}} \right \vert \ du_1 \ du_2  \ du_3.
\end{multline*}

Once again, we can use Cauchy-Schwarz inequality on $\mathcal{S'}$ to get 

\begin{equation}
    \label{eq:mathcalsprimebound}
    \mathcal{S'} \leq (\mathcal{S'\!\!}_{1})^{\frac12} \cdot (\mathcal{S'\!\!}_2)^{\frac12},
\end{equation}
where
\begin{equation}
    \label{eq:Sprime1}
    \mathcal{S'\!\!}_{1}=  p^{\varepsilon} \cdot \sum_{t_j \ll \sqrt{Q} p^{\varepsilon}}\sideset{}{^*}\sum_{\chi(p)}  \sum_{\pi \in \mathcal{H}_{it_j}(p,\chi^2)} \int_{\substack{\Re(u_j)=\sigma_j \\ \Im(u_j) \ll p^{\varepsilon}}}  \left \vert L(\ov{\pi} \otimes \chi,u_1) \right \vert^2  \ du_1 \ du_2  \ du_3 ,
\end{equation}  
and
\begin{multline}
    \label{eq:Sprime2}
     \mathcal{S'\!\!}_{2}= p^{\varepsilon} \cdot \sum_{t_j \ll \sqrt{Q}p^{\varepsilon}}\sideset{}{^*}\sum_{\chi(p)}  \sum_{\pi \in \mathcal{H}_{it_j}(p,\chi^2)} \\ \cdot \int_{\substack{\Re(u_j)=\sigma_j \\ \Im(u_j) \ll p^{\varepsilon}}}  \left \vert \sum_{m \asymp M_0} \frac{\lambda_f(m)}{m^{u_2}}\sum_{n \asymp N_0} \frac{\ov{\lambda}_f(n) \ov{\lambda}_\pi(p^2m - n)\chi(p^2m - n)}{ n^{u_3}} \right \vert ^2 \ du_1 \ du_2  \ du_3 .
\end{multline}

We can now take $\sigma_1 = \sigma_3 = \frac12$, and $\sigma_2 = \varepsilon$, and similar to \eqref{eq:mathcalS1bound} and \eqref{eq:mathcalS2bound}, use the spectral large sieve inequality to bound $\mathcal{S'\!\!}_1$ and $\mathcal{S'\!\!}_2$. The analysis is actually simpler, especially for $\mathcal{S'\!\!}_2$, with the only difference being $t_j \ll \sqrt{Q}p^{\varepsilon}$ here. The analogous bounds then become,
\begin{equation}
    \label{eq:mathcalS1primebound}
    \mathcal{S'\!\!}_1 \ll \pb{Qp^{2+\varepsilon} + p^{1+\varepsilon}}^{1+\varepsilon} \cdot \sum_{n \leq p^{1+\varepsilon}} \frac{1}{n} \cdot p^{\varepsilon} \ll Qp^{2+\varepsilon}, 
\end{equation}
    and
    \begin{equation}
    \label{eq:mathcalS2primebound}
    \mathcal{S'\!\!}_2 \ll \pb{Qp^{2+\varepsilon} + p^{2+\varepsilon}}^{1+\varepsilon} \sum_{m \asymp M_0} \sum_{n \ll p^{2+\varepsilon}} \frac{\abs{\lambda_f(n)}^2}{n} \ll Qp^{2+\varepsilon}.
\end{equation}

Using \eqref{eq:mathcalsprimebound}, \eqref{eq:mathcalS1primebound}, and \eqref{eq:mathcalS2primebound} in \eqref{eq:M0Maass2} we get (with $\sigma_1 = \sigma_3 = \frac12, \sigma_2 = \varepsilon$,)
\begin{equation}
    \label{eq:M0Maass3}
    \mathcal{M}_{\text{Maass}_0}\ll    \frac{L_0^{\tfrac12} M_0^{\varepsilon} N_0^{\tfrac12} \cdot N\sqrt{p}\cdot p^{\frac32}}{\pi C_0(NV)^{2} \cdot  p^3 \cdot p}  \cdot Q p^{2+\varepsilon}.
\end{equation}

Using the fact that $Q \asymp \frac{\sqrt{N N_0}}{pC_0} \asymp NV$, and $L_0 \leq \frac{N}{p^2},$ and $ N_0 \ll p^{2+\varepsilon}$, taking absolute values in $\eqref{eq:M0Maass3}$ we get that,

\begin{equation}
\label{eq:M0Maassfinal}
    \mathcal{M}_{\text{Maass}_0}\ll N p^{\varepsilon}.
\end{equation}

\subsubsection{$\mathcal{M}_{\text{Maass}_1}$ and $\mathcal{M}_{\text{Maass}_2}$}
\label{subsubsec:MMaass12}

We briefly go over the process for bounding $\mathcal{M}_{\text{Maass}_1}$ and $\mathcal{M}_{\text{Maass}_2}$.  Notice that, in this case $\chi = \chi_0$ is fixed, it is the unique quadratic character modulo $p$. 

So compared to \eqref{eq:MM0}, the corresponding expressions for $\mathcal{M}_{\text{Maass}_1}$ and $\mathcal{M}_{\text{Maass}_2}$ do not have a sum over the characters $\chi$ modulo $p$; and the Maass forms are level $1$, instead of $p$. Both of these facts lead to additional cancellations.

However, we can no longer use the bound $\abs{\lambda_{\pi}(p)} \leq 1$. Nevertheless the trivial bound $\abs{\lambda_{\pi}(p)} \leq \sqrt{p}$ is sufficient for our result.

For $\mathcal{M}_{\text{Maass}_1}$, similar to \eqref{eq:MM0}, we have the expression,

\begin{multline}
    \label{eq:MM1}
    \mathcal{M}_{\text{Maass}_1} =  \sum_{\substack{l,m,n \\ \text{gcd}(ln,p)=1}} \frac{\lambda_f(m)\ov{\lambda_f}(n)}{\phi(p)p^2} \tau(\chi_0)^3 \chi_0(n a_\alpha l)   \\ \cdot \sum_{t_j} \mathcal{L}^{\pm}\Phi(t_j)  \sum_{\pi \in \mathcal{H}_{it_j}(1,1)} \frac{4\pi\epsilon_{\pi}}{V(p)\mathscr{L}_\pi^*(1)} \ov{\lambda}_\pi(p^2m - n)\ov{\lambda}_\pi(pl).
\end{multline}

We can now repeat the same steps as in $\mathcal{M}_{\text{Maass}_0}$ case and in fact get an even better bound of $O(Np^{-\frac{1}{2}+\varepsilon})$. (We save a factor of $p^2$ in the corresponding expressions for $\mathcal{S}_1$ and $\mathcal{S}_1'$)

For  $\mathcal{M}_{\text{Maass}_2}$, we have the expression,

\begin{multline}
    \label{eq:MM2}
    \mathcal{M}_{\text{Maass}_2} =  \sum_{\substack{l \asymp L_0,m \asymp M_0,\cdots \\ \text{gcd}(cln,p)=1}} \frac{\lambda_f(m)\ov{\lambda_f}(n)}{\phi(p)p^2} \tau(\chi)^3 \chi(n a_\alpha l)   \\ \cdot \sum_{t_j} \mathcal{L}^{\pm}\Phi(t_j)  \sum_{\pi \in \mathcal{H}_{it_j}(1,1)} \frac{4\pi\epsilon_{\pi}}{V(p)\mathscr{L}_\pi^*(1)} \ov{\lambda}^{(p)}_\pi(p^2m - n)\ov{\lambda}^{(p)}_\pi(pl),
\end{multline}

where (recalling from \eqref{eq:lambdapidelta})

\begin{equation}
    \label{eq:lambdapi(p)}
    \lambda_{\pi}^{(p)}(m) =  \xi_p(1) \lambda_{\pi} (m) + \sqrt{p} \  \xi_p(p) \lambda_{\pi} \pb{\frac{m}{p}}.  
\end{equation}

Here, $\lambda_{\pi} (\frac{m}{p})=0$ if $p \nmid m$.

Notice that, as gcd$(n,p)=1$, $p \nmid (n-p^2m)$. Also, $\xi_p(\cdot) \ll p^{\varepsilon}$. 

Comparing with \eqref{eq:MM1}, we see that there's an additional $\sqrt{p}$ factor. However, the cancellation we get from not having the $\chi$-sum and the level of the Maass forms dropping to $1$ is more than enough to take care of this. We can once again proceed as before and show that $\mathcal{M}_{\text{Maass}_2} \ll Np^{\varepsilon}$.
\begin{myremark}
    We could use non trivial bounds for $\abs{\lambda_{\pi}(p)}$ while analysing $\mathcal{M}_{\text{Maass}_1}$ and $\mathcal{M}_{\text{Maass}_2}$. In that case we would get $\mathcal{M}_{\text{Maass}_1} = O\pb{Np^{-1 + \theta + \varepsilon}}$, and $\mathcal{M}_{\text{Maass}_2} = O\pb{Np^{-\frac12 + \theta + \varepsilon}}$. Here, $\theta$ is current progress towards the Ramanujan Peterson conjecture, so for instance, $\theta \leq \frac{7}{64}$.

\end{myremark}
 
\subsection{Holomorphic Term Analysis}

\label{subseq:holoanalysis}

For the non-oscillatory case in $\mathcal{M}_{\text{hol}}$, we can very easily prove a variant of Lemma \ref{lem:tLemma} (with $t$ replaced by $k$), by choosing
\begin{equation}
    \label{eq:hol}
    h_{\text{hol}} (s,k) = \frac{2^{s-1}}{\pi} \frac{\Gamma \pb{\frac{s+k-1}{2}}}{\Gamma\pb{\frac{k-s+1}{2}}},
\end{equation}

and

\begin{equation}
    \label{eq:Hh}
    H_{\text{hol}} (s,k,l,m,n) = \frac{1}{2\pi i}h_{\text{hol}}(s,t) \widetilde{\Phi}(s+1,\cdot) (4\pi)^{-s}.
\end{equation}

The rest of the steps are identical to the corresponding Maass form case.

For the oscillatory case, instead of \eqref{eq:B+alt}, we use 

\begin{equation}
    \label{eq:Jalt}
    J_{l-1}(x) = \sum_{\pm} \frac{e^{\mp i (l-1)\frac{\pi}{2} }}{\pi} \int_0^{\frac{\pi}{2}} \cos\pb{(l-1) \theta} e^{\pm ix \cos\theta} \, d\theta. 
\end{equation}

We then proceed similarly, the only difference being, we use the power series expansion of $\cos(y)$ instead of $\cosh{y}$ in $\eqref{eq:L0+ver4}$.

\subsection{Eisenstein Term Analysis}
\label{subseq:Eisanalysis}
The kernel function for $\mathcal{M}_{\text{Eis}}$ is the same as $\mathcal{M}_{\text{Maass}}$, and hence, the analysis is pretty similar, except for replacing $\sum_{t_j}(\cdot)$ with $\int_{-\infty}^{\infty} (\cdot) dt$.
The Eisenstein contribution, analogous to \eqref{eq:MMaass0} is

\begin{equation}
    \label{eq:MEis}
     \mathcal{M}_{\text{Eis}} =  \sum_{\substack{l \asymp L_0,m \asymp M_0,\cdots \\ \text{gcd}(cln,p)=1}} \frac{\lambda_f(m)\ov{\lambda_f}(n)}{\phi(p)p^2}\sideset{}{^*} \sum_{\chi(p)} \tau(\chi)^3 \chi(n a_\alpha l) \cdot \mathcal{K}_{\text{Eis}},
\end{equation}

where

\begin{equation}
\label{eq:KEis0}
    \mathcal{K}_{\text{Eis}}= \frac{1}{4\pi} \int_{-\infty}^{\infty} \mathcal{L}^{\pm}\Phi(t) \sum_{r_1r_2=p} \sum_{\pi \in \mathcal{H}_{it,\text{Eis}}(r_2,\chi^2)} \frac{4\pi\epsilon_{\pi}}{V(p)\mathscr{L}_\pi^*(1)}\sum_{\delta \mid r_1} \ov{\lambda}^{(\delta)}_\pi(\abs{n-p^2m})\ov{\lambda}^{(\delta)}_\pi(\abs{-pl}) dt.
\end{equation}

Also,

\begin{equation}
    \label{eq:HitEisdefn}
    \mathcal{H}_{it,\text{Eis}}(r_2,\chi^2) = \{E_{\chi_1,\chi_2}\pb{z,\frac12 + it}, \chi_i \shortmod{q_i}, \text{ for } i=1,2; \ q_1q_2 = r_2, \ \chi_1\ov{\chi_2} \simeq \chi^2 \}, 
\end{equation}

where we use the notation $\psi \simeq \chi$ to mean that $\psi$ and $\chi$ share the same underlying primitive character.

In the spirit of \eqref{eq:MMaassdefnsplit}, as $r_1 r_2 =p$ and $\delta \mid r_1$, we can split \eqref{eq:MEis} as 

\begin{equation}
    \label{eq:MEisdefnsplit}
    \mathcal{M}_{\text{Eis}} =  \mathcal{M}_{\text{Eis}_0} + \mathcal{M}_{\text{Eis}_1} + \mathcal{M}_{\text{Eis}_2},
\end{equation}

by considering separately the cases where $(r_1,r_2,\delta)$ equals $(1,p,1)$, $(p,1,1)$, or $(p,1,p)$. 

Once again, the latter two cases appear only if $\chi$ is the unique quadratic character modulo $p$, and so $\mathcal{H}_{it,\text{Eis}}(r_2,\chi^2) = \{E_{1,1}\pb{z,\frac12 + it} \}$. Bounding these two is easier as there is no sum over $\chi$ anymore.

We focus on $\mathcal{M}_{\text{Eis}_0}$. We then proceed in a similar fashion to the way we did for $\mathcal{M}_{\text{Maass}_0}$, replacing $\sum_{t_j \ll p^\varepsilon}$ with $\int_{t \ll p^\varepsilon}$ where necessary.

In this case, $\mathcal{H}_{it,\text{Eis}}(p,\chi^2) = \{E_{1,\ov{\chi}^2}\pb{z,\frac12 + it}, E_{\chi^2,1}\pb{z,\frac12 + it} \}$, and $\lambda_{\pi} ^{(1)} (n) = \lambda_{\pi} (n)$.

Also,  for $\pi = E_{\chi_1,\chi_2} \pb{z, \frac12 + it}$, we have
\begin{equation}
    \label{eq:lambdaEis}
  \lambda_{\pi} (n) = \lambda_{\chi_1,\chi_2,t}(n) = \chi_2(\text{sgn} (n)) \sum_{ab = \abs{n}}\frac{\chi_1(a)\ov{\chi}_2(b)}{a^{it} b^{-it}}.  
\end{equation}

We note that the corresponding expressions for \eqref{eq:lnonosc} (and \eqref{eq:losc} for the oscillatory case) become

\begin{multline}
\label{eq:lnonoscEis}
    \sum_{l} \frac{\chi(l)\ov{\lambda_{\pi}}(pl)}{l^{\frac{s}{2}+u_1}} = \ov{\lambda_{\pi}}(p) \sum_{l} \frac{\chi(l){\lambda_{\ov{\chi_1},\ov{\chi_2},t}}(l)}{l^{\frac{s}{2}+u_1}} =\ov{\lambda_{\pi}}(p) \sum_{a}\sum_{b} \frac{\ov{\chi_1}(a)\chi_2(b)\chi(ab)}{(ab)^{\frac{s}{2}+u_1}a^{-it}b^{it}} = \ov{\lambda}_{\pi}(p) \sum_{l} \frac{\lambda_{\chi\ov{\chi_1}, \chi\ov{\chi_2},t}(l)}{l^{\frac{s}{2}+u_1}}\\ =\ov{\lambda}_{\pi}(p) {L}(\chi\cdot\ov{\chi_1},\frac{s}{2}+u_1-it){L}(\chi\cdot \chi_2,\frac{s}{2}+u_1 +it) .
\end{multline}

The last expression can be further simplified depending on if $(\chi_1, \chi_2) = (1,\ov{\chi}^2)$, or $(\chi^2,1)$. In each case, we get a product of Dirichlet $L$-functions for the characters $\chi$ and $\ov{\chi}$. Notice that as $\chi$ is non-trivial, these $L$-functions do not have a pole. Also, this product is the same as the twisted $L$-function, $L(\ov{\pi} \otimes \chi,\tfrac{s}{2}+u_1)$.

We also note that when $\pi = E_{\chi_1,\chi_2} \pb{z, \frac12 + it} \in \mathcal{H}_{it,\text{Eis}}(p,\chi^2)$, $\ov{\pi} \otimes \chi \in \mathcal{H}_{it,\text{Eis}}(p^2,1)$. So once again, we can eventually make use of Proposition \ref{prop:spectrallargesieve} for the family $\mathcal{H}_{it,\text{Eis}}(p^2)$.

Proceeding as before, we will get analogous versions of equations \eqref{eq:MMaassrenewed}, \eqref{eq:MMaass2}, \eqref{eq:mathcalS},\eqref{eq:mathcalS1}, \eqref{eq:mathcalS2}, and we can bound $\mathcal{M}_{\text{Eis}_0}$ just as we bounded $\mathcal{M}_{\text{Maass}_0}$. 

The analysis of the terms $ \mathcal{M}_{\text{Eis}_1} $ and $ \mathcal{M}_{\text{Eis}_2}$ is again similar to their Maass form counterparts. In fact, it is even easier because we can use stronger bounds on the size of $\lambda_\pi(p)$.

\section{Remaining Terms}
\label{sec:Remterms}

We complete the proof of Theorem \ref{thm:redthm} by proving Lemma \ref{lem:S4} (bounding the terms introduced by the delta symbol with $c$ divisible by $p$) and  Lemma \ref{lem:E123bound} (bounding the smaller terms introduced after Voronoi summation). We start with the latter.
\subsection{Proof of Lemma \ref{lem:E123bound}}
\label{subsec:Remtermsvoronoi}
Recall the term $E_1$ was defined in \eqref{eq:afterVoronoi}. This can be rewritten as,
\begin{multline}
\label{eq:E1new}
    E_1 =\sum_{\substack{l \leq \frac{N}{p^2}\\\text{gcd}(l,p)=1}} \sum_{m \ll p^{\varepsilon}}\sum_{\substack{n \ll p^{2+\varepsilon} \\ \text{gcd}(n,p^2)=p}} \sum_{\substack{c\\\text{gcd}(c,p)=1}}\frac{\lambda_f(m)\ov{\lambda_f}(n)}{p^2c^2}  I_N(c,l,m,n) \\ \cdot \  \sum_{\substack{a < c \\\text{gcd}(a,c)=1 }}e\pb{\frac{-ap^2l}{c}}e\pb{\frac{-a\ov{m}}{c}}e\pb{\frac{\ov{ap^2}n}{c}},
\end{multline}

where
\begin{equation}
\label{eq:Iintegral}
    I_N(c,l,m,n) = \int_{-\infty}^{\infty} g_c(v)e(-p^2lv)\widetilde{w}_{c,v,N}(m)\widetilde{w}_{pc,-v,N}(n) dv. 
\end{equation}

Here, $\widetilde{w}_{c,v,N}(\cdot)$ is given by \eqref{eq:wmtilde}. Similarly to Lemma \ref{lem:Ibound}, we claim that $I_N(c,l,m,n) \ll Np^{\varepsilon}$.  

The $a$-sum at the end is actually a Kloosterman sum,
\begin{equation}
 \sum_{\substack{a < c \\\text{gcd}(a,c)=1 }}e\pb{\frac{-ap^2l}{c}}e\pb{\frac{-a\ov{m}}{c}}e\pb{\frac{\ov{ap^2}n}{c}}=S(n\ov{p}^2-m,-p^2l;c) = S(\ov{p}(n-p^2m),-pl;c).
\end{equation}

Using the Weil bound, it follows that
\begin{equation}
    \label{eq:E1final}
    E_1 \ll \frac{N}{p^2}\cdot p \cdot \frac{1}{p^2} \cdot Np^{\varepsilon} \cdot \sum_{c \leq \sqrt{N}} c^{-\frac{3}{2}} \ll N^{\frac{3}{4}+\varepsilon}.
\end{equation}
The proofs for the corresponding bounds for $E_2$ and $E_3$ are very similar. Following the same steps, we can, in fact, obtain an even better bound of $O( N^{\frac34}p^{-1+\varepsilon})$.

\subsection{Proof of Lemma  \ref{lem:S4}}
\label{subseq:remtermsDelta}
We use the Voronoi summation twice to prove Lemma \ref{lem:S4}.
\\Using Proposition \ref{prop:Voronoi} on $T_1(a,c,v)$, we have
\begin{equation}
    \label{eq:mVoronoiEp}
T_1(a,c,v) = \frac{1}{c}\sum_{m \geq 1} \lambda_f(m) e\pb{-\frac{\ov{a}m}{c}}\widetilde{w}_{c,v,N}(m).
\end{equation}

Here, 
\begin{equation*}
    \widetilde{w}_{c,v,N}(m) = 2\pi i^k \int_o^\infty J_{k-1}\pb{\frac{4\pi\sqrt{mx}}{c}}e(xv)w_N(x)dx.
\end{equation*}

Using an integration by parts argument similar to the one in Lemma \ref{lem:wtildebehavior}, we can show that the sum on the right in \eqref{eq:mVoronoiEp} is effectively for $m \ll \frac{c^2}{N}$. As $c \leq C = \sqrt{N}$, this is non-trivial only when $c^2 \geq N^{1-\varepsilon}$, in which case $1 \leq m \ll p^\varepsilon $.

$T_2(a,c,v)$ needs more attention. Since $c$ is divisible by $p$, $ e\pb{\frac{a_\alpha l \ov{n}}{p}} e\pb{\frac{-an}{c}}$ is periodic modulo $c$ (whenever $n$ is coprime to $p$), and hence can be written as a finite sum of additive characters. Using this and taking $r = a_\alpha l$, we get

\begin{multline}
 \label{eq:nVoronoiEp}
   T_2(a,c,v) =\sum_{\substack{N \leq n \leq 2N\\\text{ gcd}(n,p)=1}} \ov{\lambda_f}(n) w_N(n) e(-nv) \frac{1}{c}\sum_{t\shortmod{c}} e\pb{\frac{nt}{c}}\sum_{\substack{u\shortmod{c}\\\text{gcd}(u,p)=1}} e\pb{\frac{r\ov{u}}{p}}  e\pb{-\frac{au+tu}{c}}\\ = \frac{1}{c}\sum_{t\shortmod{c}} \sum_{\substack{u\shortmod{c}\\ \text{gcd}(u,p)=1}} e\pb{\frac{r\ov{u}}{p}} e\pb{-\frac{au+tu}{c}}\sum_{\substack{N \leq n \leq 2N\\\text{ gcd}(n,p)=1}} \ov{\lambda_f}(n) e\pb{\frac{nt}{c}} w_N(n) e(-nv) .  
\end{multline}

We want to use the Voronoi summation for the last sum. Once again, because of an integration by parts argument, the length of the summation, after Voronoi summation, will effectively be up to $n \ll \frac{1}{N}\pb{\frac{c}{\text{gcd}(c,t)}}^2 $. As $c \leq \sqrt{N}$, this will have non-trivial contribution only when $\text{gcd}(c,t) \ll p^\varepsilon$, and $c^2 \geq N^{1-\varepsilon}$. We can in fact show that we can restrict to just $\text{gcd}(c,t)=1$ case.

Suppose $\text{gcd}(c,t)=g >1$. As $g \ll p^{\varepsilon}, \ p\mid \frac{c}{g}$. Now, we can write that the contribution of these terms in \eqref{eq:nVoronoiEp} is

\begin{equation}
\label{eq:nVoronoiEpalt}
    T_2'(a,c,v)=\frac{1}{c} \sum_{\substack{g \mid c \\ 1 < g \ll p^{\varepsilon}}} S_{a,c,v}(g), 
\end{equation}

where
\begin{equation}
\label{eq:Sg}
S_{a,c,v}(g) = \sum_{\substack{t \equiv 0\shortmod{c}\\ \text{gcd}(t,c)=g }} \sum_{\substack{N \leq n \leq 2N\\\text{ gcd}(n,p)=1}} \ov{\lambda_f}(n) e\pb{\frac{nt}{c}} w_N(n) e(-nv)\sum_{\substack{u\shortmod{c}\\ \text{gcd}(u,p)=1}}  e\pb{\frac{r\ov{u}}{p}} e\pb{-\frac{au+tu}{c}}.
\end{equation}

Consider the $u$-sum in \eqref{eq:Sg}. If we change variables with $u \rightarrow u + \frac{c}{g}$ in the $u$-sum in \eqref{eq:Sg} (this is well defined as $p \mid \frac{c}{g}$), we notice that 
\begin{equation}
    \label{Sgalt}
    \sum_{\substack{u\shortmod{c}\\ \text{gcd}(u,p)=1}}  e\pb{\frac{r\ov{u}}{p}} e\pb{-\frac{au+tu}{c}} = e\pb{\frac{a}{g}} \sum_{\substack{u\shortmod{c}\\ \text{gcd}(u,p)=1}}  e\pb{\frac{r\ov{u}}{p}} e\pb{-\frac{au+tu}{c}}.
\end{equation}
As gcd$(a,c)=1$, and $g \mid c$, $e\pb{-\frac{a}{g}} \neq 1$. So, the $u$-sum in \eqref{eq:Sg} must be zero. Thus $T_2'(a,c,v) =0$, and it suffices to only consider the case when gcd$(c,t)=1$.

Using Proposition \ref{prop:Voronoi} for that case, we have (up to a small error),
\begin{equation}
\label{eq:nVoronoiEp3}
   T_2(a,c,v)=\frac{1}{c}\sum_{\substack{t\shortmod{c}\\ \text{gcd}(t,c)=1}}  \sum_{\substack{u\shortmod{c}\\ \text{gcd}(u,p)=1}} e\pb{\frac{r\ov{u}}{p}} e\pb{-\frac{au+tu}{c}}\frac{1}{c} \sum_{n\ll p^{\varepsilon}} \ov{\lambda_f}(n) e\pb{-\frac{n\ov{t}}{c}} \widetilde{w}_{c,-v,N}(n). 
\end{equation}
So, using \eqref{eq:mVoronoiEp} and \eqref{eq:nVoronoiEp3} in \eqref{eq:S4defn}, we get that (up to a small error),

\begin{equation}
    \label{eq:EpafterVoronoi}
S_4(N,\alpha) = \sum_{\substack{l \leq N/p^2\\\text{gcd}(l,p)=1}}\sum_{m \ll p^{\varepsilon}} \sum_{n\ll p^{\varepsilon}} \sum_{\substack{c \\ \text{gcd}(c,p)=p}} \frac{\lambda_f(m)\ov{\lambda_f}(n)}{c^3} \tilde{I}_{N}(c,l,m,n) B(c,l,m,n),
\end{equation}

where

\begin{equation}
\label{eq:ItildeNlmnc}
\tilde{I}_{N}(c,l,m,n) = \int_{-\infty}^{\infty} g_c(v)e(-p^2l v) \widetilde{w}_{c,v,N}(m) \widetilde{w}_{c,-v,N}(n) dv, 
\end{equation}

and
\begin{equation}
    \label{eq:B}
     B(c,l,m,n) = \sum_{\substack{a \shortmod{c} \\ \text{gcd}(a,c)=1}}e\pb{\frac{-ap^2l}{c}}   e\pb{-\frac{\ov{a}m}{c}}
    \sum_{\substack{t\shortmod{c}\\ \text{gcd}(t,c)=1}}  \sum_{\substack{u\shortmod{c}\\ \text{gcd}(u,p)=1}} e\pb{\frac{r\ov{u}}{p}} e\pb{-\frac{au+tu}{c}} e\pb{-\frac{n\ov{t}}{c}}. 
\end{equation}
Let $c = k\cdot p$, and as $c \leq \sqrt N \ll p^{\frac{3}{2}+\varepsilon}$, $k \ll p^{\frac12 + \varepsilon}$ and gcd$(k,p)=1$. 
\\ We can then rewrite \eqref{eq:EpafterVoronoi} (up to small error) as

\begin{equation}
   \label{eq:EpafterCRT}
S_4(N,\alpha) =  \sum_{\substack{l \leq N/p^2\\\text{gcd}(l,p)=1}}\sum_{\substack{m,n \\m, n\ll p^{\varepsilon}}} \sum_{k \ll p^{\frac12+\varepsilon}} \frac{\lambda_f(m) \ov{\lambda_f}(n)}{k^3p^3}  \tilde{I}_N(kp,l,m,n) B(kp,l,m,n).
\end{equation}

We complete the proof by stating the following two lemmas. 

\begin{mylemma}
    \label{lem:Ibound}
Let $\varepsilon>0$ and let $N \ll p^{3+\varepsilon}$. Let $l \leq \frac{N}{p^2}$ with $(l,p)=1$, and $m,n \ll p^{\varepsilon}$. We define $\tilde{I}_N(c,l,m,n)$ as in \eqref{eq:ItildeNlmnc}. Now, if in addition $c$ satisfies $N^{1-\varepsilon} \ll c^2 \leq N$, then
\begin{equation}
    \label{eq:Ibound}
     \tilde{I}_N(c,l,m,n) \ll_{\varepsilon} Np^{\varepsilon}. 
\end{equation}
\end{mylemma}

\begin{mylemma}
    \label{lem:Bbound}
    Let $\varepsilon, N, l, m, n$ be same as before, in Lemma \ref{lem:Ibound}. We define $B(c,l,m,n)$ as in \eqref{eq:B}. Now, if in addition, $c=k\cdot p$, with $(k,p)=1$, then
\begin{equation}
    \label{eq:Bbound}
     B(kp,l,m,n) \ll k\sqrt{k}\ p^{2}. 
\end{equation}
\end{mylemma}

We quickly show how Lemmas \ref{lem:Ibound} and \ref{lem:Bbound} imply Lemma \ref{lem:S4}. Using \eqref{eq:Ibound} and \eqref{eq:Bbound} in \eqref{eq:EpafterCRT}, we get (using trivial bounds)

\begin{equation}
 \label{eq:Epfinal}
    S_4(N,\alpha) \ll \sum_{\substack{l \leq N/p^2\\\text{gcd}(l,p)=1}}\sum_{\substack{m,n \\m, n\ll p^{\varepsilon}}} \sum_{k \ll p^{\frac12+\varepsilon}} \frac{\abs{\lambda_f(m) \ov{\lambda_f}(n)}}{k^3p^3} N k^{\frac{3}{2}}p^{2+\varepsilon} \ll N^2 p^{-3-\frac14+\varepsilon} \ll Np^{-\frac14+\varepsilon}.  
\end{equation}

Clearly, \eqref{eq:Epfinal} implies \eqref{eq:S4}.

\subsubsection{Proof of Lemma \ref{lem:Ibound}}
Using Proposition \ref{prop:gcbound}, we can restrict the integral in \eqref{eq:ItildeNlmnc}, up to a small error, to $\abs{v} \ll \frac{p^\varepsilon}{cC}$. Here, $C=\sqrt{N}$. Using  the trivial bounds $\widetilde{w}_{c,v,N}(m) \ll N $, we get

\begin{equation}
    \tilde{I}_{N}(c,l,m,n) \ll N^{2+\varepsilon} \frac{1}{cC}. 
\end{equation}

We have already noted that $c^2 \geq N^{1-\varepsilon} $. Using this, Lemma \ref{lem:Ibound} follows.

\subsubsection{Proof of Lemma \ref{lem:Bbound}}

Using the Chinese Remainder theorem, we can rewrite $B(kp,l,m,n)$ as

\begin{equation}
    B(kp,l,m,n) = B_k \cdot B_p
\end{equation}
where
\begin{equation}
\label{eq:Bk}
    B_k =  \sum_{\substack{a_1 \shortmod{k} \\ \text{gcd}(a_1,k)=1}} e\pb{\frac{-a_1pl}{k}}  e\pb{-\frac{\ov{a_1}m\ov{p}}{k}} 
 \sum_{\substack{t_1\shortmod{k}\\  \text{gcd}(t_1,k)=1}} e\pb{-\frac{n\ov{t_1}\ov{p}}{k}}  \sum_{u_1 \shortmod{k}}  e\pb{-\frac{a_1u_1\ov{p}}{k}}e\pb{-\frac{t_1u_1\ov{p}}{k}}, 
\end{equation}
and
\begin{equation}
    \label{eq:Bp}
   B_p =\sum_{\substack{u_2 \shortmod{p} \\\text{gcd}(u_2,p)=1}}  e\pb{\frac{r\ov{u_2}}{p}} \sum_{\substack{a_2 \shortmod{p} \\\text{gcd}(a_2,p)=1}}e\pb{-\frac{\ov{a_2}m\ov{k}}{p}}  e\pb{-\frac{a_2u_2\ov{k}}{p}} \sum_{\substack{t_2 \shortmod{p} \\\text{gcd}(t_2,p)=1}} e\pb{-\frac{n\ov{t_2}\ov{k}}{p}} e\pb{-\frac{t_2u_2\ov{k}}{p}} .  
\end{equation}

Consider $B_k$ first. Using orthogonality of characters modulo $k$ for the $u_1$ sum in \eqref{eq:Bk}, and rearranging terms, we get 
\begin{equation}
    \label{eq:Bk2}
      B_k= k \sum_{\substack{a_1 \shortmod{k} \\ \text{gcd}(a_1,k)=1}} e\pb{\frac{-a_1pl}{k}}e\pb{\frac{\ov{a_1}n\ov{p}}{k}} e\pb{-\frac{\ov{a_1}m\ov{p}}{k}}. 
\end{equation}

We note that the $a_1$-sum is in fact a Kloosterman sum. 
\begin{equation}
    \label{eq:Bk3}
    B_k = k \cdot S(-pl,(n-m)\ov{p};k) = k \cdot S(-l,(n-m);k).
\end{equation}
The Weil bound immediately gives,
\begin{equation}
    \label{eq:Bkfinal}
    B_k \ll k\sqrt{k}.
\end{equation}

Consider $B_p$ now. The $t_2$ and the $u_2$ sums in \eqref{eq:Bp} both define Kloosterman sums. This gives us 
\begin{equation}
    \label{eq:Bp2}
     B_p  =  \sum_{\substack{u_2\shortmod{p}\\ \text{gcd}(u_2,p)=1}} e\pb{\frac{r\ov{u_2}}{p}} S(u_2\ov{k},m\ov{k};p) S(u_2\ov{k}, n\ov{k};p). 
\end{equation}

Trivially bounding the $u_2$ sum, and using the Weil bound twice gives us

\begin{equation}
    \label{eq:Bpfinal}
    B_p \ll p^2.
\end{equation}

Now, \eqref{eq:Bkfinal} and \eqref{eq:Bpfinal} imply \eqref{eq:Bbound}. This completes the proof of Lemma \ref{lem:Bbound}.

\section{Application to Non-vanishing }
\label{sec:nonvanishing}
We use the upper bound obtained in Theorem \ref{thm:MainThm} to prove Theorem \ref{thm:nonv}.
We first state a lemma regarding the first moment for the family of $L$-functions we have been working with so far.

\begin{mylemma}
\label{lem:first moment}
Let $p,q,f,\alpha$ be as before. Then for $\varepsilon > 0$, 
\begin{equation}
\label{eq:first moment}
  \sum_{\psi \shortmod{p^2}} L\pb{\tfrac{1}{2}, f \otimes \pb{\alpha\cdot \psi}}     =  p^2  + O(p^{\frac{7}{4}+\varepsilon}).
\end{equation}
\end{mylemma}
Assume Lemma \ref{lem:first moment} for now. Using Cauchy-Schwarz inequality, we have

\begin{multline*}
\left \vert \sum_{\psi \shortmod{p^2}} L\pb{\tfrac{1}{2}, f \otimes \pb{\alpha\cdot \psi}} \right \vert ^2 =\left \vert \sum_{\psi \shortmod{p^2}} L\pb{\tfrac{1}{2}, f \otimes \pb{\alpha\cdot \psi}} \cdot \delta \pb{L\pb{\tfrac{1}{2}, f \otimes \pb{\alpha\cdot \psi}}  \neq 0}\right \vert ^2
\\ \leq \sum_{\psi \shortmod{p^2}} \left \vert  L\pb{\tfrac{1}{2}, f \otimes \pb{\alpha\cdot \psi}} \right \vert ^2 \cdot  \sum_{\psi \shortmod{p^2}} \left \vert \delta \pb{L\pb{\tfrac{1}{2}, f \otimes \pb{\alpha\cdot \psi}}  \neq 0}  \right \vert ^2  \\ =  \sum_{\psi \shortmod{p^2}} \left \vert  L\pb{\tfrac{1}{2}, f \otimes \pb{\alpha\cdot \psi}} \right \vert ^2 \cdot \sum_{\substack{\psi \shortmod{p^2}\\L\pb{\tfrac{1}{2}, f \otimes \pb{\alpha\cdot \psi}}  \neq 0 }}1.
\end{multline*}
     
This implies, 

\begin{equation}
    \label{eq:nvfinal}
     \# \{ \psi \shortmod {p^2} ; \ L\pb{\tfrac{1}{2}, f \otimes \pb{\alpha\cdot \psi}}  \neq 0\}= \sum_{\substack{\psi \shortmod{p^2}\\L\pb{\tfrac{1}{2}, f \otimes \pb{\alpha\cdot \psi}}  \neq 0 }}1 \geq \ \ \frac{\left \vert \sum_{\psi \shortmod{p^2}} L\pb{\tfrac{1}{2}, f \otimes \pb{\alpha\cdot \psi}} \right \vert ^2}{\sum_{\psi \shortmod{p^2}} \left \vert  L\pb{\tfrac{1}{2}, f \otimes \pb{\alpha\cdot \psi}} \right \vert ^2}.
\end{equation}

Theorem \ref{thm:nonv} now follows using the corresponding bounds in Lemma \ref{lem:first moment}  and Theorem \ref{thm:MainThm} in \eqref{eq:nvfinal}.

\subsection{Proof of Lemma \ref{lem:first moment}}
\label{subsec:firstmombound}
Using the approximate functional equation in Proposition \ref{prop:approxfe}, we have

\begin{equation}
\label{eq:firstmom}
    \sum_{\psi \shortmod{p^2}} L\pb{\tfrac{1}{2}, f \otimes \pb{\alpha\cdot \psi}}   =  A+B,
 \end{equation} with   
    
    \begin{equation}
    \label{eq:firstmomA}
        A =  \sum_{\psi \shortmod{p^2}} \sum_{n \ll (Xq)^{1+\varepsilon}} \frac{\lambda_f(n) \alpha(n) \psi(n)}{\sqrt{n}}V\left ( \frac{n}{Xq} \right ), 
    \end{equation}

    and
    \begin{equation}
     \label{eq:firstmomB}
        B = \sum_{\psi \shortmod{p^2}}  \varepsilon\left(f \otimes (\alpha\cdot \psi),\tfrac{1}{2}\right)\sum_{n \ll \frac{q}{X}\cdot q^{\varepsilon}}\frac{\ov{\lambda_f}(n)\ov{\alpha}(n)\ov{\psi}(n)} {\sqrt{n}}V\left ( \frac{nX}{q} \right ). 
    \end{equation}

Here, $V(y)$ is a smooth function satisfying $V(y) \ll_{f,A} (1 + \tfrac{y}{\sqrt{\mathfrak{q}_\infty}})^{-A}$, for any $A>0$.

 \subsubsection{Bounds on A} 
Interchanging the order of summation in $A$, we get, via the orthogonality of characters (and separating the diagonal term),

\begin{equation}
\label{eq:BoundA}
    A = p^2 \cdot V\pb{\frac{1}{Xq}} + p^2 \cdot \sum_{1 \leq l \ll \frac{Xq}{p^2}} \frac{\lambda_f(1+p^2l) \alpha(1+p^2l) \psi(1+p^2l)}{\sqrt{1+p^2l}}V\left ( \frac{1+p^2l}{Xq} \right ). 
\end{equation}

 Using

 \begin{equation*}
    \sum_{1 \leq l \ll \frac{Xq}{p^2}} \frac{\lambda_f(1+p^2l) \alpha(1+p^2l) \psi(1+p^2l)}{\sqrt{1+p^2l}}V\left ( \frac{1+p^2l}{Xq} \right ) \ll   \frac{1}{p}\sum_{1 \leq l \ll \frac{Xq}{p^2}} \frac{1}{\sqrt{l}} \ll \frac{(Xq)^{\frac12}}{p^2}, 
 \end{equation*}
 we can get
\begin{equation}
\label{eq:BoundAfinal}
    A = p^2 + O(X^{\frac12} q^{\frac12}) = p^2 + O(X^{\frac12} p^{\frac32}).
\end{equation}
  \subsubsection{Bounds on B}
    Here, we prove that $B = O (p^2 X^{-\tfrac{1}{2}})$. Along with \eqref{eq:BoundA}, this proves Lemma \ref{lem:first moment} once we set $X=p^{\tfrac{1}{2}}$.

  The root number of $L\pb{{f \otimes \pb{\alpha \cdot \psi}}}$ is given by 
  \begin{equation}
      \label{eq:rootnumber}
       \varepsilon\left(f \otimes (\alpha\cdot \psi),\tfrac{1}{2}\right) = i^k \frac{\tau(\psi\cdot \alpha)^2}{p^3}.
  \end{equation}

  Using this on \eqref{eq:firstmomB}, we get that
  \begin{equation}
    \label{eq:BoundB}
    B = i^k\sum_{\psi \shortmod{p^2}} \frac{\tau(\psi \cdot \alpha)^2}{p^3} \sideset{}{^*}\sum_{n \ll \frac{q}{X}} \frac{\ov{\lambda_f}(n)\ov{\alpha}(n)\ov{\psi}(n)}{\sqrt{n}} V\pb{\frac{nX}{q}}.
  \end{equation}

  Expanding out the Gauss sum, and using orthogonality for the $\psi$-sum, we get

    \begin{equation}
    \label{eq:BoundB2}
    B =i^k\frac{1}{p^3} \sideset{}{^*}\sum_{1\leq a,b \leq p^3}\sideset{}{^*}\sum_{n \ll \frac{q}{X}}  \frac{\ov{\lambda_f}(n)}{\sqrt{n}} V\pb{\frac{nX}{q}} e\pb{\frac{a+b}{q}} \alpha(ab\ov{n}) p^2 \cdot \delta(ab \equiv n\shortmod{p^2}). 
  \end{equation}

Here, $\ov{n}\cdot n \equiv 1 \Mod{p^3}$. 
\\The $\delta(ab \equiv n\Mod{p^2})$ condition implies $\exists \ l, \ 1 \leq l \leq p$ with $b = \ov{a}n(1+p^2l)$, where  $\ov{a}\cdot a \equiv 1 \Mod{p^3}$. 

Using this,

    \begin{equation}
    \label{eq:BoundB3}
    B =i^k\frac{1}{p} \sideset{}{^*}\sum_{n \ll \frac{q}{X}}  \frac{\ov{\lambda_f}(n)}{\sqrt{n}} V\pb{\frac{nX}{q}}\sideset{}{^*}\sum_{1\leq a \leq p^3} \sum_{1\leq l \leq p} e\pb{\frac{a+\ov{a}n(1+p^2l)}{p^3}} \alpha(1+p^2l).
  \end{equation}

   Notice that $\alpha(1+p^2(\cdot))$ is an additive character mod $p$. So, $\exists \ c_\alpha \Mod{p}$ such that $\alpha(1+p^2l) = e\pb{\dfrac{c_\alpha\cdot l}{p}}$.

  Using this, and orthogonality in the $l$-sum,

    \begin{equation}
    \label{eq:BoundB4}
    B =i^k\frac{1}{p} \sideset{}{^*}\sum_{n \ll \frac{q}{X}}  \frac{\ov{\lambda_f}(n)}{\sqrt{n}} V\pb{\frac{nX}{q}}\sideset{}{^*}\sum_{1\leq a \leq p^3}  e\pb{\frac{a+\ov{a}n}{p^3}} p \cdot \delta\pb{c_\alpha \equiv -\ov{a}n \shortmod{p}}.  
  \end{equation}

Again, the $\delta\pb{c_\alpha \equiv -\ov{a}n \Mod{p}}$ condition implies $\exists \ y_0,y_1, \ 1 \leq y_0,y_1 \leq p$ with \\$a = - n\ov{c_\alpha}(1+py_0+p^2y_1)$, where  $\ov{c_\alpha}\cdot c_\alpha \equiv 1 \Mod{p^3}$. So,

    \begin{equation}
    \label{eq:BoundB5}
    B =i^k \sideset{}{^*}\sum_{n \ll \frac{q}{X}}  \frac{\ov{\lambda_f}(n)}{\sqrt{n}} V\pb{\frac{nX}{q}} e\pb{-\frac{c_\alpha+n\ov{c_\alpha}}{p^3}}\sum_{1\leq y_0 \leq p} e\pb{\frac{-c_\alpha p y_0^2 + c_\alpha - n\ov{c_\alpha}}{p^2}}\sum_{1 \leq y_1 \leq p}  e\pb{\frac{y_1(c_\alpha-n\ov{c_\alpha})}{p}}. 
  \end{equation}

Orthogonality in the $y_1$ sum implies that we have $n\equiv c_\alpha^2 \Mod{p}$. So,

   \begin{equation}
    \label{eq:BoundB6}
    B = i^kp\cdot \sum_{\substack{n \ll \frac{q}{X}\\ n\equiv c_\alpha^2 \shortmod{p}}}  \frac{\ov{\lambda_f}(n)}{\sqrt{n}} V\pb{\frac{nX}{q}} e\pb{-\frac{c_\alpha+n\ov{c_\alpha}}{p^3}}\cdot R,
  \end{equation}

  with 

   \begin{equation}
    \label{eq:RBoundB}
    R = \sum_{1\leq y_0 \leq p} e\pb{\frac{-c_\alpha y_0^2 + k_\alpha y_0}{p}} ,
  \end{equation}

 where $k_\alpha = \frac{c_\alpha - n\ov{c_\alpha}}{p}$. Completing the square in \eqref{eq:RBoundB}, we have

    \begin{equation}
    \label{eq:RBoundB2}
    R = e\pb{\frac{k\alpha^2}{4c_\alpha}}\sum_{1 \leq y_0 \leq p} e\pb{\frac{c_\alpha (y_0+\frac{k_\alpha}{2c_\alpha})^2}{p}} = e\pb{\frac{k\alpha^2}{p 4c_\alpha}} \varepsilon_p \sqrt{p} \cdot \legendre{c_\alpha}{p}.
  \end{equation}
Here $\varepsilon_p$ is a constant of absolute value $1$.

Using \eqref{eq:RBoundB2} in \eqref{eq:BoundB6}, we have
\begin{equation}
    \label{eq:BoundB7}
    B = i^kp\cdot \sum_{\substack{n \ll \frac{q}{X}\\ n\equiv c_\alpha^2 \shortmod{p}}}  \frac{\ov{\lambda_f}(n)}{\sqrt{n}} V\pb{\frac{nX}{q}} e\pb{\frac{k\alpha^2}{4p c_\alpha}-\frac{c_\alpha+n\ov{c_\alpha}}{p^3}}\cdot \varepsilon_p \sqrt{p} \cdot \legendre{c_\alpha}{p}.
  \end{equation}

Using trivial bounds, we finally get 

\begin{equation}
    \label{eq:BoundBfinal}
     B \ll p^{\frac12 + \varepsilon} \pb{\frac{q}{X}}^{\frac12} = O(p^2X^{-\frac12}).
\end{equation}

Using \eqref{eq:BoundAfinal} and \eqref{eq:BoundBfinal} in \eqref{eq:firstmom}, we get

\begin{equation}
     \sum_{\psi \shortmod{p^2}} L\pb{\tfrac{1}{2}, f \otimes \pb{\alpha\cdot \psi}} =  p^2 + O(X^{\frac12} p^{\frac32}) + O(p^{2}X^{-\frac12}).
\end{equation}

Choosing $X=p^\frac12$ now completes the proof of Lemma \ref{lem:first moment}.

\section{Outline for Higher Prime Power Cases}
\label{sec:higherppower}

In this section, we briefly sketch the steps required to prove an analogous version of Theorem \ref{thm:MainThm} when $q=p^j, j\geq 4$. The statement for the theorem is modified as follows.

\begin{mytheo}
    Let $f$ be a level $1$ cusp form. Let $p$ be an odd prime and $q =p^j$, $j\geq 4$. Let $d = p^{\lceil\frac{2j}{3}\rceil}$. Let $\alpha$ be a primitive character modulo $q$. For any $\varepsilon >0$, 
    \begin{equation} 
        \label{eq:MainThm2}
        \sum_{\psi\Mod{d}} \abs{ L\pb{ f \otimes \pb{\alpha\cdot \psi}, \tfrac{1}{2}}}^2 \ll_{f,\varepsilon} dq^{\varepsilon}.
    \end{equation}
\end{mytheo}

The proof strategy remains the same and once again most of the work is involved in bounding terms similar to $\mathcal{M}_{\text{Maass}_0} $ in the proof of Theorem \ref{thm:MainThm}. 

We begin with the usual application of approximate functional equation and a dyadic partition of unity. After some simplification, we end up trying to establish the following bound for the associated shifted convolution problem,

\begin{equation}
   \label{eq:scpppower}
    S^j(N,\alpha) \coloneqq \sideset{}{^*}\sum_{\substack{m \equiv n \Mod{d} \\ m \neq n }} \phi(d) \lambda_f(m)\alpha(m\ov{n})\ov{\lambda_f}(n)w_N(m)w_N(n) \ll_{f,\varepsilon} Ndq^{\varepsilon}.  
\end{equation}
Here, the dyadic variable $N$ satisfies $N \ll q^{1+\varepsilon}.$ We note that the expression inside the sum exhibits no cancellation when $m \equiv n \Mod{q}$. However, the number of such terms is $O(Nq^\varepsilon)$, so this does not pose any problem.

We modify $S^j(N,\alpha)$ further by choosing $l=\frac{m-n}{d}$. Note that from our choices of $q,d,N$, we have $l \asymp \frac{N}{d} \ll \frac{q}{d} q^\varepsilon = p^{\lfloor\frac{j}{3}\rfloor+\varepsilon}$. It suffices to restrict to positive values of $l$, and once again the dominant contribution to $S^j(N,\alpha)$ is from the terms $\textrm{gcd}(l,p)=1$. In this case, a variant of Proposition \ref{prop:Postnikov} allows us to rewrite $\alpha((n+dl)\ov{n})$ as $e_{\frac{q}{d}}\pb{a_\alpha l \ov{n}}$ for some $a_\alpha \Mod{\frac{q}{d}}$. As we have observed earlier, this is a conductor-dropping phenomenon.

For other cases like $\textrm{gcd}(l,p)=p$, $\textrm{gcd}(l,p)=p^2, \cdots$ , we have two additional benefits - these are fewer in number, and the term $\alpha((n+dl)\ov{n})$ is an additive character of a lower modulous than $\frac{q}{d}$, making these cases easier to handle.

We now introduce a delta symbol to separate the oscillations of the $m$ and $n$-sums. Similarly to \eqref{eq:S1Nalpha}, it suffices now to show the bound,

\begin{multline}
 \label{eq:S1Nalphappower}
    S_1^j(N,\alpha)\coloneqq \\ \sideset{}{^*} \sum_{m,l,n}\lambda_f(n+d l) \ov{\lambda_f}(n)e_{\frac{q}{d}}(a_\alpha l \ov{n})w_N(m) w_N(n)
    \sum_{c \leq 2C} S(0,m-n-dl;c)\int_{-\infty}^{\infty} g_c(v)e\pb{(m-n-dl)v}dv  \ll_{f,\varepsilon} Nq^\varepsilon.
\end{multline}

Here, too, the dominant contribution arises from terms with $\textrm{gcd}(c,p)=1$. 
We first use the Voronoi summation formula on the $m$ and $n$-sums. We should note that for the $m$-sum, excluding the length of the sum, everything remains exactly the same as \eqref{eq:T1defn}. For the $n$-sum, the exponential factor $e_p\pb{a_\alpha l \ov{n}}$ in \eqref{eq:T2defn} changes to $e_{\frac{q}{d}}\pb{a_\alpha l \ov{n}}$ here. We can use a suitably modified variant of Proposition \ref{prop:modifiedVoronoi} here. Similar to the right side of \eqref{eq:Voronoifinal}, there will be multiple terms but the main contribution will be from the term with the hyper Kloosterman sum. Eventually, in order to prove the bound in \eqref{eq:S1Nalphappower}, it will be sufficient to show the same bound for 

\begin{multline}
\label{eq:E_0ppower}
    E_0^j = \sideset{}{^*}\sum_{l} \ \sideset{}{^*}\sum_{c \leq 2C} \ \sideset{}{^*}\sum_{a \shortmod{c}}e\pb{\frac{-adl}{c}}\int_{-\infty}^{\infty} g_c(v)e(-dlv) \pb{\frac{1}{c}\sum_{m } \lambda_f(m)e\pb{\frac{-\ov{a}m}{c}}\widetilde{w}_{c,v,N}(m)}\\
    \pb{\sideset{}{^*}\sum_{n} \frac{\ov{\lambda_f}(n)\widetilde{w}_{\tfrac{qc}{d},-v,N}(n)}{\tfrac{q^2c}{d^2}}e\pb{\frac{\ov{\tfrac{aq^2}{d^2}}n}{c}}\text{Kl}_3(n\ov{c}^2a_\alpha l,1,1;p)} dv.
\end{multline} 

We note that just as before these dual sums are shorter - $m \ll q^\varepsilon$ and $n \ll \frac{q^{2+\varepsilon}c^2}{d^2N}$. We want to get further cancellations, and for this we will use the Bruggeman Kuznetsov formula. Similar to the $q=p^3$ case, we first decompose the hyper Kloosterman sum as 

\begin{equation*}
    \text{Kl}_3(r,1,1;\frac{q}{d}) = \frac{1}{\phi(\frac{q}{d})}\sum_{\chi(\frac{q}{d})} \tau(\chi)^3 \chi(r),
\end{equation*}

where $\frac{q}{d} = p^{ \lfloor \frac{j}{3} \rfloor}$, and $\textrm{gcd} (r,\frac{q}{d}) =1$. 

We can then dyadically decompose the variables $c,l,m,n,v$ similar to what was done in Section \ref{subsec:dyadic}, and we are finally ready to spectrally decompose. The analogous expressions to \eqref{eq:E_0decomposition}, \eqref{eq:E_0dyadic} and \eqref{eq:KuznetsovKdyadic} in this case are

\begin{equation}
    \label{eq:E_0decompositionppower}
    E_0^j = \sum_{\substack{C_0,L_0,M_0,N_0,V\\ \text{dyadic}}} E_{C_0,L_0,M_0,N_0,V}^j + O(p^{-A}),
\end{equation}

where

\begin{equation}
\label{eq:E_0dyadicppower}
    E_{C_0,L_0,M_0,N_0,V}^j = \sideset{}{^*}\sum_{m,l,n,c} \frac{\lambda_f(m)\ov{\lambda_f}(n)}{\phi(\tfrac{q}{d})\frac{q^2}{d^2}} \sum_{\chi(\frac{q}{d})} \tau(\chi)^3 \chi(n a_\alpha l) \cdot \mathcal{K}_{C_0,L_0,M_0,N_0,V}^j,
\end{equation} 
with
\begin{equation}
    \label{eq:KuznetsovKdyadicppower}
    \mathcal{K}_{C_0,L_0,M_0,N_0,V}^j = \sideset{}{^*}\sum_{c}\frac{1}{c^2}S\pb{\ov{\frac{q}{d}}\pb{n-\pb{\frac{q}{d}}^2m},-\frac{d^2l}{q};c}\ov{\chi^2}({c})I_{N,V}(c,l,m,n).
\end{equation}

Here, 

\begin{equation}
    I_{N,V} (c,l,m,n) = \int_{V}^{2V} g_c(v)e(-dlv) \widetilde{w}_{c,v,N}(m) \widetilde{w}_{\tfrac{qc}{d},-v,N}(n) dv.
\end{equation}

The dyadic variables satisfy - $C_0 \leq 2\sqrt{N}, \ M_0 \ll q^\varepsilon, N_0 \ll \frac{q^{2+\varepsilon}c^2}{d^2N}, \ L_0 \ll \frac{N}{d}, \abs{V} \ll \frac{q^\varepsilon}{cC}$.

We apply Proposition \ref{prop:Kuznetsov} to $\mathcal{K}_{C_0,L_0,M_0,N_0,V}^j$. Once again, this gives rise to three sets of contributions: from Maass forms, from holomorphic forms, and from Eisenstein series. The way to bound each of them is similar and we discuss only the Maass form contributions here. The contribution from Maass form terms is 

\begin{multline}
    \label{eq:MMaass0ppower}
     \mathcal{M}_{\text{Maass}} =  \sum_{\substack{l \asymp L_0,m \asymp M_0,\cdots \\ \text{gcd}(ln,p)=1}} \frac{\lambda_f(m)\ov{\lambda_f}(n)}{\phi(\tfrac{q}{d})\frac{q^2}{d^2}}\sideset{}{^*} \sum_{\chi(\tfrac{q}{d})} \tau(\chi)^3 \chi(n a_\alpha l) \\ \cdot \sum_{t_j} \mathcal{L}^{\pm}\Phi(t_j) \sum_{r_1r_2=\frac{q}{d}} \sum_{\pi \in \mathcal{H}_{it}(r_2,\chi^2)} \frac{4\pi\epsilon_{\pi}}{V(\frac{q}{d})\mathscr{L}_\pi^*(1)}\sum_{\delta \mid r_1} \ov{\lambda}^{(\delta)}_\pi\pb{\left \vert{n-\pb{\tfrac{q}{d}}^2m}\right \vert}\ov{\lambda}^{(\delta)}_\pi(\abs{-\tfrac{d^2l}{q}}).
\end{multline}

Note that since $\frac{q}{d}$ is no longer necessarily $p$, we might have more choices for the triplet $(r_1,r_2,\delta)$ than the three choices outlined in Section \ref{subsec:Maassanalysis}. However, the dominant contribution is once again from when $r_2$ is maximum, so $(r_1,r_2,\delta) = (1, \frac{q}{d},1)$. This is because, similar to the reasons outlined in Section \ref{subsubsec:MMaass12},  the relative increase in the sizes of $\lambda_\pi^{(\delta)}(\cdot)$ is more than compensated for by the fact that the level of the Maass forms $\pi \in \mathcal{H}_{it}(r_2,\chi^2)$ drop. The contribution of $(r_1,r_2,\delta) = (1, \frac{q}{d},1)$ terms is (analogous to \eqref{eq:MM0})- 

\begin{multline}
    \label{eq:MM0ppower}
    \mathcal{M}_{\text{Maass}_0} =  \sum_{\substack{l \asymp L_0,m \asymp M_0,\cdots \\ \text{gcd}(ln,p)=1}} \frac{\lambda_f(m)\ov{\lambda_f}(n)}{\phi(\tfrac{q}{d})\frac{q^2}{d^2}}\sideset{}{^*} \sum_{\chi(\tfrac{q}{d})} \tau(\chi)^3 \chi(n a_\alpha l) \\ \cdot \sum_{t_j} \mathcal{L}^{\pm}\Phi(t_j)  \sum_{\pi \in \mathcal{H}_{it_j}(\frac{q}{d},\chi^2)} \frac{4\pi\epsilon_{\pi}}{V(\frac{q}{d})\mathscr{L}_\pi^*(1)}\sum_{\delta \mid r_1} \ov{\lambda}^{(1)}_\pi\pb{\left \vert{n-\pb{\tfrac{q}{d}}^2m}\right \vert}\ov{\lambda}^{(1)}_\pi(\abs{-\tfrac{d^2l}{q}}).
\end{multline}
We bound this exactly as in Section \ref{subsec:Maassanalysis}. The integral transform can be analysed leading to an oscillatory and non-oscillatory case. For instance, for the non-oscillatory case, similar to \eqref{eq:MMaassren}, we will have

\begin{multline}
\label{eq:MMaassrenppower}
    \mathcal{M}_{\text{Maass}_0}=  \frac{1}{\phi(\tfrac{q}{d})\frac{q^2}{d^2}} \sideset{}{^*}\sum_{\chi(\frac{q}{d})}  \cdot \sum_{t_j} \sum_{\pi \in \mathcal{H}_{it_j}(\frac{q}{d},\chi^2)}  \int_{\substack{\Re(s)=\sigma \\ \Re(u_j)=\sigma_j}}\frac{\tau(\chi)^3 \chi( a_\alpha) 4\pi\epsilon_{\pi} }{V(\tfrac{q}{d})\mathscr{L}_\pi^*(1)\pb{\tfrac{d^2}{q}}^{\frac{s}{2}}} \widetilde{H}(s,t,u_1,u_2,u_3) \\ \pb{\sum_{\substack{l  \\ \text{gcd}(l,p)=1}} \frac{\ov{\lambda}_\pi\pb{\tfrac{d^2}{q}l} \chi(l)}{l ^ {\frac{s}{2} + u_1} }} \pb{ \sum_{\substack{m,n  \\ \text{gcd}(n,p)=1}} \frac{\chi(n)\lambda_f(m)\ov{\lambda_f}(n)\ov{\lambda}_\pi\pb{\pb{\tfrac{q}{d}}^2m - n}}{\pb{\pb{\tfrac{q}{d}}^2m-n} ^ {\frac{s}{2}}m^{u_2} n^{u_3}}} \ ds \  du_1 \ du_2  \ du_3.
\end{multline}

The transform $ \widetilde{H}(s,t,u_1,u_2,u_3)$ satisfies the bound

\begin{equation}
    \label{eq:H2boundsppower}
     \widetilde{H}(s,t,u_1,u_2,u_3) \ll  (\sqrt{\tfrac{q}{d}}C_0)^{\sigma-1} (N\tfrac{q}{d})^{1+\varepsilon}  L_0^{\sigma_1} M_0^{\sigma_2} N_0^{\sigma_3} (1+ \abs{s})^{-A} \prod_{j=1}^{3} (1 + \abs{u_j})^{-A} .  
\end{equation}

The rest of the analysis is similar except for one small distinction. If $j \not \equiv 0 \Mod{3}$, then $\tfrac{d^2}{q} \neq \tfrac{q}{d}$. In fact, $\tfrac{d^2}{q} > \tfrac{q}{d}$. This means unlike the $q=p^3$ case (or more generally the case when $q=p^j$ for any $j\equiv 0 \Mod{3}$), we cannot use $\left \vert {\lambda_\pi\pb{\frac{d^2}{q}}} \right \vert \leq 1$ as $\tfrac{q^2}{d}$ does not divide the level, $\frac{q}{d}$. On the flip side, we have additional savings in the denominator because of the $\pb{\tfrac{d^2}{q}}^{\frac{s}{2}}$ factor. As a result of this we need to readjust our contour by choosing different values of $\sigma$ and $\sigma_i$, $i=1,2,3$. In fact, we only need to alter $\sigma_3$ to $\frac{\varepsilon}{2}$, and keep $\sigma, \sigma_1$ and $\sigma_2$ unchanged ($\sigma = \frac12, \sigma_1 = \frac14 + \frac{\varepsilon}{2}$, and $\sigma_2 = \frac{\varepsilon}{2}$). We can similarly work out the remaining cases.

\printbibliography

\end{document}